\documentclass [12pt]{amsart}
\usepackage[utf8]{inputenc}
\pdfoutput=1
\usepackage{amsmath,amssymb,amsthm,amsfonts}
\usepackage{mathtools}%

\usepackage[english]{babel}%
\usepackage{hyperref}%
\usepackage{bbm}%
\usepackage{bm}%
\usepackage{mathrsfs}%

\usepackage{tikz,graphicx,color}
\usepackage{tikz-cd}%
\usepackage{tikz-3dplot}%
\usetikzlibrary{calc}%
\usetikzlibrary{arrows}%
\usetikzlibrary{shapes}%
\usetikzlibrary{patterns}%
\usetikzlibrary{positioning}%
\usetikzlibrary{arrows.meta}
\usetikzlibrary{decorations.markings}
\usetikzlibrary{knots}

\usepackage{epstopdf}%

\usepackage[arrow]{xy}%
\usepackage{diagbox}%
\usepackage{subfig}%
\usepackage{arcs}%
\usepackage{xcolor}%
\usepackage{xspace}

\usepackage[margin=1.0in]{geometry}%

\usepackage{enumitem}
\usepackage{letltxmacro}
\usepackage{thmtools,etoolbox}

\def\myarabic#1{\normalfont(\roman{#1})}
\newlist{theoremlist}{enumerate}{1}
\setlist[theoremlist]{label=\myarabic{theoremlisti},ref={\myarabic{theoremlisti}},itemindent=0pt,labelindent=0pt,
	leftmargin=*,noitemsep}

\makeatletter
\renewcommand{\p@theoremlisti}{\perh@ps{\thetheorem}}
\protected\def\perh@ps#1#2{\textup{#1#2}}
\newcommand{\itemrefperh@ps}[2]{\textup{#2}}
\newcommand{\itemref}[1]{\begingroup\let\perh@ps\itemrefperh@ps\ref{#1}\endgroup}
\makeatother

\usepackage{nameref,hyperref}
\usepackage[capitalize]{cleveref}

\newtheorem{theorem}{Theorem}[section]

\newtheorem{lemma}[theorem]{Lemma}
\newtheorem{proposition}[theorem]{Proposition}
\newtheorem{corollary}[theorem]{Corollary}
\theoremstyle{definition}
\newtheorem{remark}[theorem]{Remark}
\theoremstyle{definition}
\newtheorem{definition}[theorem]{Definition}

\theoremstyle{definition}
\newtheorem{problem}[theorem]{Problem}
\theoremstyle{definition}
\newtheorem{example}[theorem]{Example}

\crefname{figure}{Figure}{Figures}

\def\figref#1(#2){Figure~\hyperref[#1]{\ref*{#1}(#2)}}

\addtotheorempostheadhook[theorem]{\crefalias{theoremlisti}{theorem}}
\addtotheorempostheadhook[lemma]{\crefalias{theoremlisti}{lemma}}
\addtotheorempostheadhook[proposition]{\crefalias{theoremlisti}{proposition}}
\addtotheorempostheadhook[corollary]{\crefalias{theoremlisti}{corollary}}

\def\Ocal{\mathcal{O}}\def\Zcal{\mathcal{Z}}

\def\abf{\mathbf{a}}

\def\R{{\mathbb{R}}}

\def\Z{{\mathbb{Z}}}
\def\Q{{\mathbb{Q}}}

\def\T{{\mathbb{T}}}
\def\D{{\mathbb{D}}}
\def\A{{\mathbb{A}}}
\def\<{{\langle}}
\def\>{{\rangle}}
\def\eps{{\epsilon}}
\def\la{{\lambda}}

\def\id{\operatorname{id}}

\def\det{{ \operatorname{det}}}

\def\Im{{ \operatorname{Im}}}

\def\wt{\operatorname{wt}}

\def\mod{{\operatorname{mod}}}

\def\xing{{\operatorname{xing}}}
\def\supp{\operatorname{supp}}

\def\xrasim{\xrightarrow{\sim}}

\def\SL{\operatorname{SL}}

\def\Povtp_#1{\Pi_{#1}^{>0}}
\def\Povtnn_#1{\Pi_{#1}^{\geq0}}

\definecolor{calpolypomonagreen}{rgb}{0, 0.6, 0.2}

\tikzset{qvert/.style={draw,black,circle,fill=gray,minimum size=5pt,inner sep=0pt}  } 
\tikzset{bvert/.style={draw,circle,fill=black,minimum size=5pt,inner sep=0pt}  }  
\tikzset{gbvert/.style={draw, gray, circle,fill=gray,minimum size=5pt,inner sep=0pt}  } 
\tikzset{gvert/.style={draw,gray,circle,fill=white,minimum size=5pt,inner sep=0pt}  } 
\tikzset{wvert/.style={draw,circle,fill=white,minimum size=5pt,inner sep=0pt}  } 
\tikzset{fvert/.style={text=MidnightBlue}  } 
\tikzset{sqvert/.style={draw,black,rectangle,fill=black,minimum size=5pt,inner sep=0pt}  } 
\tikzset{lvert/.style={draw,circle,fill=black,minimum size=4pt,inner sep=0pt}  }  
\usetikzlibrary{arrows}
\tikzcdset{arrow style=tikz, diagrams={>={Stealth[round,length=4pt,width=4.95pt,inset=2.75pt]}}}

\numberwithin{equation}{section}

\makeatletter
\@namedef{subjclassname@2020}{\textup{2020} Mathematics Subject Classification}
\makeatother
\def\ilen|#1|{|#1|} %
\def\bfla{\bm{\la}} %
\def\e{e} %
\def\cc{\alpha} %
\def\bfcc{\bm{\alpha}} %
\def\rectangle{\mathbb{D}} %
\def\u{u}
\def\d{d}
\def\l{l}
\def\r{r} %
\def\sa{S} %
\def\sb{T} %
\def\am{\pi} %
\def\Nwdec{\dot{N}} %
\def\Ndec{\ddot{N}} %

\def\Sn{S_n}
\def\Saff{\widetilde{S}}
\def\Saffn{\Saff_n}
\def\Saffx_#1{\Saff_{#1}}
\def\Saffon{\Saff^{(0)}_n}
\def\Saffox_#1{\Saff^{(0)}_{#1}}
\def\Saffkn{\Saff^{(k)}_n}
\def\Saffxx_#1^#2{\Saff^{(#2)}_{#1}}

\def\Sext{\widehat{S}}
\def\Sextn{\Sext_n}

\def\Sextxx_#1^#2{\Sext^{#2}_{#1}}

\def\kop{\operatorname{k}}
\def\kopf{\operatorname{k}_f}
\def\kopx_#1{\operatorname{k}_{#1}}
\def\nop{\operatorname{n}}
\def\nopf{\operatorname{n}_f}
\def\nopx_#1{\operatorname{n}_{#1}}

\def\At{\widetilde{A}}
\def\La{\Lambda}
\def\Lan{\La}
\def\O{\Ocal}
\def\Omin{\O_{\min}}
\def\Inv{\operatorname{Inv}}
\def\rasi{\xrightarrow{s_i}}
\def\Cycle{C}

\def\restrictC{r_{\Cycle}}
\def\slopes_#1{\bm{\nu}_{#1}}
\def\SLF{\slopes_f}

\def\slp{\nu}
\def\slpfC{\slp_f(\Cycle)}
\def\slpfx(#1){\nu_f(#1)}
\def\slupp{C_{f,\slp}}
\def\sluppp{C_{f,\slp'}}
\def\Cycles{\bm{C}}
\def\CLF{\Cycles_f}
\def\CLFS{\Cycles_f(\slp)}

\def\fb{\bar f}

\def\approx{\stackrel{\scalebox{0.6}{\ \normalfont{c}}}{\sim}}
\def\edge{e}
\def\ef{\edge_f}
\def\ilen|#1|{|#1|} %
\def\bflaf{\bfla^f}
\def\lafs{\la^{f,\slp}}
\def\VC{E}
\def\VCD{\dot{E}}
\def\VCF{E_f}
\def\VCFD{\dot{E}_f}
\def\VCFPD{\dot{E}_{f'}}
\def\VCDDe{\ddot{E}}
\def\VCFDD{\ddot{E}_f}
\def\VCFPDD{\ddot{E}_{f'}}
\def\VCGDD{\ddot{E}_g}

\def\Zon{\Zcal}

\def\Area{\operatorname{Area}}
\def\xing{\operatorname{xing}}
\def\clicks{\operatorname{d}}
\def\rot{\operatorname{rot}}
\def\minv{\mu}
\def\bfccf{\bfcc_f}
\def\bfccpera{\bfcc_{\pera}}
\def\bfccperb{\bfcc_{\perb}}
\def\ccfs{\cc^{f,\slp}}
\def\ccfsp{\cc^{f,\slp'}}
\def\ccfps{\cc^{f',\slp}}
\def\fhat{\hat f}
\def\suppxx_#1(#2){\supp_{#1}(#2)}
\def\suppkx(#1){\supp_k(#1)}

\def\fs{f|_\slp}
\def\Ln{\mathcal{L}_n}

\def\Lng{\Ln^\circ}
\def\DiagOp{\operatorname{D}}
\def\LPConf(#1){\DiagOp(#1)} %
\def\Diag(#1){\DiagOp_f(#1)} %
\def\Diagx_#1(#2){\DiagOp_{#1}(#2)} %
\def\ellf{\ell_f}
\def\Straight{\operatorname{Str}}
\def\Stre{\Straight_{\eps}}
\def\Streg{\Straight_{\eps}^\circ}

\def\slpfi{\slp_f(i)}
\def\slpfj{\slp_f(j)}

\def\VCDD_#1{\ddot{E}_{#1}}
\def\dist{\operatorname{dist}}
\def\abf{\bm{a}}
\def\xx(#1){x_{#1}}
\def\xt{x(t)}
\def\xtz{x(t_0)}
\def\xtt_#1{x(#1)}
\def\xtx(#1){x_{#1}(t)}
\def\xtzx(#1){x_{#1}(t_0)}
\def\xtxt_#1(#2){x_{#2}(#1)}

\def\exop{\operatorname{exc}}
\def\excess#1{\exop(#1)}
\def\fkn{f_{k,n}}
\def\Path{P^{(f)}}
\def\dop{\operatorname{d}}
\def\dopf{\dop_f}
\def\Ibm_#1{\bm{I}(#1)}
\def\Icc{\Ibm_{\cc}}

\def\AGSET{B}
\def\seedpt{\bar p}
\def\Seedpt{\bar P}
\def\Seedptl{P}
\def\SP{\bm{P}}
\def\DP{\DiagOp}
\def\seedptp{p'}

\def\barseedptp{\bar{p}'}
\def\seedpts{\bm{p}}
\def\rotsig{\sigma}

\def\daffper{w}
\def\pera{f}
\def\perb{{\overline{f}}}

\def\SD{\bm{S}(D)} %

\def\lw{1.1pt}

\def\Msq{(M1)\xspace} %
\def\Mres{(M2)\xspace}
\def\MMs{\Msq--\Mres}

\def\RRs{(R1)--(R3)\xspace}

\def\ilen|#1|{|#1|_{\Z}} %

\def\e{e} %
\def\SG{\bm{S}(\Gamma)}
\def\SGp{\bm{S}(\Gamma')}
\def\SE{\bm{S}^e(\Gamma)}
\def\lab{\gamma}
\def\labe{\gamma^\e}
\def\labem{\bar\gamma^\e}
\def\er{\textcolor{red}{\e}}
\def\eb{\textcolor{blue}{\e'}}
\def\eer{{\textcolor{red}{\e_1}}}
\def\eeb{{\textcolor{blue}{\e_2}}}
\def\eeg{{\textcolor{green!80!black}{\e_3}}}
\def\laber{\gamma^{\er}}
\def\labeb{\gamma^{\eb}}
\def\Reg{F}
\def\Regs{\bm{F}}
\def\Rege{\Regs^\e}

\def\dashcolor{black!50} %
\def\CoxGens{\Pi}
\def\Chrs{\Sigma}
\def\Min{\operatorname{Min}}

\def\Cred{\textcolor{red}{\Cycle}}
\def\Cblue{\textcolor{blue}{\Cycle'}}

\def\wind#1{\operatorname{wind}(#1)}
\def\curve{\rho}
\def\curves{{\bm\rho}}
\def\sat{{\tilde\sa}}
\def\Gammat{{\tilde\Gamma}}

\def\dSaffHuge{(\Saffon \times \Saffon)\rtimes \langle \Lan \rangle/\langle \Lan^n \rangle}
\def\dSaff{\ddot{S}_n}
\crefformat{equation}{(#2#1#3)}
\def\trvl{zero-homology\xspace}

\def\amD{\am_D}
\def\Strip{\mathbb{S}}
\def\dist{\operatorname{dist}}
\let\ovl\overline
\def\perab{\phi}

\def\itwo{$-$ii}
\def\ithree{$-$iii}
\def\rev{\operatorname{rev}}
\begin{document}
	\numberwithin{equation}{section}
	
	\title{Move-reduced graphs on a torus}
	\author{Pavel Galashin and Terrence George}
	\address{Department of Mathematics, University of California, Los Angeles, CA 90095, USA}
	\email{{\href{mailto:galashin@math.ucla.edu}{galashin@math.ucla.edu}}}
	\address{Department of Mathematics, University of Michigan, Ann Arbor, MI 48103, USA}
	\email{{\href{mailto:georgete@umich.edu}{georgete@umich.edu}}}
	\thanks{P.G.\ was supported by an Alfred P. Sloan Research Fellowship and by the National Science Foundation under Grants No.~DMS-1954121 and No.~DMS-2046915.}
	\date{\today}
	
	\subjclass[2020]{
		Primary:
		05C10. %
		Secondary:
		13F60. %
	}
	
	\keywords{Bipartite graphs on a torus, square moves, affine permutations, conjugation.}

	\begin{abstract}
		We determine which bipartite graphs embedded in a torus are move-reduced. In addition, 
		we classify equivalence classes of such move-reduced graphs under square/spider moves.
		This extends the class of minimal graphs on a torus studied by Goncharov--Kenyon, and gives a toric analog of Postnikov's results on a disk.
	\end{abstract}
	
	\maketitle

	\maketitle

	\section*{Introduction}\label{sec:intro}
	Let $\T=\R^2/\Z^2$ be a torus, and let $\Gamma$ be a bipartite graph embedded in $\T$. We say that two such graphs $\Gamma,\Gamma'$ are \emph{move-equivalent} if they are related by the moves \MMs shown in \cref{fig:intro:moves}. We say that $\Gamma$ is \emph{move-reduced} if there does not exist a graph $\Gamma'$ move-equivalent to $\Gamma$ to which we can apply one of the \emph{reduction moves} \RRs shown in \cref{fig:localmoveplabic}. The goal of this paper is to describe which graphs $\Gamma$ are move-reduced, and which pairs of move-reduced graphs are move-equivalent. 
	A similar problem has been considered in~\cite{GK13} for the class of \emph{minimal graphs}. Each minimal graph is move-reduced, however, the converse is not true; see \cref{fig:move_red_vs_minimal}. 
	
	We briefly summarize our main results; see \cref{sec:main} for more details. It was shown in~\cite{GK13} that move-equivalence classes of minimal graphs are classified by their Newton polygons $N$. The sides of $N$ are obtained by taking the homology classes of strands in $\Gamma$. Here, a \emph{strand} is a path making a sharp right (resp., left) turn at each black (resp., white) vertex. A strand of a move-reduced (as opposed to minimal) graph $\Gamma$ may intersect itself, and this induces a \emph{weak decoration} $\bfla=(\la^\e)_{\e\in E(N)}$ of $N$, labeling each side $\e=(i,j)$ of $N$ by a partition $\la^\e$ of $\gcd(i,j)$. Our first main result (\cref{thm:intro:move_red}) gives a characterization of move-reduced graphs in terms of weakly decorated Newton polygons that parallels the results of~\cite{GK13,Postnikov}.
	
	Our second main result concerns move-equivalence classes of move-reduced graphs. The solution to this problem turns out to be more subtle than its counterparts in~\cite{GK13,Postnikov}. First, we show that in a move-reduced graph, different strands corresponding to the same side of $N$ never cross each other. This induces a \emph{strong decoration} $\bfcc=(\cc^\e)_{\e\in E(N)}$ of $N$, labeling each side $\e=(i,j)$ of $N$ with a cyclic composition $\cc^\e$ of $\gcd(i,j)$. We associate a \emph{rotation number} $\clicks(\bfcc)$ to $\bfcc$, and our second main result (\cref{thm:intro:move_eq}) is that the set of all move-reduced graphs with strongly decorated Newton polygon $(N,\bfcc)$ is a union of $\clicks(\bfcc)$ move-equivalence classes. The classes are distinguished by the value of an explicit \emph{modular invariant} $\minv(\Gamma)\in\Z/\clicks(\bfcc)\Z$ associated to each move-reduced graph $\Gamma$.

	Our motivation to study move-reduced graphs arises from the dimer model on $\Gamma$ and the associated \emph{spectral transform} of~\cite{KOS,KeOk}. Each weighted bipartite graph $(\Gamma,\wt)$ with positive real edge weights embedded in $\T$ determines a simple Harnack curve with a distinguished line bundle. It is thus natural to study which limiting objects appear when one sends some edge weights to zero. This corresponds to deleting edges from $\Gamma$ and then applying reduction moves. Note in particular that the move-reduced graph $\Gamma_2$ in \figref{fig:move_red_vs_minimal}(b) is obtained from the minimal graph $\Gamma_1$ in \figref{fig:move_red_vs_minimal}(a) by removing a single edge, which demonstrates that the class of move-reduced graphs is more naturally suited for this problem. %
	
	For the case of planar bipartite graphs in a disk, the resulting space of limiting objects is the \emph{totally nonnegative Grassmannian}~\cite{Postnikov}, where the role of the spectral transform is played by Postnikov's boundary measurement map. In particular, Postnikov characterized move-reduced graphs on a disk and showed that their move-equivalence classes are classified by \emph{positroids}. 
	The present manuscript is the first in a series of papers aimed at studying the toric analog of the totally nonnegative Grassmannian and its positroid stratification.

	\begin{figure}
		\begin{tabular}{ccc}
			\includegraphics[width=0.4\textwidth]{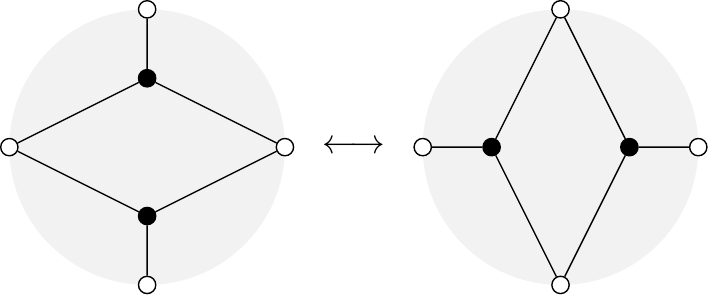}
			& \qquad &
			\includegraphics[width=0.4\textwidth]{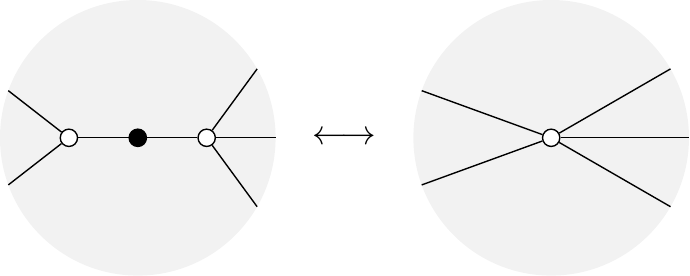}
			\\
			(M1) The spider move. & & (M2) The contraction-uncontraction move. 
		\end{tabular}
		\caption{\label{fig:intro:moves} Equivalence moves for bipartite graphs in $\T$. One can also apply these moves with the roles of white and black swapped. For \Msq, the vertices of the square are assumed to have degree at least three. For \Mres, the two white vertices are assumed to be distinct and have degree at least two. The shaded area denotes a small open disk inside $\T$.}
	\end{figure}
	
	\begin{figure}
		\def\twd{0.23\textwidth}
		\begin{tabular}{ccccc}
			\includegraphics[width=\twd]{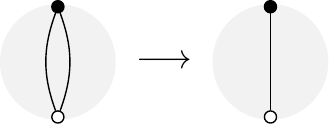}
			& \qquad &
			\includegraphics[width=\twd]{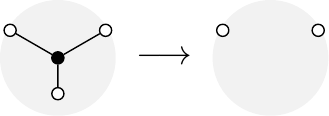}
			& \qquad &
			
			\includegraphics[width=\twd]{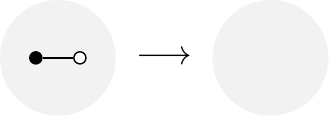}
			\\
			(R1) Parallel edge reduction. & & (R2) Leaf reduction. & & (R3) Dipole reduction.
		\end{tabular}
		\caption{\label{fig:localmoveplabic} Reduction moves for bipartite graphs. (R1) removes one of two parallel edges, (R2) removes a leaf together with its single neighbor, and (R3) removes an isolated edge. 
		}
	\end{figure}

	\section{Main results}\label{sec:main}

	In \cref{sec:intro:newton}, we introduce the notions of \emph{weakly} and \emph{strongly decorated polygons}. In \cref{sec:intro:move_reduced}, we will associate a weakly decorated polygon with any bipartite graph embedded in the torus, and we will use it to characterize move-reduced graphs. In \cref{sec:intro:move_eq}, we will associate a strongly decorated polygon to any move-reduced graph $\Gamma$, and will use it to characterize which graphs are move-equivalent to $\Gamma$.

	\subsection{Decorated polygons}\label{sec:intro:newton}
	A convex polygon $N$ in the plane $\R^2$ is called \textit{integral} if its vertices are contained in $\Z^2 \subset \R^2$. We denote the set of edges of $N$ by $E(N)$, and orient them counterclockwise around the boundary of $N$ so that each edge is a vector in $\Z^2$. %
	For an edge $\e=(a,b)$ of $N$, let $\ilen|\e|:=\gcd(a,b)$ be its \emph{integer length}. For vectors $\e,\e'\in\Z^2$, let $\det(\e,\e')$ be the determinant of the $2\times2$ matrix with columns $\e,\e'$. 

	A \emph{partition} of $n$ with $k$ parts is a tuple $\la=(\la_1\geq\la_2\geq\dots\geq\la_k>0)$ such that $|\la|:=\la_1+\la_2+\dots+\la_k=n$. A \textit{composition} of $n$ with $k$ parts is a tuple $\alpha=(\alpha_1,\alpha_2,\dots, \alpha_k) \in \Z^k_{> 0}$ such that $|\alpha|:=\alpha_1+\dots+\alpha_k=n$. A \textit{cyclic composition} of $n$ with $k$ parts is an equivalence class of compositions of $n$ with $k$ parts under cyclic shifts $(\alpha_1,\alpha_2,\dots,\alpha_k)\sim(\alpha_2,\dots,\alpha_k,\alpha_1)$. Thus, forgetting the order of the parts of a (cyclic) composition yields a partition.
	
	\begin{definition}\label{dfn:decor}\ 
		\begin{itemize}
			\item A \emph{weakly decorated polygon} is a pair $\Nwdec=(N,\bfla)$, where $N$ is a convex integral polygon, and $\bfla=(\la^\e)_{\e \in E(N)}$, where $\la^\e$ is a partition of $\ilen|\e|$.
			\item A \emph{strongly decorated polygon} is a pair $\Ndec=(N,\bfcc)$, where $N$ is a convex integral polygon, and $\bfcc=(\cc^\e)_{\e \in E(N)}$, where $\cc^\e$ is a cyclic composition of $\ilen|\e|$.
		\end{itemize}
	\end{definition}
	
	\begin{figure}
		\begin{tabular}{ccccccccc}
			\includegraphics[width=0.2\textwidth]{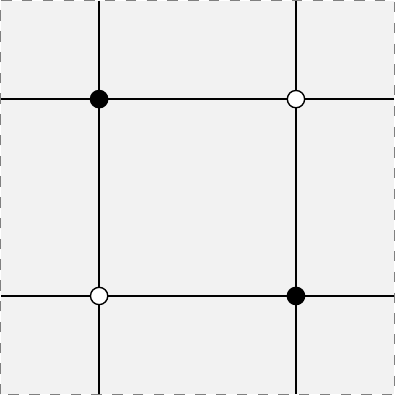}
			&
			\includegraphics[width=0.2\textwidth]{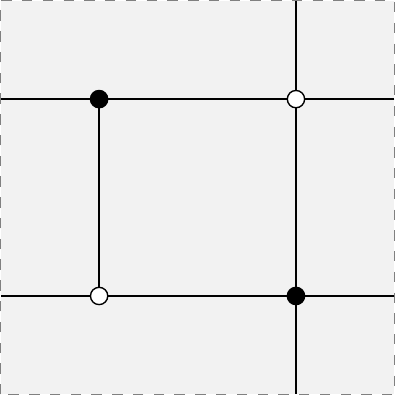}
			&
			\includegraphics[width=0.2\textwidth]{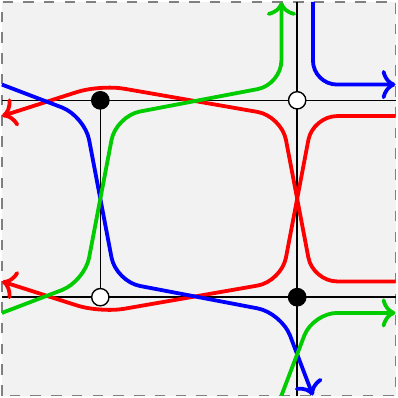}
			&
			\includegraphics[width=0.2\textwidth]{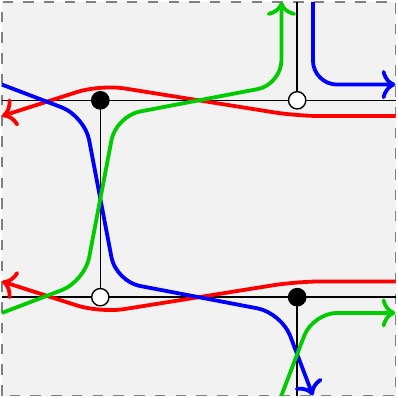}
			\\
			(a) Graph $\Gamma_1$. & (b) Graph $\Gamma_2$. & (c) $\Gamma_2$ with strands. & (d) $\Gamma_3$ with strands.
			\\
			&
			&
			\includegraphics[width=0.2\textwidth]{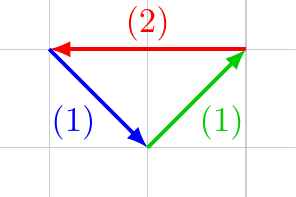}
			&
			\includegraphics[width=0.2\textwidth]{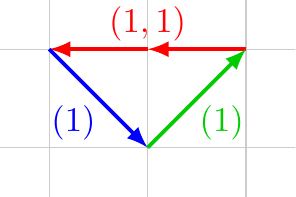}\\
			& & (e) $\Nwdec(\Gamma_2)$. & (f) $\Nwdec(\Gamma_3)$.
		\end{tabular}
		\caption{\label{fig:move_red_vs_minimal} The graphs $\Gamma_1$ and $\Gamma_3$ are minimal in the sense of~\cite{GK13} and therefore are move-reduced. The graph $\Gamma_2$ is not minimal but is move-reduced. See \cref{sec:intro:move_reduced} for a definition of strands and $\Nwdec(\Gamma)$.}
	\end{figure}

	\subsection{Move-reduced graphs}\label{sec:intro:move_reduced}
	Recall that a \emph{strand} or a \emph{zig-zag path} $\sa$ is a walk in $\Gamma$ that turns maximally right at the black vertices and maximally left at the white vertices of $\Gamma$. The set of strands of $\Gamma$ is denoted by $\SG$. Since $\Gamma$ is finite, a strand $\sa$ is a (not necessarily simple) closed walk, and we let $[\sa]\in\Z^2=H_1(\T,\Z)$ denote its homology. %
	Since each edge of $\Gamma$ is contained in two strands that traverse it in opposite directions, the sum $\sum_{\sa\in\SG} [\sa]$ is zero, so we can associate to $\Gamma$ a weakly decorated polygon $\Nwdec(\Gamma)=(N,\bfla)$ as follows. We let $N$ be the convex integral polygon $N$ (possibly degenerate, i.e., having $0$ area), unique up to translation, whose counterclockwise-oriented boundary consists of the vectors $([\sa])_{\sa \in\SG}$ in some order. We say that two strands $\sa,\sa'\in\SG$ are \emph{parallel} if $[\sa],[\sa']\neq0$ and $[\sa]\in\R_{>0}[\sa']$. For each edge $\e \in E(N)$, we let 
	\begin{equation*}%
		\SE:=\{\sa\in\SG \mid [\sa]\in\R_{>0}\e\}
	\end{equation*}
	denote the corresponding set of parallel strands. Thus, we have $\e=\sum_{\sa\in\SE} [\sa]$,  
	and we let $\la^\e:=(\ilen|[\sa]|)_{\sa\in\SE}$ be the corresponding partition of $\ilen|\e|$. The polygon $N$ is called the \emph{Newton polygon} of $\Gamma$, and we call $\Nwdec(\Gamma)$ the \emph{weakly decorated Newton polygon} of $\Gamma$. The weakly decorated Newton polygon is invariant under \MMs but not under \RRs.
	
	In \cref{prop:intro:exists}, we will see that for any weakly decorated polygon $\Nwdec$, there exists a move-reduced graph $\Gamma$ satisfying $\Nwdec(\Gamma)=\Nwdec$. On the other hand, it is clear that any graph $\Gamma$ can be transformed into a move-reduced graph using the moves \MMs and \RRs. %

	\begin{definition} \label{dfn:exc}
		For a partition $\la=(\la_1\geq\la_2\geq\dots\geq\la_k>0)$ of $n$ with $k$ parts, the \emph{excess} of $\la$ is defined by $\excess{\la}:=n-k=\sum_{i=1}^k(\la_i-1)$. If $\bfla=(\la^\e)_{\e\in E}$ is a collection of partitions, we denote  $\excess{\bfla}:=\sum_{\e\in E} \excess{\la^e}$. 
	\end{definition}

	A \emph{face} of $\Gamma$ is a connected component of $\T\setminus \Gamma$. Thus, a face of $\Gamma$ is contractible if and only if it is homeomorphic to an open disk. 
	We are ready to state our first main result.
	\begin{theorem}\label{thm:intro:move_red}
		Let $\Gamma$ be a bipartite graph embedded in $\T$ with weakly decorated Newton polygon $\Nwdec(\Gamma)=(N,\bfla)$. Assume that $\Gamma$ has a perfect matching. The following conditions are equivalent.
		\begin{enumerate}
			\item\label{item:intro:move_red} $\Gamma$ is move-reduced.
			\item\label{item:intro:area} $\Gamma$ has $2\Area(N)+\excess{\bfla}$ contractible faces, no contractible connected components, and no leaf vertices.
		\end{enumerate}
		Moreover, if $\Gamma$ is move-reduced and $\sa,\sa'\in\SG$ are two distinct parallel strands, then $\sa,\sa'$ do not share any vertices or edges of $\Gamma$.
	\end{theorem}

	\begin{figure}
		\def\twd{0.18\textwidth}
		\scalebox{0.9}{
			\begin{tabular}{ccc%
				}
				\includegraphics[width=\twd]{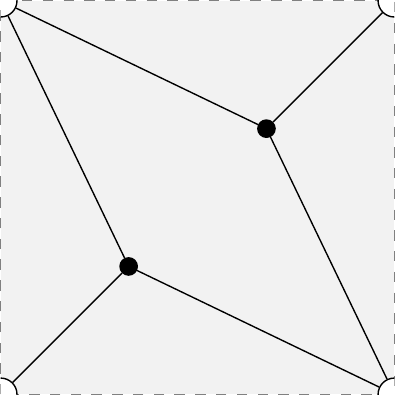} & 
				\qquad
				&
				\includegraphics[width=\twd]{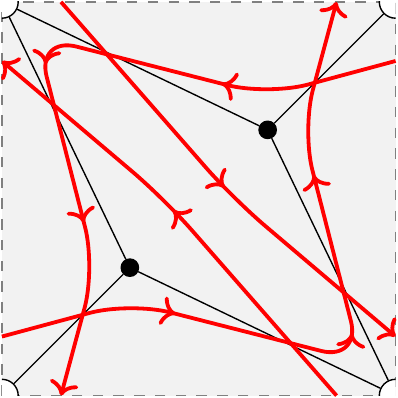}

			\end{tabular}
		}
		\caption{\label{fig:parallelbigons}A move-reduced graph with no perfect matchings and whose Newton polygon is a single point.} 
	\end{figure}

	\begin{remark}
		The assumption that $\Gamma$ has a perfect matching is essential; for example, \cref{thm:intro:move_red} fails for the graph $\Gamma$ in Figure~\ref{fig:parallelbigons}. This graph is move-reduced and does not have any perfect matchings. Thus, $\Gamma$ satisfies condition~\eqref{item:intro:move_red} but does not satisfy condition~\eqref{item:intro:area} of \cref{thm:intro:move_red}. Alternatively, if $\Gamma$ has no isolated vertices, the assumption that $\Gamma$ has a perfect matching can be replaced with either one of the following assumptions:
		\begin{itemize}
			\item the Newton polygon of $\Gamma$ is not a single point, or
			\item the number of black and white vertices in $\Gamma$ is the same;
		\end{itemize}
		see part~\ref{intro:Gamma_monogon} of \cref{thm:intro:vertical} below.
	\end{remark}
	
	\begin{remark} \label{rmk:intro:area}
		The condition that $\Gamma$ has $2\Area(N)+\excess{\bfla}$ contractible faces in~\eqref{item:intro:area} is equivalent to a statement that $\Gamma$ has the minimal possible number of contractible faces among all graphs with weakly decorated Newton polygon $\Nwdec(\Gamma)$.
	\end{remark}

	\begin{example}
		For the graphs $\Gamma_1,\Gamma_2,\Gamma_3$ shown in \cref{fig:move_red_vs_minimal}, the weakly decorated Newton polygons $\Nwdec(\Gamma_2),\Nwdec(\Gamma_3)$ are computed in \figref{fig:move_red_vs_minimal}(e--f). In particular, letting $\Nwdec(\Gamma_2)=(N,\bfla)$ and $\Nwdec(\Gamma_3)=(N,\bfla')$, we see that $\Area(N)=1$, $\excess{\bfla}=1$, and $\excess{\bfla'}=0$. This is consistent with \cref{thm:intro:move_red} since $\Gamma_2$ has $3$ faces, while $\Gamma_3$ has $2$ faces, all of which are contractible.
	\end{example}

	\subsection{Move-equivalence classes of move-reduced graphs}\label{sec:intro:move_eq}
	In this section, each graph is assumed to be bipartite and to have a perfect matching. Let $\Gamma$ be a move-reduced graph with Newton polygon $N$. By \cref{thm:intro:move_red}, for $\e\in E(N)$, any two strands $\sa\neq\sa'$ in $\SE$ do not share vertices or edges. Thus, we have a natural cyclic ordering on $\SE$ given by the direction of the normal vector to $\e$ that points into the interior of $N$. Let $\cc^\e=(\ilen|[\sa]|)_{\sa\in\SE}$ be the corresponding cyclic composition of $\ilen|\e|$. We set $\bfcc=(\cc^\e)_{\e \in E(N)}$, and we refer to $\Ndec(\Gamma):=(N,\bfcc)$ as the \emph{strongly decorated Newton polygon} of $\Gamma$. The following result is shown in \cref{sec:proof:exists}.
	
	\begin{proposition}\label{prop:intro:exists}
		For any strongly decorated polygon $\Ndec$, there exists a move-reduced graph $\Gamma$ that admits perfect matching and satisfies $\Ndec(\Gamma)=\Ndec$.
	\end{proposition}
	
	The moves \MMs never change the homology of the strands and preserve the class of move-reduced graphs. Thus, if two move-reduced graphs $\Gamma,\Gamma'$ are move-equivalent then we have $\Ndec(\Gamma)=\Ndec(\Gamma')$. One is tempted to conjecture that the converse is also true, but that is not the case; for instance, the two graphs in \cref{fig:twsq} have the same strongly decorated Newton polygons, but they are not move-equivalent, since one graph is connected and the other one is not. See \cref{fig:two_big} for a more subtle example. 
	To remedy this issue, we make the following definition.
	
	\begin{figure}
		\def\twd{0.15\textwidth}
		\scalebox{0.975}{
			\begin{tabular}{ccccc}
				\includegraphics[width=\twd]{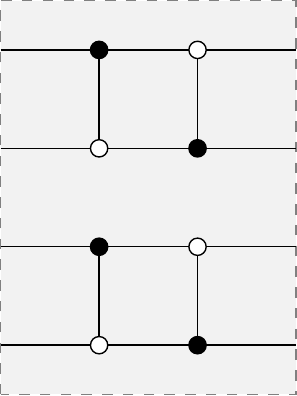}
				&
				\includegraphics[width=\twd]{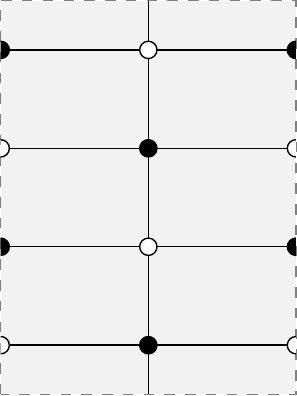}
				&
				\includegraphics[width=\twd]{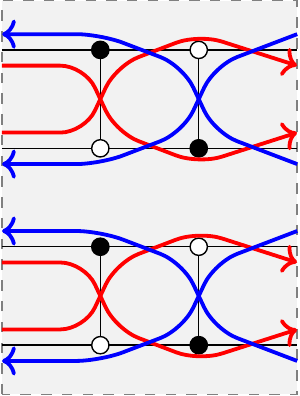}
				&
				\includegraphics[width=\twd]{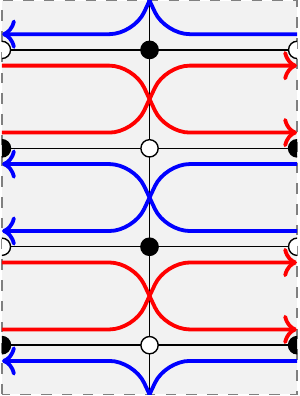}
				&
				\begin{tikzpicture}[baseline=(Z.base)]
					\coordinate(Z) at (0,-2);
					\node(A) at (0,0) {\includegraphics[width=0.22\textwidth]{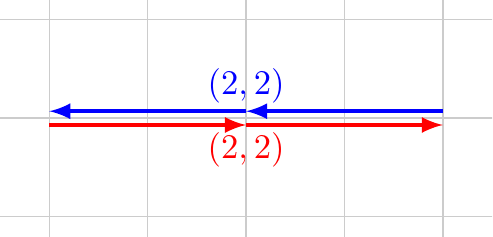}};
				\end{tikzpicture}
				\\
				(a) Graph $\Gamma_1$. & (b) Graph $\Gamma_2$. & (c) Strands in $\Gamma_1$. & (d) Strands in $\Gamma_2$.  & (e) $\Ndec(\Gamma_1)=\Ndec(\Gamma_2)$.
			\end{tabular}
		}
		\caption{\label{fig:twsq} Two move-reduced graphs that are not move-equivalent but have the same strongly decorated Newton polygons. The graph $\Gamma_2$ has vertices of degree $2$ at the vertical boundaries of the rectangle.}
	\end{figure}
	
	\begin{definition}\label{dfn:clicks}
		Let $\cc=(\cc_1,\cc_2,\dots,\cc_m)$ be a cyclic composition of $n=\cc_1+\cc_2+\dots+\cc_m$. Consider a partition
		$\Icc=\{I_1,I_2,\dots,I_m\}$ of $\Z/n\Z$ into cyclic intervals of size $|I_j|=\cc_j$ given by $I_1=[1,\cc_1]$, $I_2=[\cc_1+1,\cc_1+\cc_2]$, etc. 
		The \emph{rotation number $\rot(\cc)$} is the smallest integer $r\in[n]:=\{1,2,\dots,n\}$ such that $\rotsig^r(\Icc)=\Icc$, where $\rotsig:\Z/n\Z\to\Z/n\Z$ is the map sending $i\mapsto i+1\pmod n$ for all $i$, and $\rotsig(\Icc):=\{\rotsig(I_1),\rotsig(I_2),\dots,\rotsig(I_m)\}$.
	\end{definition}
	\noindent For example, $\rot((1,1,1,1,1,1))=1$, $\rot((2,1,2,1))=3$, and $\rot((2,2,1,1))=6$. We have $\rot((n))=n$ because by convention, we distinguish between cyclic intervals $[j,j+n-1]$ in $\Z/n\Z$ for different $j\in[n]$.
	
	The \emph{rotation number} of a collection $\bfcc=(\cc^\e)_{\e\in E}$ of cyclic compositions is given by
	\begin{equation}\label{eq:clicks_dfn}
		\clicks(\bfcc):=\gcd\{\rot(\cc^\e)\mid \e\in E\}.
	\end{equation}
	The following is our second main result.
	\begin{theorem}\label{thm:intro:move_eq}
		Let $\Ndec=(N,\bfcc)$ be a strongly decorated polygon. The set of move-reduced graphs $\Gamma$ satisfying $\Ndec(\Gamma)=\Ndec$ is a union of $\clicks(\bfcc)$ move-equivalence classes. Explicitly, two move-reduced graphs $\Gamma,\Gamma'$ are move-equivalent if and only if 
		\begin{equation*}%
			(\Ndec(\Gamma),\minv(\Gamma))=(\Ndec(\Gamma'),\minv(\Gamma')),
		\end{equation*} 
		where $\minv(\Gamma)\in\Z/\clicks(\bfcc)\Z$ is the \emph{modular invariant} defined in \cref{sec:intro:minv}.
	\end{theorem}
	
	\subsection{Modular invariant}\label{sec:intro:minv}
	We explain the construction of the modular invariant $\minv(\Gamma)$. Let $\Gamma$ be move-reduced and let $\Ndec(\Gamma)=(N,\bfcc)$ be its strongly decorated Newton polygon. Let $\e\in E(N)$ and set $r:=\rot(\cc^\e)$, $n:=|\cc^\e|=\ilen|\e|$. Thus, $r$ divides $n$. Let $\Rege$ be the set of connected components of $\T\setminus \bigcup_{\sa\in\SE} \sa$, which we call \emph{$\e$-regions}. Construct a labeling $\labe:\Rege\to\Z/n\Z$ 
	so that for any segment of a strand $\sa\in\SE$ adjacent to $\e$-regions $\Reg_-$ (resp., $\Reg_+$) to the right (resp., left) of $\sa$, the labels $\labe(\Reg_-),\labe(\Reg_+)\in\Z/n\Z$ satisfy $\labe(\Reg_+)\equiv \labe(\Reg_-)+1\pmod n$. 
	
	\begin{figure}
		\def\twd{0.145\textwidth}
		\begin{tabular}{cccccc}
			\includegraphics[width=\twd]{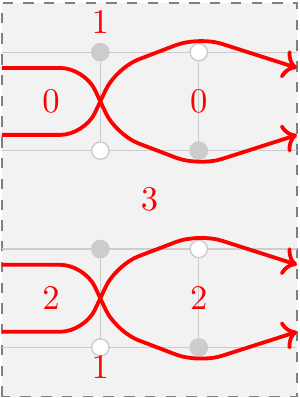}
			&
			\includegraphics[width=\twd]{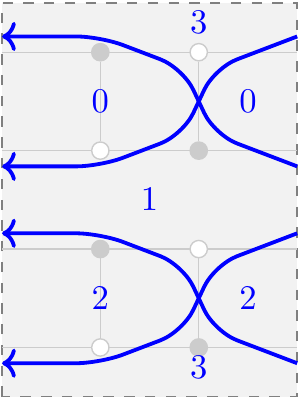}
			& 
			\includegraphics[width=\twd]{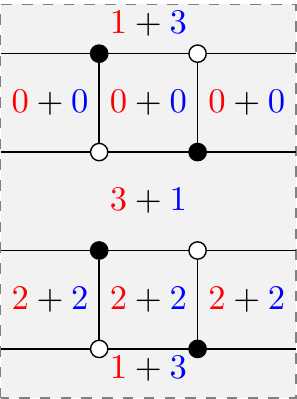}
			&
			\includegraphics[width=\twd]{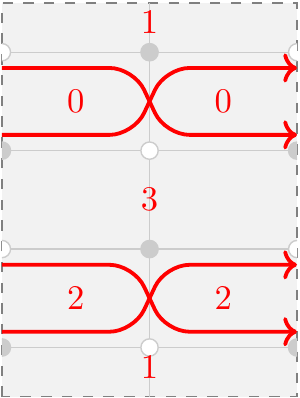}
			&
			\includegraphics[width=\twd]{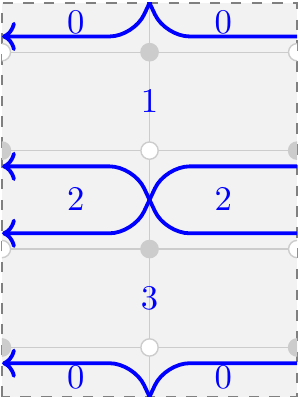}
			&
			\includegraphics[width=\twd]{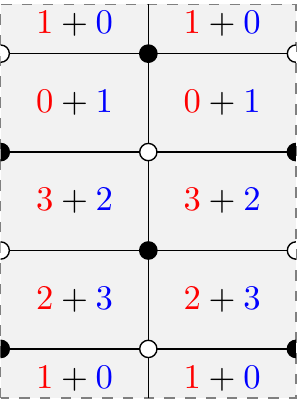}
			\\
			(a) $\laber$ for $\Gamma_1$. & (b) $\labeb$ for $\Gamma_1$. & (c) $\gamma$ for $\Gamma_1$. & (d)  $\laber$ for $\Gamma_2$.  & (e) $\labeb$ for $\Gamma_2$. & (f) $\gamma$ for $\Gamma_2$. 
		\end{tabular}
		\caption{\label{fig:twsq_minv} Computing the modular invariants (\cref{sec:intro:minv}) of graphs $\Gamma_1$ and $\Gamma_2$ from \cref{fig:twsq}. See \cref{ex:intro:minv}.}
	\end{figure}
	
	Clearly, there are $n$ ways to choose a labeling $\labe$ that satisfies the above conditions. We shall choose a particular one as follows. The labeling $\labe$ induces a partition $\Ibm_{\labe}$ of $\Z/n\Z$ into cyclic intervals so that for each strand $\sa\in\SE$, the associated cyclic interval contains $\labe(\Reg)$ for all $\Reg\in\Rege$ appearing immediately to the right of $\sa$; see \cref{fig:twsq_minv}  and \cref{ex:intro:minv}. Now, recall that $\cc^\e$ is a cyclic composition. Of all the cyclic shifts of $\cc^\e$, let $\cc^\e=(\cc^\e_1,\cc^\e_2,\dots,\cc^\e_m)$ be the lexicographically maximal one, and let $\Ibm_{\cc^\e}$ be the associated partition of $\Z/n\Z$ into cyclic intervals from \cref{dfn:clicks}. We say that the labeling $\labe$ is \emph{lex-maximal} if $\Ibm_{\labe}=\Ibm_{\cc^\e}$. Since $\rotsig^r(\Ibm_{\cc^\e})=\Ibm_{\cc^\e}$, we see that there are $n/r$ lex-maximal labelings  $\labe$. Fix one such labeling and let $\labem:\Rege\to\Z/r\Z$ be obtained by taking the values of $\labe$ modulo $r$. Thus, $\labem$ does not depend on the choice of $\labe$, and is an invariant of $\Gamma$.
	
	Repeat the above procedure for all $\e\in E(N)$. Let $\Regs(\Gamma)$ be the set of faces of $\Gamma$. We will construct a labeling $\lab:\Regs(\Gamma)\to\Z/d\Z$, where $d:=\clicks(\bfcc)$. For each face $\Reg\in\Regs(\Gamma)$, we set $\lab(\Reg):=\sum_{\e\in E(N)} \labem(\Reg) \pmod d$. This is a well-defined element of $\Z/d\Z$ in view of~\eqref{eq:clicks_dfn}. Moreover, any two adjacent faces $\Reg,\Reg'$ of $\Gamma$ are separated by two strands going in the opposite directions, so $\lab(\Reg)=\lab(\Reg')$. In other words, the labeling $\lab$ is constant. By definition, its value is the \emph{modular invariant} $\minv(\Gamma)\in \Z/d\Z$. 
	
	The moves \MMs induce bijections between $\e$-regions. Since all the faces involved in \MMs except the middle face in (M1) are in the same $\e$-regions, $\minv(\Gamma)$ is invariant under move-equivalence.
	
	\begin{example}\label{ex:intro:minv}
		Consider the graphs $\Gamma_1$ and $\Gamma_2$ from \cref{fig:twsq}. Let $\Ndec=(N,\bfcc)$ be their strongly decorated Newton polygon shown in \figref{fig:twsq}(e). Thus, $N$ is a line segment of length $4$, and let $\er=(4,0)$ and $\eb=(-4,0)$ be the two edges of $N$. We have $\cc:=\cc^{\er}=\cc^{\eb}=(2,2)$, and $\rot(\cc)=2$. Examples of lex-maximal labelings $\laber$ and $\labeb$ for $\Gamma_1$ and $\Gamma_2$ are shown in \figref{fig:twsq_minv}(a--b,d--e). The labeling $\lab$ for $\Gamma_1$ and $\Gamma_2$ is obtained by taking the labeling $\laber+\labeb$ shown in \figref{fig:twsq_minv}(c,f) modulo $\clicks(\bfcc)=2$. We see that in fact $\lab(\Reg_1)= 0\in\Z/2\Z$ is even for each face $\Reg_1$ of $\Gamma_1$, while $\lab(\Reg_2)= 1\in\Z/2\Z$ is odd for each face $\Reg_2$ of $\Gamma_2$. Therefore, $\minv(\Gamma_1)=0\in\Z/2\Z$ and $\minv(\Gamma_2)=1\in\Z/2\Z$, which is consistent with \cref{thm:intro:move_eq} since the graphs $\Gamma_1$ and $\Gamma_2$ are not move-equivalent.
	\end{example}

	\subsection{Overview of the proof}
	We shall proceed by relating bipartite graphs embedded in $\T$ to elements of the \emph{double affine symmetric group}, i.e., pairs of affine permutations. In \cref{sec:plabic,sec:cylinder}, we show the following result.

	\begin{theorem}\label{thm:intro:vertical}
		For any move-reduced graph $\Gamma$, exactly one of the following holds:
		\begin{enumerate}[label=(\roman*)]
			\item\label{intro:Gamma_monogon} $\Gamma$ has a single strand that is a simple \trvl loop. In this case, $\Gamma$ has no perfect matchings and has a different number of black and white vertices.
			\item\label{intro:Gamma_vertical} $\Gamma$ is move-equivalent to a graph $\Gamma'$ such that, for a suitable choice of the fundamental domain, each strand $\sa\in \SGp$ with $[\sa]=(i,j)$ intersects the vertical line $x=0$ minimally, i.e., exactly $|i|$ times.
		\end{enumerate}
	\end{theorem}

	\noindent In part~\ref{intro:Gamma_monogon}, a \emph{\trvl loop} is a strand $\sa$ satisfying $[\sa]=0$, and a \trvl loop $\sa$ is called \emph{simple} if the lift of $\sa$ to $\R^2$ under the covering map $\R^2\to\T$ is a simple (i.e., non self-intersecting) closed curve; see e.g. \figref{fig:parallelbigons}(a). In part~\ref{intro:Gamma_vertical}, choosing a fundamental domain corresponds to the standard $\SL_2(\Z)$-action on the Newton polygon of~$\Gamma$. 

	In \cref{sec:affine_plabic_fence}, we show that if~\ref{intro:Gamma_vertical} holds, then $\Gamma'$ can be put into a particular form called an \emph{affine plabic fence}. Such graphs correspond to shuffles of reduced words of two affine permutations on commuting sets of indices. In \cref{sec:affine-perm-cycl,sec:c-equiv-structure}, we study the associated conjugation problem for the affine symmetric group, relying on the results of~\cite{HN,Mar}. Finally, we complete the proof in Sections~\ref{sec:proof:exists}--\ref{sec:proof:move_eq}. %

	\subsection{Previous results}
	Our results specialize to~\cite[Theorem~2.5]{GK13} in the case of \emph{minimal graphs}, i.e., when no strand intersects itself in $\T$. This corresponds to all parts of $\la^\e$ and $\cc^\e$ being equal to $1$ for all $\e\in E(N)$.
	
	The idea of relating bipartite graphs embedded in $\T$ to conjugation of double affine permutations is not new and appears in~\cite{FM,GSZ}. A discussion of graphs that are move-reduced but not minimal in the sense of~\cite{GK13}, and in particular the graph $\Gamma_2$ in \figref{fig:move_red_vs_minimal}(b), appears in~\cite[Section~8.3]{FM}. 
	
	In~\cite[Section~4.4]{GSZ}, the authors also consider the problem of classifying move-reduced graphs and their move-equivalence classes. They associate a weakly decorated Newton polygon to each graph and prove a lemma classifying conjugacy classes in the double affine symmetric group. However, this classification does not imply a classification of move-reduced bipartite graphs and their move-equivalence classes. The reason is that the moves \MMs correspond only to particular kinds of conjugation in the affine symmetric group (see \cref{dfn:c_equiv}), not to arbitrary conjugation. 
	This discrepancy leads us to studying strongly decorated Newton polygons and modular invariants. We also note that in~\cite[Section~4.4]{GSZ}, the authors rely on \cref{thm:intro:vertical} and refer to~\cite{FM} for its proof, however, the argument in~\cite[Section~4.1]{FM} only applies to graphs whose strands go monotonously from left to right.

	\subsection*{Acknowledgments}
	We thank Xuhua He for bringing the paper~\cite{Mar} to our attention. We also thank Timothée Marquis for discussions related to~\cite{Mar} and for updating and extending~\cite{Mar} to the generality suited for our needs (cf. Remarks~\ref{rmk:alcove} and~\ref{rmk:Mar}). The first author is grateful to Thomas Lam for ideas that originated during the development of~\cite{GL_cat_combin}, which were influential for our overall proof strategy and specifically for the arguments in \cref{sec:affine-perm-cycl,sec:c-equiv-structure}.

	\section{Plabic graphs and triple-crossing diagrams}\label{sec:plabic}

	We discuss the properties of bipartite graphs embedded in $\T$ and explain how to recast them in the equivalent languages of \emph{plabic graphs}~\cite{Postnikov} and \emph{triple-crossing diagrams}~\cite{Thurston}.

	\subsection{Triple-crossing diagrams in the disk} \label{sec:tcddisk}
	The results of this section were independently discovered by \cite{Postnikov} and \cite{Thurston}. We state the results in terms of Thurston's notion of triple-crossing diagrams. 
	\begin{definition}\label{dfn:tcd}
		A \textit{triple-crossing diagram $D$ in the disk} $\D:=[0,1]^2$ is a smooth immersion of a disjoint union of oriented circles and closed intervals into $\D$, defined up to isotopy. The image of a connected component is called a \textit{strand}. The image of a circle is called a \textit{loop} and the image of a closed interval is called an \textit{arc}. The immersion is required to satisfy the following conditions:
		\begin{enumerate}
			\item Three strands cross at each intersection point. We call these intersection points \textit{triple crossings}.
			\item The endpoints of the arcs are distinct points on the boundary of $\D$, and there are no other points of $D$ on the boundary of $\D$.
			\item The orientations of the strands induce consistent orientations on the boundaries of the faces of $D$.
		\end{enumerate}
	\end{definition}
	Here, a \emph{face} of $D$ is a connected component of $\D\setminus D$. Property (3) implies that around every triple crossing, the orientations of strands alternate in and out, and that the orientations of the boundary vertices alternate in and out. If $D$ has $n$ arcs, then it has $2n$ boundary points, and the connectivity of the arcs induces a matching of the in-boundary points with the out-boundary points, called the \textit{trip permutation} in \cite{Postnikov}.
	
	\begin{definition}
		A triple-crossing diagram $D$ in the disk $\D$ is said to be \textit{reduced} if it has the fewest number of triple crossings among all triple-crossing diagrams with the same boundary matching.
	\end{definition}
	
	\begin{definition}\label{dfn:tcd_move_eq_red}
		Two triple-crossing diagrams are said to be \textit{move-equivalent} if they are related by move (M1)$'$ in Figure \ref{fig:localmovetcd}. A triple-crossing diagram $D$ is called \textit{move-reduced} if it is not move-equivalent to a triple-crossing diagram $D'$ to which one of the reduction moves (R1)$'-$(R2)$'$ can be applied.
	\end{definition}
	
	\begin{remark}
		Postnikov's reduction move (R1)$'$ in Figure \ref{fig:localmovetcd} differs from Thurston's $1-0$ move (see Figure \ref{fig:thurston10}). Postnikov's move will be more important for our eventual goal of understanding the behavior of the dimer model under taking limits, since it preserves dimer partition functions (cf. \cite[Theorem 12.1]{Postnikov}). On the other hand, Thurston's move preserves the boundary matching. It appears naturally in connection with double-affine permutations (see~\cref{sec:affine_plabic_fence} and~\cref{rem:postnikov10}), and will also be used in the proof of~\cref{thm:intro:move_red}. %
	\end{remark}
	
	\begin{figure}
		\def\twd{0.23\textwidth}
		\def\hsp{\qquad\qquad}
		\begin{tabular}{ccccc}
			\includegraphics[width=\twd]{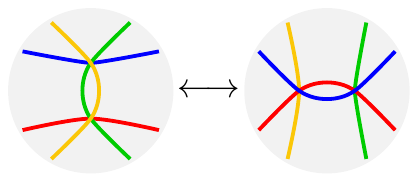}
			& \hsp &
			\includegraphics[width=\twd]{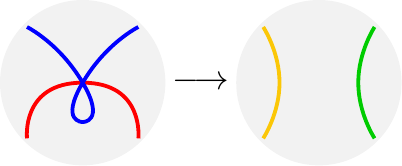}
			& \hsp &
			
			\includegraphics[width=\twd]{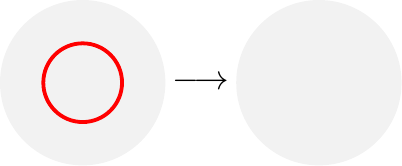}
			\\
			(M1)$'$. & & (R1)$'$. & & (R2)$'$.
		\end{tabular}
		\caption{\label{fig:localmovetcd} Equivalence move (M1)$'$ and reduction moves (R1)$'-$(R2)$'$ for triple-crossing diagrams. Each move has two possible strand orientations. (R2)$'$ removes a strand that is a simple loop. }
	\end{figure}

	A \textit{monogon} in $\D$ is a strand with a self-intersection. A \textit{parallel bigon} in $\D$ is a pair of strands with two intersection points $x\neq y$, with both strands oriented from $x$ to $y$. 
	
	\begin{theorem}[{\cite[Theorem 13.2 and Lemma 13.6]{Postnikov} and \cite[Theorem 7]{Thurston}}] \label{thm:diskminimal} Let $D$ be a triple-crossing diagram in $\D$. Then, the following are equivalent:
		\begin{enumerate}
			\item $D$ is move-reduced;
			\item $D$ is reduced; 
			\item $D$ contains no loops, monogons, or parallel bigons. 
		\end{enumerate}
	\end{theorem}
	
	\begin{theorem}[{\cite[Corollary 14.7]{Postnikov} and \cite[Theorem 3]{Thurston}}]\label{thm:disk}
		All $n!$ matchings of in- and out-boundary points are realizable by move-reduced triple-crossing diagrams. 
	\end{theorem}
	
	\begin{theorem} [{\cite[Theorem 13.4]{Postnikov} and \cite[Theorem 5]{Thurston}}]
		Any two move-reduced triple-crossing diagrams with the same boundary matching are move-equivalent.
	\end{theorem}
	Each pair of in- and out-endpoints in the matching divides the boundary of $\D$ into two intervals. Suppose that $I$ is a minimal such interval with respect to inclusion. We say that a strand $\sa$ whose endpoints are the endpoints of $I$ is \textit{boundary-parallel} if there are no triple crossings within the region between $\sa$ and $I$.
	
	\begin{proposition}[{\cite[Proof of Theorem 13.4 and Figure 13.4]{Postnikov} and \cite[Lemma 12]{Thurston}}] \label{prop:parallel}
		Suppose $I$ is an inclusion-minimal interval of the boundary matching of a move-reduced triple-crossing diagram $D$, and let $\sa$ be the strand in $D$ whose endpoints are the endpoints of $I$. Then, $D$ is move-equivalent to a triple-crossing diagram $D'$ in which $\sa$ is boundary-parallel. 
	\end{proposition}

	\subsection{Plabic graphs and triple-crossing diagrams on the torus} \label{sec:plabictcd}

	A \textit{plabic graph} $\Gamma=(B \sqcup W,E)$ on a torus $\T$ is a (finite) graph embedded in $\T$ such that:
	\begin{enumerate}
		\item The vertices of $\Gamma$ are colored black or white. The set of black vertices (resp., white vertices) is denoted by $B$ (resp., $W$).
		\item The set of edges of $\Gamma$ is denoted by $E$. Each edge is incident to two vertices of opposite colors or incident to two white vertices.
		\item The black vertices are trivalent.
	\end{enumerate}
	
	We identify plabic graphs that are related by contracting an edge incident to two distinct white vertices into a single white vertex. Therefore, we can assume that each white-white edge is a loop based at a white vertex. 
	
	\begin{remark}
		Our definition of a plabic graph is more restrictive than that of~\cite{Postnikov}. Such plabic graphs were previously studied under the name \emph{white-partite}~\cite[Definition~7.14]{GPW} or \emph{black-trivalent}~\cite[Definition~8.1, Remark~8.2]{G_crit}.
	\end{remark}

	\begin{definition}\label{dfn:tcd_T}
		A \textit{triple-crossing diagram $D$ on the torus} $\T$ is a smooth immersion of a disjoint union of oriented circles into $\T$. The image of a circle is called a \textit{strand}, and the set of strands of $D$ is denoted $\SD$. The immersion is required to satisfy the following conditions:
		\begin{enumerate}
			\item\label{TCD1} Three strands cross at each intersection point. We call these intersection points \textit{triple crossings}.
			\item\label{TCD2} The orientations of the strands induce consistent orientations on the boundaries of the faces of $D$.
		\end{enumerate}
	\end{definition}
	Similarly to \cref{dfn:tcd}, a \emph{face} of $D$ is a connected component of $\T\setminus D$. 	The property~\eqref{TCD2} implies that around every triple crossing, the orientations of the strands alternate in and out. (However, the converse need not hold if $D$ has a non-contractible face.) Each strand $\sa$ in $D$ determines a homology class $[\sa] \in H_1(\T,\Z) \cong \Z^2$.
	\begin{lemma}\label{lemma::sumhomiszero}
		The sum of the homology classes of all strands is $0$ in $H_1(\T,\Z)$.
	\end{lemma}
	\begin{proof}
		Let $R_+$ (resp., $R_-$) denote the union of the faces of $D$ such that the induced orientation is counterclockwise (resp., clockwise). Then, by property (2), we have that $\sum_{\sa \in \SD}\sa=\partial R_+ = -\partial R_-$ as $1$-cycles in $\T$, and $\overline{R}_+ \cup \overline{R}_-=\T$. Therefore,
		\[
		2 \sum_{\sa \in \SD}[\sa]=[\partial R_+]-[\partial R_-]=[\partial \T]=0. \qedhere
		\]
	\end{proof}
	\begin{figure}
		\def\twd{0.1\textwidth}
		\def\hsp{\hspace{0.1in}}
		\begin{tabular}{cccccc}
			\includegraphics[width=\twd]{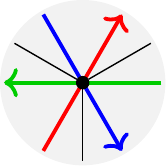}
			& \qquad &
			\hsp
			\includegraphics[width=\twd]{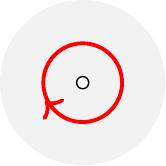}
			\hsp
			&
			\hsp
			\includegraphics[width=\twd]{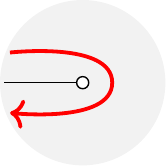}
			\hsp
			&
			\hsp
			\includegraphics[width=\twd]{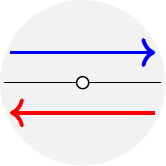}
			\hsp
			&
			\hsp
			\includegraphics[width=\twd]{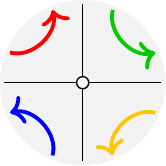}
			\hsp
			\\
			(a) Black vertex (degree-three).& &\multicolumn{4}{c}{(b) White vertex (arbitrary degree).} %
		\end{tabular}
		\caption{The procedure to convert plabic graphs into triple-crossing diagrams and vice versa. }\label{fig:tcd=plabic}
	\end{figure}	
	
	\begin{remark}	\label{rmk:translate1}
		A triple-crossing diagram can be converted into a plabic graph and vice versa using the local procedure shown in Figure \ref{fig:tcd=plabic}. The notions of \emph{move-reduced} and \emph{move-equivalent} plabic graphs are given by \cref{dfn:tcd_move_eq_red}. 
	\end{remark}
	
	\begin{remark}\label{rmk:translate2}
		If $\Gamma$ is a bipartite graph in $\T$ with all black vertices of degree at least three then one can convert $\Gamma$ into a plabic graph by applying a sequence of black uncontraction moves~\Mres.  When $\Gamma$ has black vertices of degree zero, one, or two,\footnote{This applies especially to the case of a degree-two black vertex connected to the same white vertex by both edges.} extra care needs to be taken; see \cref{sec:bipartite-to-tcd}. Conversely, any plabic graph can be converted into a bipartite graph by placing a black vertex of degree two in the middle of each white-white edge. 
		
		The notions of weakly/strongly decorated Newton polygons %
		and modular invariants introduced in \cref{sec:intro:move_reduced} for bipartite graphs extend to plabic graphs in an obvious way. Using \cref{rmk:translate1}, we can transfer them to triple-crossing diagrams. 
	\end{remark}

	In what follows, we will prove the versions of our main results translated into the language of triple-crossing diagrams and plabic graphs. The proof for bipartite graphs will then follow from the construction in \cref{sec:bipartite-to-tcd}. For instance, in \cref{sec:proof:move_red}, we will prove the following version of \cref{thm:intro:move_red}.
	\begin{theorem}  \label{thm:tcdmove_red}
		Let $D$ be a triple-crossing diagram with weakly decorated polygon $\Nwdec=(N,\bfla)$. Assume that $N$ is not a single point. Then, the following are equivalent:
		\begin{enumerate}
			\item\label{item:D:move_red} $D$ is move-reduced;
			\item\label{item:D:area} $D$ has no connected components that are contractible in $\T$ and contains $2\Area(N)+\excess{\bfla}$ triple crossings. 
		\end{enumerate}
		Similarly to \cref{rmk:intro:area}, $2\Area(N)+\excess{\bfla}$ is the minimal possible number of triple crossings for a triple-crossing diagram with weakly decorated Newton polygon $(N,\bfla)$. 
	\end{theorem}

	Let $\pi:\R^2 \rightarrow \T$ denote the universal covering map. The following result will be proved in~\cref{sec:proof:propertiestcdmovered}.
	\begin{proposition} \label{prop:propertiesofmoveredtcd}
		Let $D$ be a move-reduced triple-crossing diagram with Newton polygon~$N$.
		\begin{enumerate}
			\item\label{prop:appB_propertiesofmoveredtcd1}  The preimage $\tilde D$ contains no closed loops, and any lift $\tilde S$ of a strand $\sa \in \SD$ does not intersect itself;
			\item\label{prop:appB_propertiesofmoveredtcd2} Any strand $\sa \in \SD$ intersects itself $\ilen|[\sa]|-1$ times;
			\item\label{prop:appB_propertiesofmoveredtcd3}  Any two distinct parallel strands $\sa,\sa' \in \SD$ do not intersect, and there is no face of $D$ that contains portions of both strands in its boundary. 
		\end{enumerate}
	\end{proposition}
	
	By part~\eqref{prop:appB_propertiesofmoveredtcd3} of~\cref{prop:propertiesofmoveredtcd}, there is a natural cyclic order on each set of parallel strands, so the strongly decorated Newton polygon $\Ndec$ is well-defined.

	\section{Reduction to the cylinder}\label{sec:cylinder}
	The goal of this section is to prove the triple-crossing diagram version of \cref{thm:intro:vertical}. 
	Consider a move-reduced triple-crossing diagram $D$ on the torus $\T$. Let $\rectangle$ be a fundamental rectangle for $\T$, and let $\u, \d, \l, \r$ denote the up, down, left and right sides of $\rectangle$, respectively. Identifying the $\u$ and $\d$ sides, we get a cylinder $\A$, and further identifying the $\l$ and $\r$ sides, we get a torus $\T$. We have quotient maps $\rectangle \rightarrow \A \rightarrow \T$. The images of the $\u$ and $\d$ sides in $\A$ or in $\T$ coincide and are referred to as the \emph{$\u-\d$ side}. Similarly, the images of the $\l$ and $\r$ sides in $\T$ are referred to as the \emph{$\l-\r$ side}.

	We say that triple-crossing diagrams $D$ and $D'$ are \textit{isotopic in $\T$} if there is an ambient isotopy of $\T$ taking $D$ to $D'$. When applying such isotopies, we fix the fundamental rectangle $\rectangle$. Using an isotopy in $\T$ if necessary, we assume that the intersections of strands with the sides of $\rectangle$ are transverse. A strand $\sa$ with homology class $(i,j)$ must intersect the $\l-\r$ side at least $i$ times and the $\u-\d$ side at least $j$ times. The preimage of a strand under the map $\A \rightarrow \T$ is either a union of arcs with endpoints on the boundary of $\A$ or a closed loop in $\A$.

	\subsection{Pushing strands through the boundary}
	
	\begin{figure}
		\def\twd{0.3\textwidth}
		\def\twdd{0.4909090909\textwidth}
		\begin{tabular}{ccc}
			\includegraphics[width=\twd]{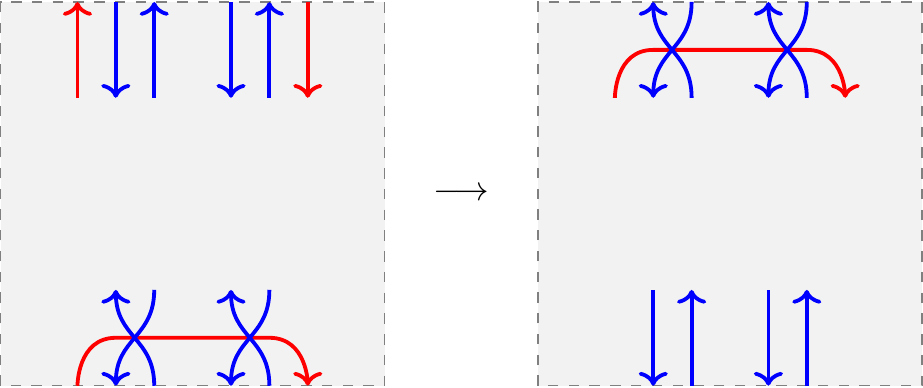}
			& \qquad \qquad &
			\includegraphics[width=\twdd]{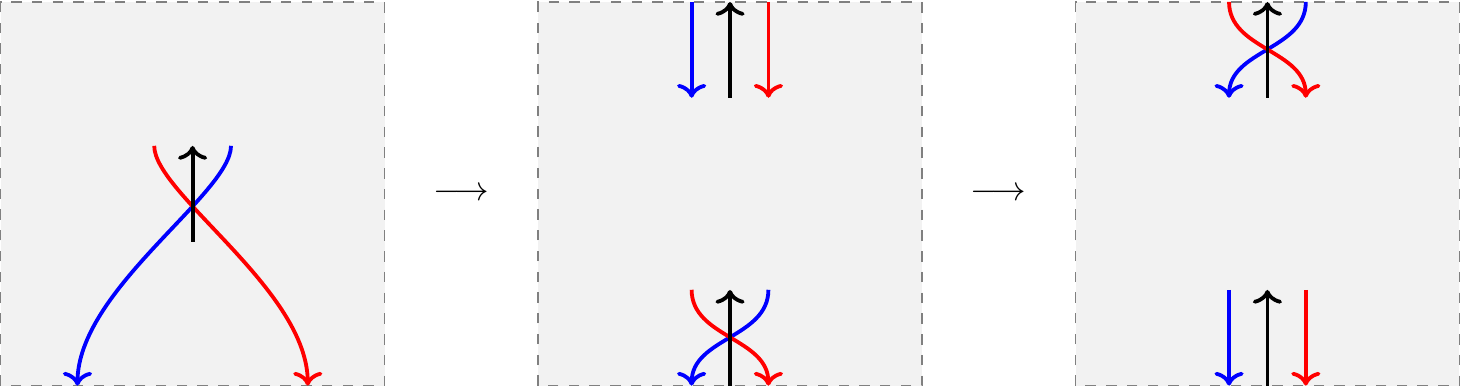}\\
			
			(a) Move (P). && (b) Move (T).
		\end{tabular}
		\caption{The move (P), pushing a boundary-parallel strand past the $\d$-side of $\rectangle$ removing the intersection points with the $\d$-side, and the move (T), interchanging the relative order of the endpoints of the red and blue strands along the $\d$-side of $\D$. thereby removing the triangular region bounded by the strands and the $\d$-side of $\rectangle$.} \label{fig:movePT}
	\end{figure}

	\begin{lemma} \label{lemma:movep}
		Suppose we have a strand $\sa$ in $\rectangle$ with both endpoints on a side $s$ of $\rectangle$. Then, using  moves and isotopy 
		in $\T$, we can remove the endpoints of $\sa$ in $s$ without increasing the number of intersections of any other strands with the boundary of $\rectangle$.
	\end{lemma}
	\begin{proof}
		Let $I$ denote the interval in $s$ between the endpoints of $\sa$. 
                  Suppose $I$ is minimal with respect to inclusion among all intervals on the boundary of $\rectangle$ between endpoints of strands. Using Proposition \ref{prop:parallel}, we make $\sa$ boundary-parallel, and then apply an isotopy in $\T$ pushing the strand $\sa$ past the $s$-side of $\rectangle$
                  (\figref{fig:movePT}(a)).

                  If $I$ is not minimal, we use induction on the number of intervals contained in $I$. Suppose $I'$ is an inclusion-minimal interval contained in $I$. Using the above procedure, we can remove the endpoints of $I'$ and thereby reduce the number of intervals contained in $I$.
	\end{proof}
	
	\begin{definition}
		We call the above procedure \textit{move (P)}; see \figref{fig:movePT}(a).
	\end{definition}

	\begin{lemma}[{\cite[Figure~12]{GK13}}] \label{lemma:gk}
		Suppose we have a pair of strands in $\rectangle$ that have consecutive in- or out-endpoints on a side of $\rectangle$, and  moreover, suppose that these two strands cross in $\rectangle$. Then, the relative order of the two endpoints along $s$ can be reversed using moves and isotopy in $\T$ without increasing the number of intersections of any other strands with the boundary of $\rectangle$.
	\end{lemma}
	We remove the consecutiveness assumption from Lemma \ref{lemma:gk}. 
	\begin{lemma} \label{lemma:movet}
		Suppose we have a pair $(\sa, \sa')$ of strands with endpoints on a side $s$ of $\rectangle$ with the same orientation (i.e., both in or both out), and suppose that $\sa,\sa'$ cross in $\rectangle$. Then, the relative order of the endpoints of $\sa,\sa'$ in $s$ can be reversed using moves and isotopy in $\T$, thereby removing the triangular region bounded by $\sa,\sa'$ and the side $s$ without increasing the number of intersections of any other strands with the boundary of~$\rectangle$. 
	\end{lemma}
	\begin{proof}
		Assume without loss of generality that both $\sa,\sa'$ have an out-endpoint in $s$. Let $I$ be the interval in $s$ between the endpoints of $\sa,\sa'$. Use move (P) to remove any strands that have both endpoints in $I$. Then, any strand with an out-endpoint in $I$ must cross at least one of $\sa,\sa'$. Let $\ell$ be the number of crossings formed by pairs of strands having an out-endpoint in $I$. Repeatedly using Lemma \ref{lemma:gk}, we can decrease $\ell$ until it becomes equal to $1$, in which case we can swap the endpoints of $\sa,\sa'$. 
	\end{proof}
	
	\begin{definition}
		We call the procedure in Lemma \ref{lemma:movet} \textit{move (T)}; see \figref{fig:movePT}(b).   
	\end{definition}
	
	\subsection{Affine matchings}\label{sec:affine-matchings}
	Fix $n\geq1$. Consider an infinite vertical strip $\Strip$ with points labeled 
	\begin{equation}\label{eq:label_strip}
		\dots, A_{\ovl0},A_1,A_{\ovl1},A_2,A_{\ovl2},\dots \quad\text{and}\quad \dots,B_{\ovl0},B_1,B_{\ovl1},B_2,B_{\ovl2},\dots
	\end{equation}
	on the left and the right boundary of $\Strip$, 
	so that the points $A_i,B_i$ are at the same height, and the points $A_{\ovl i},B_{\ovl i}$ are at the same height, for each $i\in\Z$. Let $A:=\{A_i\mid i\in\Z\}$, $\ovl A:=\{A_{\ovl i}\mid i\in\Z\}$, $B:=\{B_i\mid i\in\Z\}$, $\ovl B:=\{A_{\ovl i}\mid i\in\Z\}$. 
	
	\begin{definition}
		An \emph{affine matching} is a bijection $\am: A \sqcup \ovl B \to \ovl A \sqcup B$ such that $\am(A_{i+n})=\am(A_i)+n$ and $\am(B_{\ovl{i+n}})=B_{\ovl i}+n$ for all $i \in \Z$.
	\end{definition}
	\noindent This notion is closely related to the classical notion of \emph{affine permutations} discussed in \cref{sec:aff_perm_backgr}. An affine matching is represented by drawing an arrow from $x$ to $\am(x)$ inside $\Strip$ for all $x\in A\sqcup\ovl B$.

	A triple-crossing diagram $D$ in $\T$ gives rise to an affine matching $\amD$ as follows. Let $\rectangle$ be a fundamental rectangle. Using an $\operatorname{SL}_2(\Z)$ transformation, we can assume that there are no strands with homology classes in $(\Z \times \{0\}) \cup (\{0\} \times \Z)$ other than \trvl loops. Let $n$ denote the number of intersection points of strands with the $\l-\r$ side. Let $\mathbb S$ denote the infinite vertical strip that is the universal cover of $\A$. Then, $\mathbb S$ consists of $\Z$-many copies of $\rectangle$ glued along the $\u-\d$ sides, which we label $\dots, \rectangle_{-1}, \rectangle_0, \rectangle_{1},\dots$ from bottom to top. Applying an isotopy in $\T$, we may assume that the bottom-most intersection point of a strand with the left side of $\rectangle$ is oriented in.
	Label the intersection points of strands with the boundary of $\Strip$ as in~\eqref{eq:label_strip}. Thus, the points in $A \sqcup \ovl B$ are in-endpoints and the points in $\ovl A \sqcup B$ are out-endpoints. 
	The connectivity of strands in $\Strip$ determines an affine matching $\amD$, which, moreover, has total signed number of crossings through any horizontal line equal to $0$.
	
	\begin{remark}\label{rmk:strand_word}
		Each strand $\sa$ in $\mathbb S$ determines a word $w_{\sa}$ in the alphabet $\{\u,\d,\l,\r\}$ which records the crossings of $\sa$ with the sides of $\rectangle$. Using move (P), we can assume that there are no occurrences of $\u \d$ or $\d \u$ in $w_{\sa}$; thus, we have $w_{\sa}=xy^kz$ for some $x,z\in\{\l,\r\}$, $y\in\{\u,\d\}$, and $k\geq 0$. For a strand $\sa$ such that $w_{\sa}=xy^kz$, we denote by $\sa_1,\dots,\sa_{k+1}$ the corresponding strands in $\rectangle$.%
	\end{remark}	
	
	\begin{lemma}\label{lemma:cross}
		If the strands $\sa$ and $\sb$ emanating from $A_i$ and $A_{i+1}$ cross in $\Strip$, then we can swap their endpoints $A_i$ and $A_{i+1}$ using moves and isotopy in $\T$ without increasing the number of intersections of any other stands with the boundary of $\mathbb A$. 
	\end{lemma}	
	\begin{proof}
		By translating the fundamental rectangle, we can assume that $i=1$. Without loss of generality, we can assume that $w_\sa= \r \u^k z$ where $k \geq 0$ and $z\in\{\l,\r\}$. If $w_\sb=\r \r$ or $w_\sb=\r \d v$ for some word $v$, then the segments $\sa_1$ and $\sb_1$ cross in $\rectangle$ and we use Lemma \ref{lemma:movet}. Suppose $w_{\sb}=\r \u^m z$ with $m\geq0$ and $z \in \{\l,\r\}$. Consider a crossing of $\sa$ and $\sb$ in $\mathbb S$. If this crossing belongs to $\sa_j$ then it must also belong to $\sb_j$. Let $j\geq1$ be the minimal index such that $\sa_j$ crosses $\sb_j$. If $j\geq2$, then  applying move (T) at the $\u-\d$ side, we can swap the bottom endpoints of $\sa_j$ and $\sb_j$ so that the strands $\sa_{j-1},\sb_{j-1}$ would cross. We continue this process until $j=1$, in which case we apply \cref{lemma:movet} at the $\l-\r$ side.
	\end{proof}

	\subsection{Proof of Theorem~\ref{thm:intro:vertical}} We are now ready to prove~Theorem~\ref{thm:intro:vertical}.
	
	\begin{lemma}
		The number of intersections of strands in $D$ with the sides of $\A$ can be made either the minimum possible (i.e., equal to $\sum_{\sa\in\SD} |i|$, where $[\sa]=(i,j)$) or equal to $2$ using moves and isotopy in $\T$.
	\end{lemma}
	
	\begin{proof}
		Consider the affine matching $\amD$ and consider a strand $\sa$ from $A_i$ to $A_{\ovl{j}}$ with the smallest value of $\dist(A_i,A_{\ovl j})$. Assume $i<j$. By minimality of $\dist(A_i,A_{\ovl j})$, the strand $\sb$ emanating from $A_{i+1}$ must cross the strand $\sa$. By Lemma \ref{lemma:cross}, we can swap the endpoints $A_i$ and $A_{i+1}$, decreasing $\dist(A_i,A_{\ovl j})$. Eventually, we force $\dist(A_i,A_{\ovl j})$ to be less than the height of $\rectangle$, in which case we apply move (P) and decrease $n$. We proceed until either $n=1$ or when there are no arcs from $A_i$ to $A_{\ovl j}$, in which case there will be also no arcs from $B_{\ovl i}$ to $B_j$ (for any $i,j\in\Z$).
	\end{proof}
	
	We now study the case when $n=1$. Let $\am$ be an affine matching with $n=1$ such that the total signed number of crossings through any horizontal line is equal to $0$. Then, either:
	\begin{enumerate}
		\item\label{am1} $\am(A_1)=B_{{k+1}}$ and $\am(B_{\ovl1})=A_{\ovl{-k+1}}$ for some $k \in \Z$; or
		\item\label{am2} $\am(A_1)=A_{\ovl{{k+1}}}$ and $\am(B_{\ovl{1}})=B_{-k+1}$ for some $k \in \Z$.
	\end{enumerate}
	Therefore, if $D$ is a triple-crossing diagram move-reduced in $\D$ such that $\amD$ satisfies~\eqref{am1}, then the number of intersections of strands with the sides of $\A$ is minimal and equal to $2$. Suppose that $\amD$ satisfies~\eqref{am2}. If $k=0$, we can use move (P) to remove the two intersection points, so the number of intersections of strands with sides of $\A$ is minimal and equal to $0$. However, if $k \neq 0$, the number of intersections of strands with the sides of $\A$ is not minimal (since we have $\sum_{\sa\in\SD} |i|=0$), and we call such a triple-crossing diagram \textit{exceptional}. In this case, $D$ consists of a single strand $\sa$ in $\T$ that is a \trvl loop (see~\figref{fig:parallelbigons}(a) for the associated bipartite graph when $k=1$). It is not hard to see that the strand $\sa$ is simple, that is, lifts to a non self-intersecting closed curve in $\R^2$. This is case~\ref{intro:Gamma_monogon} of \cref{thm:intro:vertical}. In order to complete the proof, we need to show that in this case, the associated bipartite graph $\Gamma$ has no perfect matchings.

	\begin{proposition}\label{prop:monogon}
		Let  $\Gamma$ be a move-reduced bipartite graph in $\T$. Suppose that $\Gamma$ has a single strand which is a simple \trvl loop. Then $\Gamma$ has a different number of black and white vertices, and, in particular, has no perfect matchings.
	\end{proposition}
	\begin{proof}
		Given a closed immersed curve $\curve:S^1\to\R^2$ with non-vanishing differential, we let $\wind\curve\in\Z$ denote its \emph{winding number}, which is the number counterclockwise turns made by the tangent vector of $\curve$. For a collection $\curves$ of such curves, we let $\wind\curves:=\sum_{\curve\in\curves} \wind\curve$ denote their total winding number. One can check that the total winding number $\wind\curves$ is invariant under the skein relation 
		\begin{equation}\label{eq:skein}
			\begin{tikzpicture}[scale=0.4,baseline=(Z.base)]
				\coordinate(Z) at (0,0);
				\def\lw{1pt}
				\def\rad{1.414}
				\fill[black!5] (0,0) circle (\rad);
				\draw[line width=\lw,->,>={latex}] (-1,-1)--(1,1);
				\draw[line width=\lw,->,>={latex}] (1,-1)--(-1,1);
				\node[scale=1] (A) at (2.5,0) {$\to$};
				\begin{scope}[xshift=5cm]
					\fill[black!5] (0,0) circle (\rad);
					\def\rc{5}
					\draw[line width=\lw,->,>={latex},rounded corners=\rc] (-1,-1)--(0,0)--(-1,1);
					\draw[line width=\lw,->,>={latex},rounded corners=\rc] (1,-1)--(0,0)--(1,1);
				\end{scope}
			\end{tikzpicture}.
		\end{equation}
		Let $\Gammat$ be the lift of $\Gamma$ to the universal cover $\R^2$ of $\T=\R^2/\Z^2$.  Let $\sa$ be the unique strand of $\Gamma$, and let $\sat$ be some lift of $\sa$ to $\R^2$. Thus, $\sat$ is a simple closed curve, and therefore $\wind\sat=\pm1$, depending on whether $\sat$ is oriented counterclockwise or clockwise. Any two lifts of $\sa$ differ by a shift in $\Z^2$. Let $N\gg1$ be a large positive integer, and let $\curves$ be the collection of all $\Z^2$-shifts of $\sat$ that are contained inside the square $[0,N]^2\subset\R^2$. There are $\Omega(N^2)$ such shifts, and therefore $\wind\curves=\pm\Omega(N^2)$. On the other hand, resolving all crossings in $\curves$ using the skein relation~\eqref{eq:skein}, we obtain a collection $\curves'$ of simple closed curves satisfying $\wind\curves=\wind{\curves'}$. Each of these curves will contain a single vertex of $\Gammat$ inside of it. Moreover, if the vertex of $\Gammat$ inside $\curve'\in\curves'$ is black (resp., white), then $\curve'$ is oriented counterclockwise (resp., clockwise). Therefore, the difference between the numbers of black and white vertices of $\Gammat$ contained inside $[0,N]^2$ is of size $\Omega(N^2)$. This implies that $\Gamma$ must have a different number of black and white vertices. 
	\end{proof}

	\section{Affine permutations, cycles, and slopes}\label{sec:affine-perm-cycl}
	As we will explain in \cref{sec:affine_plabic_fence}, \cref{thm:intro:vertical} allows one to recast bipartite graphs embedded in $\T$ and their moves into certain conjugation moves on pairs of affine permutations. In this and the next section, we develop the properties of affine permutations needed to complete the proofs of our main results. 
	
	Our proof strategy is inspired by that of~\cite{Mar}. The reader familiar with the theory of affine Coxeter groups and their reflection representations is encouraged to consult Remarks~\ref{rmk:alcove} and~\ref{rmk:Mar}.%
	
	\subsection{Background and notation}\label{sec:aff_perm_backgr}
	Let $n\geq1$ and recall that $[n]:=\{1,2,\dots,n\}$. 
	An \emph{affine permutation} is a bijection $f:\Z\to\Z$ satisfying $f(i+n)=f(i)+n$ for all $i\in\Z$. The group of affine permutations is denoted $\Saffn$ (where the group operation is given by composition of maps $\Z\to\Z$). For $f\in\Saffn$, set 
	\begin{equation*}%
		\nop(f):=n,\quad \kop(f):=\frac1n\sum_{i\in[n]} (f(i)-i), \quad\text{and}\quad \dop(f):=\gcd(\kop(f),\nop(f)).
	\end{equation*}
	It is known (see \cref{rmk:kopf} below) that $\kop(f)$ is always an integer. We have $\Saffn=\bigsqcup_{k\in\Z} \Saffkn$, where $\Saffkn:=\{f\in\Saffn\mid \kop(f)=k\}$. For $f\in\Saffn$, let $\fb\in\Sn$ be the unique permutation (i.e., bijection $[n]\to[n]$) satisfying $\fb(i)\equiv f(i)\pmod n$ for all $i\in[n]$. For $k\in\Z$, we denote by $\fkn\in\Saffkn$ the affine permutation sending $i\mapsto i+k$ for all $i\in\Z$. The affine permutation $f$ can be written in \emph{window notation} as $[f(1),f(2),\dots,f(n)]$, which completely determines $f(i)$ for all $i\in\Z$.

	\begin{figure}
		\def\twd{0.45\textwidth}
		\begin{tabular}{ccc}
			\includegraphics[width=\twd]{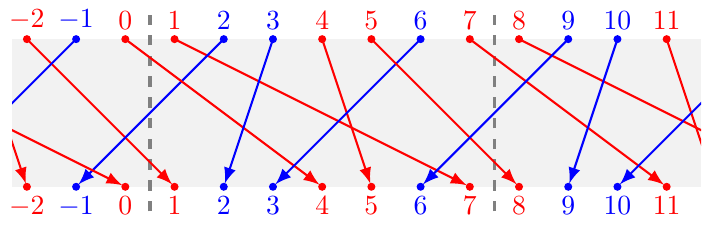}
			& \qquad &
			\includegraphics[width=\twd]{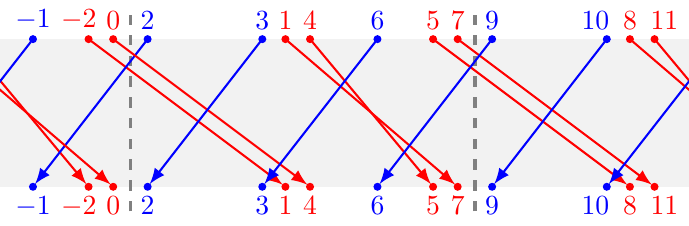}
			\\
			(a) Standard arrow diagram of $f$. && (b) Arrow diagram $\Diag(x)$.
		\end{tabular}
		\caption{\label{fig:arr_diag}Arrow diagrams of affine permutations; see \cref{sec:aff_perm_backgr,sec:arrow-diagrams}.}
	\end{figure}

	The group $\Saffon$ is a Coxeter group with generators $\CoxGens:=\{s_i\mid i\in[n]\}$, where the affine permutation $s_i:\Z\to\Z$ sends $i\mapsto i+1$, $i+1\mapsto i$, and $j\mapsto j$ for $j\not\equiv i,i+1 \pmod n$. For $i\in\Z$, we let $s_i:=s_{\bar i}$ where $\bar i\in[n]$ satisfies $\bar i\equiv i\pmod n$.
	
	The group $\Saffon$ is also known as the \emph{affine Weyl group of type $\At_{n-1}$}. Let $\Lan:=f_{1,n}\in\Saffxx_n^1$. Thus, $\Saffn=\Saffon\rtimes\<\Lan\>$. We will also be interested in the quotient group $\Sextn:=\Saffn/\<\Lan^n=\id\>$, known as the \emph{extended affine Weyl group of type $\At_{n-1}$}.
	The group $\Saffon$ is a subgroup of both $\Saffn$ and $\Sextn$. %
	We denote by $\rotsig:\Saffn\to\Saffn$ the \emph{rotation operator} given by $\rotsig(f):= \Lan f\Lan^{-1}$.
	
	Let $f\in\Saffn$. Define
	\begin{align*}%
		\Inv(f)&:=\{(i,j)\in\Z\times\Z\mid \text{$i<j$ and $f(i)>f(j)$}\},\\
		\ell(f)&:=\#\{(i,j)\in\Z\times\Z\mid \text{$i<j$, $f(i)>f(j)$, and $i\in[n]$}\}.
	\end{align*}
	The \emph{standard arrow diagram} of $f$ is obtained by drawing an arrow $(i,1)\to(f(i),0)$ for all $i\in\Z$; see \figref{fig:arr_diag}(a) for an example when $f=[7,-1,2,5,8,3,11]$ in window notation. The set $\Inv(f)$ consists of pairs of crossing arrows, and $\ell(f)$ counts the number of crossing arrows modulo the equivalence relation generated by $(i,j)\sim(i+n,j+n)$ for all $i,j\in\Z$. Alternatively, $\ell(f)$ is the minimal integer $l$ such that $f$ can be written as a product $f=s_{i_1}s_{i_2}\cdots s_{i_l} \Lan^k$ for some indices $i_1,i_2,\dots,i_l,k$; in this case, $s_{i_1}s_{i_2}\cdots s_{i_l} \Lan^k$ is called a \emph{reduced expression} for $f$.
	For the example in \figref{fig:arr_diag}(a), we have %
	\begin{equation*}%
		\kop(f)=1,\quad \ell(f)=11, \quad\text{and}\quad f=s_{3} s_{4} s_{6} s_{7} s_{2} s_{5} s_{6} s_{1} s_{4} s_{3} s_{2} \Lan.
	\end{equation*}
	\begin{remark}\label{rmk:kopf}
		In general, the integer $\kop(f)$ is equal to the signed number of intersections of the arrows with one of the dashed vertical lines.
	\end{remark}

	\begin{figure}
		\def\twd{0.5\textwidth}
		\def\hsp{\hspace{-0.25in}}
		\begin{tabular}{cc}
			\includegraphics[width=\twd]{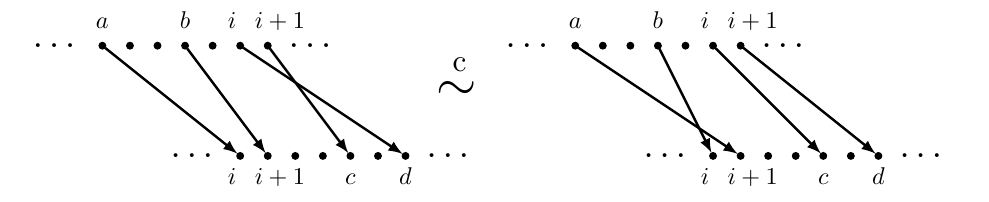}
			& \hsp
			\includegraphics[width=\twd]{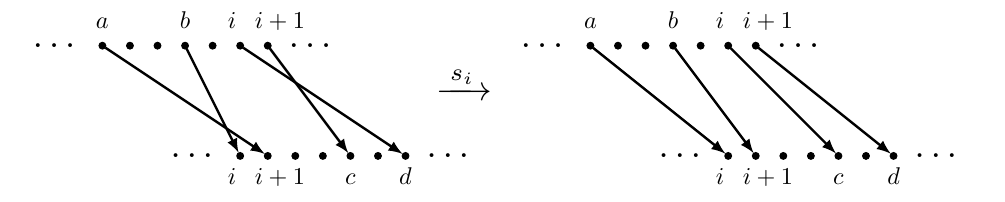}
		\end{tabular}
		\caption{\label{fig:c_equiv} The two affine permutations on the left are c-equivalent. The affine permutations $f,f'$ on the right satisfy $f\rasi f'$ but are not c-equivalent. See \cref{dfn:c_equiv}. Figure reproduced from~\cite[Figure~5]{GL_cat_combin}.}
	\end{figure}
	
	Following~\cite{GP93,GKP00,He07,He10,HN,Mar}, for $f,f'\in\Saffn$, we write $f\rasi f'$ if $f'=s_ifs_i$ and $\ell(f')\leq \ell(f)$. We write $f\to f'$ if there exists a sequence $f=f_0,f_1,\dots,f_m=f'$ of affine permutations such that for each $j\in[m]$, we have $f_{j-1}\rasi f_j$ for some $i\in[n]$.
	\begin{definition}\label{dfn:c_equiv}
		We say that $f,f'\in\Saffn$ are \emph{c-equivalent} if $f\to f'$ and $f'\to f$. In this case, we write $f\approx f'$.
	\end{definition}
	\noindent This terminology is borrowed from~\cite{GL_cat_combin}. 
	
	When talking about conjugacy classes, we always mean $\Saffon$-conjugacy classes, which we will usually denote by $\O$. Given a conjugacy class $\O$, let $\Omin$ be the set of elements of $\O$ of minimal length. We have the following important result.
	\begin{theorem}[{\cite[Theorem~2.9]{HN}}]\label{thm:HN}
		Let $f\in\Saffn$ and let $\O$ be the $\Saffon$-conjugacy class containing~$f$. 
		Then there exists $f'\in\Omin$ such that $f\to f'$.
	\end{theorem}
	\begin{definition}
		We say that $f\in\Saffn$ is \emph{c-reduced} if for all $f'\in\Saffn$ such that $f\to f'$, we have $\ell(f)=\ell(f')$ (or equivalently, $f'\to f$).
	\end{definition}\label{cor:HN}

	\noindent The following result follows immediately from \cref{thm:HN}.
	\begin{corollary} \label{cor:HN}
		An affine permutation $f\in\Saffn$ is c-reduced if and only if it has minimal length in its conjugacy class $\O$ (i.e., $f\in\Omin$).%
	\end{corollary}
	It is clear that $\approx$ yields an equivalence relation on the set of c-reduced elements in $\Saffn$. The goal of \cref{sec:affine-perm-cycl,sec:c-equiv-structure} is to give a solution to the following problem.
	\begin{problem}\label{prob:c_eq}
		Determine the structure of c-equivalence classes of c-reduced elements in $\Saffn$.
	\end{problem}
	\subsection{Cycles and slopes}\label{sec:cycles_and_slopes}
	A set $\Cycle\subset\Z$ is called \emph{$n$-periodic} if for all $i\in\Z$, we have $i\in \Cycle$ if and only if $i+n\in \Cycle$.
	\begin{definition}
		Let $f\in\Saffn$. A set $\Cycle\subset\Z$ is called \emph{$f$-closed} if it is nonempty, $n$-periodic, and for all $i\in\Z$, we have $i\in \Cycle$ if and only if $f(i)\in \Cycle$. 
	\end{definition}

	\begin{definition}\label{dfn:restrict}
		Let $\Cycle$ be an $f$-closed set. Because it is $n$-periodic, the set $\Cycle\cap[n]$ is nonempty. Let $\nopf(\Cycle):=\#(\Cycle\cap[n])$.  There exists a unique order-preserving bijection $\restrictC: \Cycle\to\Z$ sending $\min(\Cycle\cap[n])$ to $1\in\Z$. The \emph{restriction $f|_\Cycle\in\Saffx_{\nopf(\Cycle)}$} is an affine permutation defined by
		\begin{equation}\label{eq:restrict}
			f|_\Cycle:=\restrictC\circ f \circ\restrictC^{-1}.
		\end{equation}	
	\end{definition}
	Given an $f$-closed set $\Cycle$, we let
	\begin{equation*}%
		\nopf(\Cycle)=\nop(f|_\Cycle),\quad \kopf(\Cycle):=\kop(f|_\Cycle),\quad \dopf(\Cycle):=\dop(f|_\Cycle), \quad\text{and}\quad
		\slpfC:=\frac{\kopf(\Cycle)}{\nopf(\Cycle)}.
	\end{equation*}
	The rational number $\slpfC$ is called the \emph{slope} of $\Cycle$. Thus, we have $f|_\Cycle\in\Saffxx_{n'}^{k'}$ for $n'=\nopf(\Cycle)$ and $k'=\kopf(\Cycle)$. 

	\begin{definition}
		A \emph{cycle of $f$} is a minimal by inclusion $f$-closed set $\Cycle$. The set of cycles of $f$ is denoted $\CLF$.
	\end{definition}
	\noindent Thus, the cycles of $f$ are in bijection with the cycles of $\fb$, and a nonempty subset of $\Z$ is $f$-closed if and only if it is a disjoint union of cycles of $f$. For $i\in \Z$, we write $\slpfi:=\slpfC$, where $\Cycle$ is the cycle of $f$ containing $i$. 
	
	\begin{example}
		Let $f=[7,-1,2,5,8,3,11]$ in window notation be the affine permutation in \figref{fig:arr_diag}(a). Then $f$ has two cycles: $\Cred$ (resp., $\Cblue$) consists of all $i\in\Z$ congruent to $1,4,5,7$ (resp., to $2,3,6$) modulo $n=7$. We have
		\begin{align*}
			\nopf(\Cred)&=4,  &\kopf(\Cred)&=2, &\dopf(\Cred)&=2, &\slpfx(\Cred)&=1/2,\\
			\nopf(\Cblue)&=3,  &\kopf(\Cblue)&=-1, &\dopf(\Cblue)&=1, &\slpfx(\Cblue)&=-1/3.
		\end{align*}
	\end{example}
	
	Given $f\in\Saffn$ and $\slp\in\Q$, we set
	\begin{equation*}%
		\begin{split}
			\CLFS&:=\{\Cycle\in\CLF\mid\slpfC=\slp\}, 
			\qquad
			\slupp:=\bigsqcup_{\Cycle\in\CLFS} \Cycle, \\ \nopf(\slp)&:=\sum_{\Cycle\in\CLFS} \nopf(\Cycle),  \quad \kopf(\slp):=\sum_{\Cycle\in\CLFS}\kopf(\Cycle),\quad \dopf(\slp):=\sum_{\Cycle\in\CLFS} \dopf(\Cycle).
		\end{split}
	\end{equation*}

	For $f\in\Saffn$, we set
	\begin{equation*}%
		\SLF:=\{\slp\in\Q\mid\text{ $\slupp$ is nonempty}\}; \quad\text{therefore,}\quad \bigsqcup_{\slp\in\SLF} \slupp=\Z.
	\end{equation*}
	For $\slp\in\SLF$, we have $\slp=\kopf(\slp)/\nopf(\slp)$ and $\gcd(\kopf(\slp),\nopf(\slp))=\dopf(\slp)$. 
	\begin{definition}\label{dfn:bflaf}
		Let $\bflaf:=(\lafs)_{\slp\in\SLF}$, where $\lafs$ is the integer partition of $\dopf(\slp)$ induced by $(\dopf(\Cycle))_{\Cycle\in\CLFS}$.
	\end{definition}

	\subsection{Arrow diagrams}\label{sec:arrow-diagrams}
	
	Let $\Ln$ be the quotient of the space
\begin{equation*}%
  \{x:\Z\to\R\mid \xx(i+n)=\xx(i)+1\text{ for all $i\in\Z$}\}
\end{equation*}
        by the space of constant maps $\Z\to\R$. 
	It is an affine space of dimension $n-1$. For any $g\in \Saffn$, we have a point $\frac1n g\in\Ln$ sending $i\mapsto \frac1n g(i)$. To a point $x\in\Ln$, we associate a \emph{labeled point configuration} $\LPConf(x)$, that is, a collection of labeled points on the real line: a point labeled $i\in\Z$ is located at coordinate $\xx(i)$. We denote $\Im(x):=\{\xx(i)\mid i\in\Z\}\subset\R$.
	
	Recall the notion of a standard arrow diagram from \cref{sec:aff_perm_backgr}. More generally, to each $f\in\Saffn$ and $x\in\Ln$ one can associate an \emph{arrow diagram} $\Diag(x)$ obtained by drawing an arrow $(\xx(i),1)\to (\xx({f(i)}),0)$ for each $i\in\Z$. For example, the standard arrow diagram of $f$ is just $\Diag(x)$ for $x=\frac1n\id$, where $\id\in\Saffon$ is the identity map.
	
	We say that $x\in\Ln$ is \emph{generic} if $\xx(i)\neq \xx(j)$ for all $i\neq j\in\Z$. We denote by $\Lng$ the set of generic elements of $\Ln$. The \emph{cutoff point} for $x\in\Lng$ is the midpoint of the interval of all $c\in\R\setminus\Im(x)$ such that
	\begin{equation}\label{eq:cutoff_dfn}
		\#\{i\leq0\mid x_i>c\}=\#\{i\geq1\mid x_i<c\}.
	\end{equation}
	For $x\in\Lng$, we let $g_x$ be the affine permutation in $\Saffon$ such that for all $i,j\in\Z$, $g_x(i)<g_x(j)$ if and only if $\xx(i)<\xx(j)$. Explicitly, if $c\in\R$ is the cutoff point for $x$ and $i_1,i_2,\dots,i_n\in\Z$ are such that $c<x_{i_1}<x_{i_2}<\cdots<x_{i_n}<c+1$ then we have $g_x^{-1}=[i_1,i_2,\dots,i_n]$ in window notation; cf. \cref{ex:Diag_f} below.

	For $f\in\Saffn$ and $x\in\Lng$, the arrow diagram $\Diag(x)$ is topologically equivalent to the standard arrow diagram of $g_xfg_x^{-1}$. That is, we have an order-preserving bijection $\phi_x:=x\circ g_x^{-1}: \Z\to\Im(x)$ such that $(i,j)\in\Inv(g_xfg_x^{-1})$ if and only if the arrows starting at $(\phi_x(i),1)$ and $(\phi_x(j),1)$ cross in $\Diag(x)$.
	
	\begin{figure}
		\def\twd{0.45\textwidth}
		\begin{tabular}{ccc}
			\includegraphics[width=\twd]{figures/arrow_diagram2}
			&
			\qquad &
			\includegraphics[width=\twd]{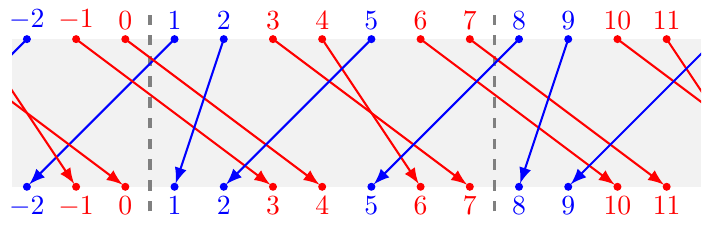}
			\\
			(a) Arrow diagram $\Diag(x)$.& & (b) Standard arrow diagram of $g_xfg_x^{-1}$. 
		\end{tabular}
		\caption{\label{fig:arr_diag2} The arrow diagram $\Diag(x)$ is topologically equivalent to the standard arrow diagram of $g_xfg_x^{-1}$; see \cref{ex:Diag_f}.}
	\end{figure}

	\begin{example}\label{ex:Diag_f}
		Let $f=[7,-1,2,5,8,3,11]$ be the affine permutation shown in \figref{fig:arr_diag}(a). An example of the arrow diagram $\Diag(x)$ for some $x\in\Lng$ is shown in \figref{fig:arr_diag}(b). We have $g_x^{-1}=[2,3,1,4,6,5,7]$ in window notation, which is obtained by reading off the labels between the two vertical dashed lines. These dashed lines are located at positions $c$ and $c+1$, where $c$ is the cutoff point of $x$. We find $g_x=s_5s_2s_1\in\Saffon$, and thus $g_xfg_x^{-1}=[-2, 1, 7, 6, 2, 10,11]$ in window notation. The standard arrow diagram of $g_xfg_x^{-1}$ is shown in \figref{fig:arr_diag2}(b). Comparing it with $\Diag(x)$ shown in \figref{fig:arr_diag2}(a), we see that indeed the two arrow diagrams are topologically equivalent (modulo a relabeling of the points given by the map $\phi_x$).
	\end{example}

	We think of an arrow diagram $\Diag(x)$ for $(f,x)\in\Saffn\times\Lng$ as a ``geometric realization'' of the affine permutation $g_xfg_x^{-1}$, and extend our definitions and notation to this case. For example, we denote by $\ellf(x):=\ell(g_xfg_x^{-1})$ the number of crossings in $\Diag(x)$ modulo the shift $(i,j)\mapsto (i+n,j+n)$. For $f\in\Saffn$ and $x,x'\in\Lng$, write $\Diag(x)\to\Diag(x')$ if $g_xfg_x^{-1}\to g_{x'}fg_{x'}^{-1}$, and $\Diag(x)\approx \Diag(x')$ if $g_xfg_x^{-1}\approx g_{x'}fg_{x'}^{-1}$. We say that $\Diag(x)$ is \emph{c-reduced} if so is $g_xfg_x^{-1}$.

	We say that $x$ is \emph{almost generic} if there exist $(i_0,j_0)\in\Z^2$ such that  for all $i\neq j$, we have $\xx(i)\neq \xx(j)$ unless $\{i,j\}=\{i_0+dn,j_0+dn\}$ for some $d\in\Z$. Thus, $\Diag(x)\to\Diag(x')$ if there exists a continuous curve $\xt\in\Ln$, $t\in[0,1]$, such that $\xtt_0=x$, $\xtt_1=x'$, $\xt$ is almost generic for $t$ in some finite set $\AGSET$ and generic for $t\in[0,1]\setminus \AGSET$, 
	and $\ellf(\xt)$ is a weakly decreasing function on $[0,1]\setminus \AGSET$.

	\subsection{\texorpdfstring{$\eps$}{Epsilon}-straight arrow diagrams}\label{sec:eps_straight}
	Fix $f\in\Saffn$. Recall that for $\Cycle\in\CLF$ and $i\in\Cycle$, we write $\slpfi:=\slpfC$. 
	For $\eps>0$ and $x\in\Ln$, we say that the arrow diagram $\Diag(x)$ is \emph{$\eps$-straight} if for all $i\in\Z$, $\xx({f(i)})$ is $\eps$-close to $\xx(i)+\slpfi$. For example, the arrow diagram $\Diag(x)$ shown in \figref{fig:arr_diag}(b) is $\eps$-straight for some $0<\eps<0.15$.
	
	Denote by $\Stre(f)$ the set of $\eps$-straight elements in $\Ln$:
	\begin{equation*}%
		\Stre(f):=\{x\in\Ln: |\xx({f(i)})-(\xx(i)+\slpfi)|\leq\eps\text{ for all $i\in\Z$}\}.
	\end{equation*}
	We set $\Streg(f):=\Stre(f)\cap \Lng$.

	The following result is an analog of~\cite[Lemma~6.8(1)]{Mar}; see also~\cite[Lemma~5.4]{Mar18} and~\cite[Proposition~3.4]{Mar14}. See \cref{rmk:alcove} below for the relation between our results and those of Marquis.
	\begin{proposition}\label{prop:ODE}
		For any $f\in\Saffn$, $x\in\Lng$, and $\eps>0$, there exists $y\in\Streg(f)$ such that $\Diag(x)\to\Diag(y)$.
	\end{proposition}
	\begin{example}
		The diagram in \figref{fig:arr_diag}(a) can be continuously deformed into the diagram in \figref{fig:arr_diag}(b). During the deformation, the point labeled $1$ passes to the right through the points labeled $2,3$ while the point labeled $5$ passes to the right through the point labeled~$6$. The resulting sequence of swaps is recorded in the reduced word for $g_x=s_5s_2s_1$; cf. \cref{ex:Diag_f}.
	\end{example}
	\begin{proof}
		We will find a smooth curve $\xt,t\in\R_{\geq0}$ in $\Ln$ such that $\xtt_0=x$ and such that we can take $y:=\xt$ for $t$ sufficiently large. The curve will be defined via the following linear system of first order ordinary differential equations (ODEs): %
		\begin{equation}\label{eq:ODE}
			\partial \xtx(i)/\partial t=\xtx({f(i)})-\xtx(i),\quad\text{for all $i\in\Z$.}
		\end{equation}
		Rewriting each $\xtx(i)$ in terms of $(\xtx(j))_{j\in[n]}$, we obtain an $n\times n$ inhomogeneous linear system of ODEs. It splits into independent systems for each cycle of $f$.
		
		Fix a single cycle $\Cycle$ of $f$, and let $m:=\nopf(\Cycle)$. We have an $m\times m$ system of ODEs of the form $\partial z(t)/\partial t=Az(t)+b$, for a constant $m\times m$ matrix $A$ and a constant vector $b\in\R^m$. Let $w:=\fb|_\Cycle\in S_m$ be the permutation obtained by taking $f|_\Cycle$ modulo $m$. The permutation matrix $P_w$ of $w$ has eigenvalues $e^{2\pi i r/m}$ for $r=0,1,\dots, m-1$. We have $A=P_w-I_m$, where $I_m$ is an $m\times m$ identity matrix. Thus, the eigenvalues of $A$ are
		$\la_r:=e^{2\pi i r/m}-1$ for $r=0,1,\dots,m-1$. (In particular, they are all distinct and have nonpositive real part.) A general solution to the homogeneous system $\partial z(t)/\partial t=Az(t)$ is then a linear combination of vector-valued functions of the form $\exp(\la_r t)z_r$, where $z_r$ is the  eigenvector of $A$ corresponding to~$\la_r$.
		
		One eigenvalue of $A$ is $\la_0=0$, and the corresponding eigenvector is $z_0:=(1,1,\dots,1)^T$. The vector $b$ is a $0,1$-vector with $1$'s in positions corresponding to $i\in[n]\cap \Cycle$ such that $f(i)>n$. In particular, the sum of coordinates of $b$ is $\kopf(\Cycle)$, and thus $\slpfC z_0-b$ belongs to the image of $A$. Letting $\tilde z_0$ be one of its preimages under $A$, we see that $z(t)=\slpfC t z_0 -\tilde z_0$ is a solution to the inhomogeneous system $\partial z(t)/\partial t=Az(t)+b$. Thus, an arbitrary solution differs from it by a linear combination of the functions $\exp(\la_r t)z_r$, each of which is constant (for $r=0$) or decays exponentially (for $r\neq0$).

		It follows that for $t$ large enough and $i\in\Z$, we have $\xtx(i)=\slpfi t+o(t)$, and $\partial \xtx(i)/\partial t=\slpfi+o(1)$. By~\eqref{eq:ODE}, we get $\xt\in\Stre(f)$ for all $t$ sufficiently large. It is also clear that for $t$ outside a discrete set, we have $\xt\in\Streg(f)$.
		
		Since $x=\xtt_0$ was generic, we can change it slightly so that each point $\xt$ is almost generic for $t$ in some discrete set $\AGSET$ and generic for $t\in[0,\infty)\setminus \AGSET$. We claim that $\ellf(\xt)$ is weakly decreasing for $t\in[0,\infty)\setminus\AGSET$. Indeed, let $t_0\in\AGSET$ be such that $\xtzx(i)=\xtzx(j)$ for some $i,j\in\Z$. Since $\xtz$ is almost generic, we have $\xtzx({f(i)})\neq \xtzx({f(j)})$.\footnote{This statement is true unless $f(i)=i+kn$ and $f(j)=j+kn$ for some $k$. But in that case, we have $\xx(i)\neq \xx(j)$ (because $x$ was generic) and $\xtzx(i)=i+kt_0$, $\xtzx(j)=j+kt_0$, so $\xtzx(i)\neq \xtzx(j)$ for all $t_0\geq0$.} Thus, $\partial \xtx(i)/\partial t\neq \partial \xtx(j)/\partial t$ at $t=t_0$. Suppose that $\partial \xtx(i)/\partial t> \partial \xtx(j)/\partial t$ at $t=t_0$, so $\xtzx({f(i)})>\xtzx({f(j)})$. Then $\xtxt_{t^-}(i)<\xtxt_{t^-}(j)$ and $\xtxt_{t^+}(i)>\xtxt_{t^+}(j)$ for some $t^-<t_0<t^+$ very close to $t_0$. We still have $\xtxt_{t^-}({f(i)})>\xtxt_{t^-}({f(j)})$. Thus, the arrows starting at $\xtxt_{t^-}(i)$ and $\xtxt_{t^-}(j)$ form a crossing in $\Diag(\xtt_{t^-})$ but do not form a crossing in $\Diag(\xtt_{t^+})$. Therefore $\ellf(\xtt_{t^-})\geq \ellf(\xtt_{t^+})$.
	\end{proof}

	\begin{remark}\label{rmk:alcove}
		Our constructions can be translated into the well-studied geometric setup as we now explain. The group $\Saffon$ acts simply transitively on the set $\Chrs$ of chambers of an infinite hyperplane arrangement $\{x_i=x_j+k\mid i\neq j\in [n], k\in\Z\}$ in $\R^{n}/\<(1,1,\dots,1)\>$. Choosing a distinguished fundamental chamber $C_0$, the map $g\mapsto gC_0$ yields a bijection $\Saffon\xrasim \Chrs$. Identifying $\Ln\xrasim \R^{n}/\<(1,1,\dots,1)\>$ by a linear isomorphism sending $x\mapsto (\xx(1),\dots,\xx(n))$, the $\Saffn$-action on $\Ln$ coincides with its action on $\R^{n}/\<(1,1,\dots,1)\>$. For $g\in\Saffon$, the point $\frac1n g$ gets mapped to the center of the corresponding chamber $gC_0$. An element $x\in\Ln$ is generic if and only if it belongs to the interior of a chamber, and almost generic if and only if it belongs to the interior of a facet of a chamber. The set $\Stre(f)$ for $\eps=0$ equals the set denoted $\Min(f)$  in~\cite{Mar}. The map sending $f$ to the tuple $(f|_{\slp})_{\slp\in\SLF}$ of its restrictions is the map denoted $\pi_{\Sigma^\eta}$ in~\cite{Mar} (whose image is an element of finite order; cf. \cref{lemma:fin}). Our proof strategy may be considered an adaptation of~\cite[Proof of Proposition~6.20]{Mar}: given an arbitrary chamber $C$, construct a walk from $C$ to a chamber intersecting $\Min(f)$, and then use the projection $\pi_{\Sigma^\eta}$ to obtain an element of finite order. The notion of a modular invariant was inspired by~\cite[Part $(A_\ell^{(1)})$ of Theorem~10.12]{Mar}. 
	\end{remark}

	\begin{remark}
		One key point that allows for a significant simplification in our approach in type A (compared to the approach of~\cite{Mar} for arbitrary Coxeter groups) is a new proof of \cref{prop:ODE} using ODEs. We hope that this argument can be of independent interest. It appears to generalize to affine Weyl groups but not to arbitrary Coxeter groups.
	\end{remark}

	\subsection{Vector configurations and conjugacy} We return to \cref{prob:c_eq}. Our first goal is to describe $\Saffon$-conjugacy classes in $\Saffn$. 
	
	Let $f\in\Saffn$. For $\slp\in\SLF$ and $\Cycle\in\CLF$, set
	\begin{equation*}%
		\ef(\slp):=(\nopf(\slp),\kopf(\slp)) \quad\text{and}\quad \ef(\Cycle):=(\nopf(\Cycle),\kopf(\Cycle)).
	\end{equation*}
	Clearly, $\ef(\slp)=\sum_{\Cycle\in\CLFS} \ef(\Cycle)$ is a sum of collinear vectors, and their integer lengths are given by $\ilen|\ef(\Cycle)|=\dopf(\Cycle)$, so $\ilen|\ef(\slp)|=\dopf(\slp)$. 
	We let $\VCF:=(\ef(\slp))_{\slp\in\SLF}$ be the \emph{vector configuration} associated to $f$. By analogy with \cref{dfn:decor}, we call $\VCFD:=(\VCF,\bflaf)$ the \emph{weakly decorated vector configuration} associated to $f$, where $\bflaf$ was introduced in \cref{dfn:bflaf}.
	
	\begin{proposition}\label{prop:conjugate_VCFD}
		Let $f,f'\in\Saffn$. Then $f$ is $\Saffon$-conjugate to $f'$ if and only if $\VCFD=\VCFPD$.
	\end{proposition}
	\begin{proof}
		Since $\VCFD$ depends only on the cycles of $f$ and their slopes, it is clearly invariant under conjugation, which shows the ``only if'' direction. Suppose now that $f,f'\in\Saffn$ are such that $\VCFD=\VCFPD$. Because the permutations $\fb,\fb'\in\Sn$ have the same cycle type, they are conjugate in $\Sn$. We may therefore apply $\Sn$-conjugation to $f'$ (permuting the cycles along the way) to obtain an element $f''$ such that $\fb=\fb''$ (in particular, $f$ and $f''$ have the same sets of cycles), and such that for each cycle $\Cycle$ of $f$, we have $\nopf(\Cycle)=\nopx_{f''}(\Cycle)$ and $\kopf(\Cycle)=\kopx_{f''}(\Cycle)$. Let $t_{e_i}\in\Saffn$ be the affine permutation sending $i\mapsto i+n$ and $j\mapsto j$ for $j\not\equiv i\pmod n$, called a \emph{translation element}. Thus, $t_{e_i-e_j}:=t_{e_i}t_{e_j}^{-1}$ belongs to $\Saffon$, and we see that $f$ can be obtained from $f''$ via conjugations by such elements $t_{e_i-e_j}$ for $i,j$ belonging to the same cycle of $f$.
	\end{proof}

	\subsection{A characterization of minimal-length elements}	
	Our next goal is to give an explicit characterization of c-reduced affine permutations; see \cref{cor:c_red_charact}.

	Given two subsets $A,B\subset\R^2$, define their \emph{Minkowski sum} by $A+B:=\{a+b\mid a\in A,\ b\in B\}$. Given a vector configuration $\VC=\{\edge_1,\edge_2,\dots,\edge_m\}\subset\Z^2$, the associated \emph{zonotope} $\Zon(\VC)$ is the convex polygon in $\R^2$ obtained as the Minkowski sum of line segments
	\begin{equation*}%
		\Zon(\VC):=[0,\edge_1]+[0,\edge_2]+\cdots+[0,\edge_m].
	\end{equation*}
	For $\edge_1,\edge_2\in\R^2$, recall that $\det(\edge_1,\edge_2)$ is the determinant of the $2\times2$ matrix with columns $\edge_1,\edge_2$. The following formula for the area of $\Zon(\VC)$ is well known~\cite{McMullen}:
	\begin{equation*}%
		\Area(\Zon(\VC))=\sum_{1\leq i<j\leq m} |\det(\edge_i,\edge_j)|.
	\end{equation*}
	
	Recall the notion of $\excess{\bfla}$ from \cref{dfn:exc}.

	\begin{lemma}\label{lemma:c_red_ell}
		Let $f\in\Saffn$. Then $f$ is c-reduced if and only if
		\begin{equation}\label{eq:c_red_ell}
			\ell(f)=\Area(\Zon(\VC))+\excess{\bfla}, \quad\text{where $\VCFD=(\VC,\bfla)$.}
		\end{equation}
	\end{lemma}
	In the proof of the lemma, we will count inversions $(j,j')\in\Inv(f)$ according to the cycles containing $j$ and $j'$.
	
	\begin{definition}
		Given two cycles $\Cycle,\Cycle'\in\CLF$, their \emph{ordered crossing number} is defined as
		\begin{equation*}%
			\xing(\Cycle,\Cycle'):=\#\{(j,j')\in \Inv(f)\mid j\in[n]\cap\Cycle\text{ and }j'\in\Cycle'\}.
		\end{equation*}
	\end{definition}
	\noindent Thus, we have $\sum_{\Cycle,\Cycle'\in\CLF} \xing(\Cycle,\Cycle') = \ell(f)$.
	
	\begin{proof}[Proof of \cref{lemma:c_red_ell}]
		Denote the right-hand side of~\eqref{eq:c_red_ell} by $\ell(\VCFD)$. First, we show that for any $f\in\Saffn$, we have $\ell(f)\geq\ell(\VCFD)$. Observe that if $g\in\Saffkn$ has a single cycle then $\ell(g)\geq \dop(g)-1$ (where $\dop(g)=\gcd(k,n)$), because the map $\fkn=\Lan^k$ has $\gcd(k,n)$ cycles, and for each $i\in[n]$, $s_ig$ has either one more or one less cycle than $g$. Thus, each cycle $\Cycle$ of $f$ contributes at least $\dopf(\Cycle)-1$ to $\ell(f)$:
		\begin{equation}\label{eq:xing_geq_CC}
			\xing(\Cycle,\Cycle)\geq \dopf(\Cycle)-1.
		\end{equation}
		It follows that for each $\slp\in\SLF$, we have
		\begin{equation*}%
			\sum_{\Cycle\in\CLFS} \xing(\Cycle,\Cycle) \geq \excess{\lafs}.
		\end{equation*}
		
		To each cycle $\Cycle$ we can associate a piecewise-linear curve $\Path(\Cycle)$ in $\R^2$ obtained by choosing some $i\in\Cycle$ and joining the points $p_d:=\left(d,\frac1n f^d(i)\right)$ for $d=0,1,\dots,\nopf(\Cycle)$; cf.~\cite[Section~4]{GL_cat_combin}. We have $p_0=(0,\frac in)$ and $p_{\nopf(\Cycle)}=(\nopf(\Cycle),\frac in+\kopf(\Cycle))$, thus $\Path(\Cycle)$ gives rise to a closed curve on $\T=\R^2/\Z^2$ with homology $\ef(\Cycle)=(\nopf(\Cycle),\kopf(\Cycle))$. It is well known that given integers $n',k',n'',k''$ with $k'/n'>k''/n''$, a curve in $\T$ with homology $(n',k')$ intersects a curve with homology $(n'',k'')$ from below at least 
		$\left|\det \begin{pmatrix}
			n' & k'\\ n'' & k''
		\end{pmatrix}\right|$ times.  Thus, given cycles $\Cycle\neq\Cycle'$, we have
		\begin{equation}\label{eq:xing_geq_CC'}
			\xing(\Cycle,\Cycle')\geq 
			\begin{cases}
				0, &\text{if $\slpfC\leq \slpfx(\Cycle')$;}\\
				|\det(\ef(\Cycle),\ef(\Cycle'))|,&\text{otherwise.}
			\end{cases}
		\end{equation}
		We have shown that $\ell(f)\geq\ell(\VCFD)$.
		Conversely, consider a weakly decorated vector configuration $\VCD=(\VC,\bfla)$.  By \cref{prop:conjugate_VCFD}, $\O:=\{f\in\Saffn\mid \VCFD=\VCD\}$ is an $\Saffon$-conjugacy class. 
		By \cref{cor:HN}, $f\in\O$ is c-reduced if and only if $f\in\Omin$. We have shown above that for any $f\in\O$, $\ell(f)\geq \ell(\VCD)$. It remains to construct $g\in\O$ such that $\ell(g)=\ell(\VCD)$.  Such an affine permutation will be constructed in \cref{sec:eps_straight_construct}.
	\end{proof}

	\begin{corollary}\label{cor:c_red_charact}
		Let $f\in\Saffn$. Then $f$ is c-reduced if and only if all of the following conditions are satisfied.
		\begin{enumerate}
			\item For each $\Cycle\in\CLF$, $\xing(\Cycle,\Cycle)=\dopf(\Cycle)-1$.
			\item For each $\Cycle\neq\Cycle'$ in $\CLF$, we have
			\begin{equation*}%
				\xing(\Cycle,\Cycle')=
				\begin{cases}
					0, &\text{if $\slpfC\leq \slpfx(\Cycle')$;}\\
					|\det(\ef(\Cycle),\ef(\Cycle'))|,&\text{otherwise.}
				\end{cases}
			\end{equation*}
		\end{enumerate}
	\end{corollary}
	\begin{proof}
		We have lower bounds on $\xing(\Cycle,\Cycle)$ and $\xing(\Cycle,\Cycle')$ given by~\eqref{eq:xing_geq_CC}--\eqref{eq:xing_geq_CC'}. Moreover, we showed in \cref{lemma:c_red_ell} that $f$ is c-reduced if and only if all of these inequalities are equalities. 
	\end{proof}
	\begin{remark}
		\cref{cor:c_red_charact} was obtained jointly with Thomas Lam during the development of~\cite{GL_cat_combin}.
	\end{remark}

	\begin{corollary}\label{cor:restr_c_reduced}
		If $f\in\Saffn$ is c-reduced and $\Cycle\subset\Z$ is $f$-closed then $f|_\Cycle$ is c-reduced.
	\end{corollary}
	
	\section{The structure of c-equivalence classes}\label{sec:c-equiv-structure}
	The goal of this section is to give a complete description of c-equivalence classes of c-reduced affine permutations; see \cref{thm:c-equivalent}.
	\subsection{Cyclic compositions}
	Let $f\in\Saffn$ be c-reduced. Fix a slope $\slp\in\SLF$. By \cref{cor:c_red_charact}, we have $\xing(\Cycle,\Cycle')=0$ for all $\Cycle\neq\Cycle'$ in $\CLFS$. We thus get a natural cyclic order on the set $\CLFS$ induced by the cyclic order on $[n]\cong\Z/n\Z$. Recall that $\sum_{\Cycle\in\CLFS} \dopf(\Cycle)=\dopf(\slp)$. In other words, the cyclic order on $\CLFS$ yields a cyclic composition $\ccfs$ of $\dopf(\slp)$.
	Letting $\bfccf:=(\ccfs)_{\slp\in\SLF}$, we consider the \emph{strongly decorated vector configuration} $\VCFDD:=(\VCF,\bfccf)$.
	
	\begin{lemma}\label{lemma:c-red=>VCF}
		Let $f,f'\in\Saffn$ be c-reduced. If $f\approx f'$ then $\VCFDD=\VCFPDD$.
	\end{lemma}
	\begin{proof}
		By assumption, $\ell(f)=\ell(f')$. It suffices to consider the case $f'=s_ifs_i$ for some $i\in[n]$. By \cref{lemma:c_red_ell}, we only need to check that the relative order on $\CLFS$ is preserved for each slope $\slp\in\SLF$. This is clear since $f,f'$ have no crossings between different cycles of the same slope by \cref{cor:c_red_charact}.
	\end{proof}

	To give the converse to \cref{lemma:c-red=>VCF}, we need to consider \emph{modular invariants} discussed in \cref{sec:intro:minv}. Recall from \cref{dfn:clicks} that for a cyclic composition $\cc$, we have the rotation number $\rot(\cc)$, and for a family $\bfcc$ of cyclic compositions, $\clicks(\bfcc)$ is the greatest common divisor of their rotation numbers.
	Given a conjugacy class $\O$ and a strongly decorated vector configuration $\VCDDe$, let
	\begin{equation*}%
		\Omin[\VCDDe]:=\{f\in\Omin\mid \VCFDD=\VCDDe\}.
	\end{equation*}
	
	The goal of this section is to prove the following result.
	\begin{theorem}\label{thm:c-equivalent}
		Let $f\in\Saffn$ be c-reduced. Let $\O$ be the $\Saffon$-conjugacy class of $f$. Then $\Omin[\VCFDD]$ is a union of $\clicks(\bfccf)$-many c-equivalence classes. Moreover, for any two c-reduced $f,f'\in\Omin$, we have %
		\begin{equation}\label{eq:c-equivalent}
			f\approx f' \quad\Longleftrightarrow \quad
			(\VCFDD,\minv(f))=(\VCFPDD,\minv(f')),
		\end{equation}
		where $\minv(f)\in \Z/\clicks(\bfccf)\Z$ is the \emph{modular invariant} defined in~\eqref{eq:minv_f_dfn}.
	\end{theorem}
	\begin{remark}\label{rmk:Mar}
		Alternatively, \cref{thm:c-equivalent} may be deduced from the recently updated version of \cite[Theorem~B]{Mar}.
	\end{remark}

	\subsection{Constructing \texorpdfstring{$\eps$}{epsilon}-straight diagrams explicitly}\label{sec:eps_straight_construct}
	
	\begin{figure}
		\begin{tabular}{cc}
			\includegraphics[width=0.15\textwidth]{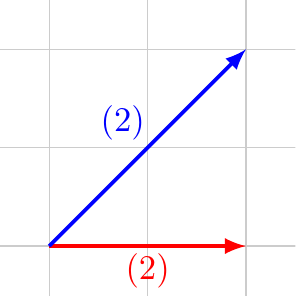}
			&
			\includegraphics[width=0.4\textwidth]{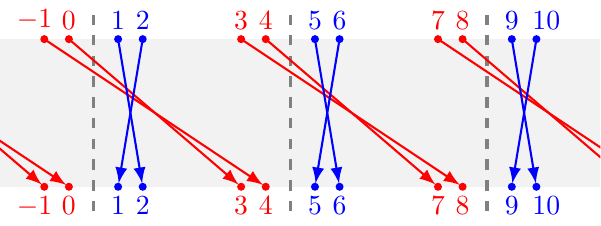}
			\\
			(a) $\VCDDe$. &(b) $\DP(\VCDDe)$. 
			\\
			\includegraphics[width=0.22\textwidth]{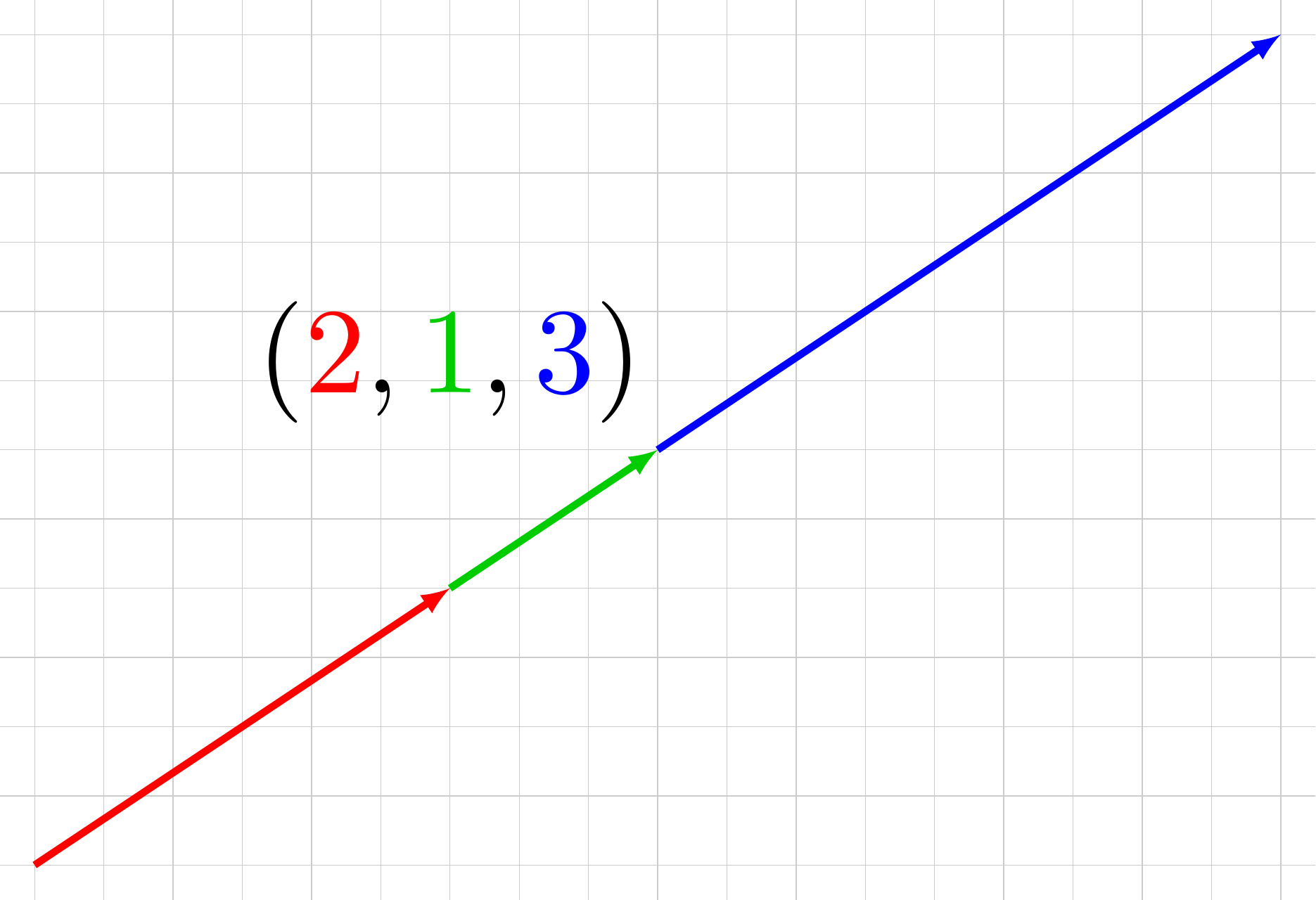}
			&
			\includegraphics[width=0.6\textwidth]{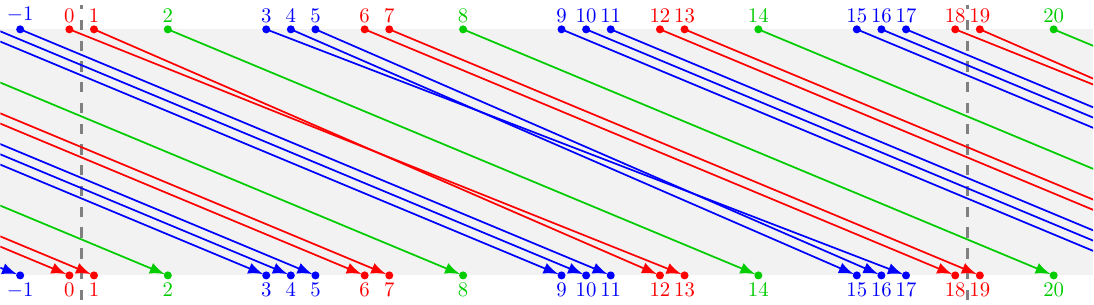}
			\\
			(c) $\VCDDe$. &(d) $\DP(\VCDDe)$. 
		\end{tabular}
		\caption{\label{fig:arrowdiagram} A strongly decorated vector configuration (left) and an $\eps$-straight arrow diagram (right); see \cref{sec:eps_straight_construct}.}
	\end{figure}
	
	Let $\VCDDe=(\VC,\bfcc)$ be a strongly decorated vector configuration and fix a small $\epsilon>0$. Our goal is to construct an $\eps$-straight arrow diagram $\DP(\VCDDe)=\Diagx_g(x)$ for some $x\in\Lng$ and c-reduced $g\in\Saffn$ with $\VCGDD=\VCDDe$. We start with an example and then proceed with a formal description.

	\begin{example}
		Let $\VCDDe=(\VC,\bfcc)$ denote the strongly decorated vector configuration shown in \figref{fig:arrowdiagram}(a). Thus, the vectors in $\VC$ are $\eer = (2,0)$, $\eeb=(2,2)$, and $\cc^{\eer}=\cc^{\eeb}=(2)$. An $\eps$-straight arrow diagram $\DP(\VCDDe)$ is shown in \figref{fig:arrowdiagram}(b).
		On the other hand, if $\VCDDe=(\VC,\bfcc)$ is the strongly decorated vector configuration shown in \figref{fig:arrowdiagram}(c), then $\VC$ consists of a single vector $\e=(18,12)$ decorated by a cyclic composition $\cc^\e=(2,1,3)$. The associated $\eps$-straight arrow diagram $\DP(\VCDDe)$ is constructed in \figref{fig:arrowdiagram}(d).
	\end{example}

	For a vector $\e=(a,b)\in\Z^2$, we denote $\nop(\e):=a$ and $\kop(\e):=b$. For $\e \in \VC$, let $\slp(\e)=\kop(\e)/\nop(\e)$ denote its slope. Assume that $\nop(\e) > 0$ for all $\e \in \VC$. Let $\bfcc=(\cc^\e)_{\e \in \VC}$ and $\cc^\e=(\cc^\e_1,\dots,\cc^\e_{m_\e})$. Consider the circle $\R/\Z$ and choose a collection of \emph{starting points} $\seedpts=(\seedpt^\e_i)_{\e \in E, i \in [m_\e]}$, where $\seedpt^\e_i \in \R/\Z$. Let $\Seedpt^\e_i:=\{\seedpt^\e_i+r\slp(\e) \mid r\in\Z\}\subset\R/\Z$ be the set containing $\seedpt^\e_i$ and consisting of $\nop(\e)/\ilen|\e|$ equally spaced points on a circle. We choose $\seedpts$ so that we additionally have:
	\begin{enumerate}
		\item $\dist_{\R/\Z}(\Seedpt^\e_i,\Seedpt^{\e'}_{i'})>\eps$ for all $(\e,i) \neq (\e',i')$; and
		\item the points $(\seedpt^\e_1,\seedpt^\e_2,\dots,\seedpt^\e_{m_\e})$ are cyclically ordered in $\R/\Z$.
	\end{enumerate}
	Now, for each fixed $\e \in \VC$ and $i \in [m_e]$, we construct an arrow diagram $\DiagOp^\e_i$. Let $\Seedptl^\e_i\subset\R$ be the preimage of $\Seedpt^\e_i$ under the projection $\R\to\R/\Z$, and choose $\seedptp\in\Seedptl^\e_i$. Set $d:=\alpha^\e_i$. For each $r \in [d]$, set $\seedptp_r:=\seedptp+\frac {r\eps}{d}$. We refer to the points $(\seedptp_r)_{r\in[d]}$ as the \emph{block} associated to $\seedptp$, and denote by $\Seedptl'_{\e,i}:=\Seedptl^\e_i+\frac{\eps}d[d]$ the set of points in all such blocks. Let $\barseedptp\in\Seedpt^\e_i$ be the image of $\seedptp$ in $\R/\Z$. If $\barseedptp\not\equiv \seedpt^\e_i\pmod{\Z}$ then we draw an arrow $(\seedptp_r,1)\to(\seedptp_r+\slp(\e),0)$ for each $r\in[d]$. Otherwise, we draw an arrow $(\seedptp_r,1)\to(\seedptp_{\sigma(r)}+\slp(\e),0)$ for each $r\in[d]$, where $\sigma=(1\,2\,\dots\,d)\in S_d$ is a $d$-cycle. The resulting arrow diagram is denoted $\DiagOp^\e_i$.
	
	Let $\SP:=\bigsqcup_{\e \in \VC, i \in [m_e]} \Seedptl'_{\e,i}\subset\R$ be the resulting set of points, and let $\DP(\VCDDe):=\bigcup^\e_{i \in [m_e]} \DiagOp^\e_i$ be the corresponding arrow diagram. 
	Let $x:\Z\to\SP$ be an order-preserving map. Then there exists a unique affine permutation $g\in\Saffn$ such that $\DP(\VCDDe)=\Diagx_g(x)$. By construction, $\VCGDD{}=\VCDDe$ and $\ell(g)=\ell(\VCDDe)$, which completes the proof of \cref{lemma:c_red_ell}. By~\cref{lemma:c_red_ell}, $g$ is c-reduced.

	\subsection{Affine permutations of constant slope}
	\begin{definition}
		Let $f\in\Saffn$ and $\slp\in\Q$. We say that $f$ is \emph{of constant slope $\slp$} if $\SLF=\{\slp\}$. (That is, if all cycles of $f$ are of the same slope $\slp$.)
	\end{definition}
	\noindent It is clear that if $f\in\Saffkn$ is of constant slope $\slp$ then we must have $\slp=k/n$.
	
	Recall that $\Sextn$ is a quotient of $\Saffn$ by $\Lan^n$. We denote the quotient map $\Saffn\to\Sextn$ by $f\mapsto \fhat$.
	\begin{lemma}\label{lemma:fin}
		Let $f\in\Saffn$. Then $\fhat\in\Sextn$ has finite order if and only if $f$ is of constant slope.
	\end{lemma}
	\begin{proof}
		Let $N$ be the least common multiple of $\nopf(\Cycle)$ for all $\Cycle\in\CLF$. Then $f^N$ is a translation element; that is, $f^N(i)=i+d_in$ for all $i\in\Z$, where $(d_i)_{i\in\Z}$ is some sequence of integers. Explicitly, if $i\in\Cycle$ then $d_i=N\slpfC\in\Z$. This implies the result.
	\end{proof}

	Let $f\in\Saffkn$ be c-reduced and of constant slope, and set $d:=\gcd(k,n)$. By \cref{cor:c_red_charact}, the arrows between different cycles of $f$ do not cross. Therefore, for each cycle $\Cycle\in\CLF$, we have $\Cycle=\Cycle+d$ as subsets of $\Z$. Denoting by $I_\Cycle\subset\Z/d\Z$ the image of $\Cycle$ under the map $\Z\mapsto \Z/d\Z$, we get a partition $\Ibm_f=\{I_\Cycle\mid \Cycle\in\CLF\}$ of $\Z/d\Z$ into cyclic intervals.\footnote{The case where $f$ is a single cycle requires special care. As mentioned after \cref{dfn:clicks}, we distinguish between different cyclic intervals $[j,j+d-1]$ of $\Z/d\Z$. Topologically, the standard arrow diagram of $f$ (viewed as a union of arrows) will be disconnected, and we choose $\Ibm_f:=\{[j,j+d-1]\}$ for $j\in\Z/d\Z$ such that the points $(j,1)$ and $(j-1,1)$ belong to different connected components.} It is clear that $\Ibm_f$ is invariant under c-equivalence.
	\begin{proposition}[{\cite[Proposition~A]{Mar}}]\label{prop:Mar_A}
		Let $f,f'\in\Saffkn$ be c-reduced and of constant slope. Then %
		\begin{equation*}%
			f\approx f' \quad\text{if and only if} \quad \Ibm_f=\Ibm_{f'}.
		\end{equation*}
	\end{proposition}

	We say that a cyclic composition $\cc=(\cc_1,\cc_2,\dots,\cc_m)$ is written in \emph{normal form} if the sequence $(\cc_1,\cc_2,\dots,\cc_m)$ is lexicographically maximal out of all sequences obtained by rotating $\cc$, i.e., $(\cc_r,\cc_{r+1},\dots,\cc_m,\cc_1,\dots,\cc_{r-1})$ for $r\in[m]$. As in \cref{dfn:clicks}, we associate to $\cc$ a partition $\Icc=(I_1,I_2,\dots,I_m)$ of $\Z/d\Z$ (where $d=\cc_1+\cc_2+\cdots+\cc_m$) into cyclic intervals given by $I_1=[1,\cc_1]$, $I_2=[\cc_1+1,\cc_1+\cc_2]$, etc.
	
	Note that if $\cc=\ccfs$ then we have $d=\cc_1+\cc_2+\cdots+\cc_m=\gcd(k,n)$, and therefore we have two partitions $\Icc$ and $\Ibm_f$ of $\Z/d\Z$ into cyclic intervals. These partitions are related by a rotation $\rotsig^r$ of $\Z/d\Z$ for some $r$; however, this rotation is only defined up to a symmetry of $\Icc$, i.e., up to $\rotsig^{\rot(\cc)}$. (Here, $\rot(\cc)$ divides $d$.) %
	\begin{definition}\label{dfn:minv_const_slp}
		Let $f\in\Saffkn$ be c-reduced of constant slope $\slp=k/n$, and let $\cc:=\ccfs$ be written in normal form. The \emph{modular invariant} $\minv(f)\in\Z/\rot(\cc)\Z$ is the unique element such that $\rotsig^{\minv(f)}(\Icc)=\Ibm_f$.
	\end{definition}
	
	\begin{corollary}\label{cor:c_red_const_slp}
		Let $f,f'\in\Saffkn$ be c-reduced and of constant slope $\slp=k/n$. Then %
		\begin{equation*}%
			f\approx f' \quad\text{if and only if} \quad (\ccfs,\minv(f)) = (\ccfps,\minv(f')).
		\end{equation*}
	\end{corollary}
	\begin{proof}
		The $\Longrightarrow$ direction is clear since both $\ccfs$ and $\minv(f)$ are invariant under c-equivalence. Conversely, having $\ccfs=\ccfps$ implies that $\Ibm_f$ and $\Ibm_{f'}$ coincide up to cyclic shift, and $\minv(f)=\minv(f')$ guarantees that $\Ibm_f=\Ibm_{f'}$. The result then follows from \cref{prop:Mar_A}.
	\end{proof}

	\subsection{Finishing the proof}
	For $f\in\Saffn$ and $\slp\in\SLF$, let $\fs:=f|_{\slupp}$. Thus, $\fs$ has constant slope $\slp$. If in addition $f$ is c-reduced then by \cref{cor:restr_c_reduced}, so is $\fs$. In this case, recall from \cref{dfn:minv_const_slp} that the modular invariant $\minv(\fs)$ is an element of $\Z/\rot(\ccfs)\Z$. By~\eqref{eq:clicks_dfn}, $\clicks(\bfccf)$ is defined as the greatest common divisor of the numbers $\clicks(\ccfs)$ over all $\slp\in\SLF$.
	
	\begin{definition}
		For c-reduced $f\in\Saffn$, define the \emph{modular invariant} $\minv(f)\in \Z/\clicks(\bfccf)\Z$ by
		\begin{equation}\label{eq:minv_f_dfn}
			\minv(f):=\sum_{\slp\in\SLF} \minv(\fs) \quad \mod\ \clicks(\bfccf).
		\end{equation}
	\end{definition}

	\begin{lemma}\label{lemma:c-red=>minv}
		Let $f,f'\in\Saffn$ be c-reduced. If $f\approx f'$ then $\minv(f)=\minv(f')$.
	\end{lemma}
	\begin{proof}
		Suppose that $f\xrightarrow{s_i} f'$ for some $i\in[n]$. The restrictions $f|_\slp$ and $f'|_\slp$ are c-equivalent for all $\slp\in\SLF$ (which implies the result by \cref{cor:c_red_const_slp}) unless $i=n$ and $\slpfx(0)\neq \slpfx(1)$. Suppose that we are in that case and let $\slp_0:=\slpfx(0)$, $\slp_1:=\slpfx(1)$. 
		Since $\slp_0\neq\slp_1$, by the definition of $f|_{\slp_0}$ in~\eqref{eq:restrict}, we see that $f'|_{\slp_0}=\rotsig(f|_{\slp_0})$ and $f'|_{\slp_1}=\rotsig^{-1} (f|_{\slp_1})$. Here, $\rotsig(g)=\Lan g \Lan^{-1}$ is the rotation operator introduced in \cref{sec:aff_perm_backgr}. Thus, $\minv(f'|_{\slp_0})=\minv(f|_{\slp_0})+1$ and $\minv(f'|_{\slp_1})=\minv(f_{\slp_1})-1$, and the sum in~\eqref{eq:minv_f_dfn} remains the same.
	\end{proof}

	We will need one more tool for working with $\eps$-straight diagrams from \cref{sec:eps_straight}. Fix c-reduced $f\in\Saffn$ and small $\eps>0$. For $x\in\Streg(f)$ such that $\Diag(x)$ is c-reduced, recall from \cref{cor:c_red_charact} that $\Diag(x)$ contains no crossings between distinct cycles of the same slope.
	
	\begin{definition}
		Let $x\in\Streg(f)$ be  c-reduced and let $\abf:=(a_\Cycle)_{\Cycle\in\CLF}$ be a family of real numbers associated to the cycles of $f$. Consider a curve $\xt$, $t\geq0$, given for each $i\in\Z$ by $\xtx(i)=\xx(i)+ta_{\Cycle}$, where $\Cycle$ is the cycle containing $i$. %
		Let $T>0$ be such that for $t\in[0,T]$, $\xtx(i)\neq \xtx(j)$ for any $i\neq j$ such that $\slpfi=\slpfj$. In this case, we say that $x':=\xtt_T$ is obtained from $x=\xtt_0$ by \emph{block-shifting}.
	\end{definition}
	
	In other words, block-shifting allows us to move the collections of points $(\xx(i))_{i\in\Cycle}$ independently for each cycle $\Cycle$, subject to the condition that two cycles of the same slope never collide. It is clear that for $\eps$ sufficiently small, if $x\in\Streg(f)$ is c-reduced and $x'\in\Streg(f)$ is obtained from $x$ by block-shifting then $x'$ is c-reduced and $\Diag(x)\to\Diag(x')$.

	\begin{proof}[Proof of \cref{thm:c-equivalent}]
		The $\Longrightarrow$ direction follows from \cref{lemma:c-red=>VCF,lemma:c-red=>minv}. 
		
		For the $\Longleftarrow$ direction, let $f,f'\in\Omin$. Thus, $f'=gfg^{-1}$ for some $g\in\Saffon$. Let $x,x'\in\Streg(f)$ be obtained from $\frac1n \id, \frac1n g\in\Lng$ via \cref{prop:ODE} so that $\Diag(\frac1n\id)\to\Diag(x)$ and $\Diagx_{f'}(\frac1n\id)=\Diag(\frac1n g) \to \Diag(x')$.

		Set $h:=g_xfg_x^{-1}$ and $h':=g_{x'}fg_{x'}^{-1}$. We have $f\approx h$ and $f'\approx h'$. Since $f,f'$ are c-reduced, so are $\Diag(x),\Diag(x')$ and $h,h'$. Since $\VCDD_f=\VCDD_{f'}$, and thus $\VCDD_h=\VCDD_{h'}$, we see that the 
		partitions $\Ibm_{h|_\slp}$ and $\Ibm_{h'|_\slp}$ of $\Z/\dopf(\slp)\Z$ into cyclic intervals
		differ by rotation for all $\slp\in\SLF$. Our goal is to apply block-shifting to $\Diag(x)$ with the aim of achieving $\Ibm_{h|_\slp}=\Ibm_{h'|_\slp}$ for all $\slp\in\SLF$. To do so, consider the following operation on the partitions $(\Ibm_{h|_\slp})_{\slp\in\SLF}$:
		\begin{equation}\label{eq:rot_pm}
			\text{for some $\slp\neq\slp'$ in $\SLF$, replace $\Ibm_{h|_\slp}\mapsto \rotsig(\Ibm_{h|_\slp})$ and $\Ibm_{h|_{\slp'}}\mapsto\rotsig^{-1}(\Ibm_{h|_{\slp'}})$.}
		\end{equation}
		We first explain how to obtain~\eqref{eq:rot_pm} via block-shifting.

		Applying block-shifting to $\Diag(x)$ corresponds to applying a sequence ${h\xrightarrow{s_{i_1}} h_1\xrightarrow{s_{i_2}} \cdots \xrightarrow{s_{i_l}} h'}$ of c-equivalences. In order to control how each restriction $h|_\slp$ changes under such operations, we need to distinguish between the cases $i_j=n$ and $i_j\neq n$ as we did in the proof of \cref{lemma:c-red=>minv}. 

		Recall the notion of the cutoff point from~\eqref{eq:cutoff_dfn}. Suppose that applying block-shifting to $x$ switches the positions of adjacent points $x_j$ and $x_k$ for some $j,k$. If the cutoff point of $x$ is between $x_j+d$ and $x_k+d$ for some $d\in\Z$ then the corresponding c-equivalence corresponds to $s_n$, otherwise it corresponds to $s_i$ for $i\in [n-1]$.

		Consider slopes $\slp\neq\slp'$ in $\SLF$. We may apply block-shifting to move $\slupp$ (resp., $\sluppp$) to the right (resp., left) so that no point in $\Im(x)$ passes through the cutoff point $c$ of $x$, until $c$ is located in an interval of $\R\setminus\Im(x)$ between a point of $\slupp$ and a point of $\sluppp$. We may then shift $\slupp$ (resp., $\sluppp$) further to the right (resp., left) until these two points swap places. This corresponds to replacing $h|_\slp$ with $\rotsig(h|_\slp)$ and $h|_{\slp'}$ with $\rotsig^{-1}(h|_{\slp'})$, which results in applying~\eqref{eq:rot_pm} to $\Ibm_{h|_\slp}$ and $\Ibm_{h|_{\slp'}}$.

		Recall that for $\slp\in\SLF$, by the definition of $\rot(\ccfs)$, we have $\rotsig^{\rot(\ccfs)}(\Ibm_{h|_\slp})=\Ibm_{h|_\slp}$. 
		Let $d:=\clicks(\bfccf)=\gcd\{\ccfs\mid\slp\in\SLF\}$. Write $d=\sum_{\slp'\in\SLF} a_{\slp'} \rot(\ccfsp)$ for some integers $a_{\slp'}$. Then, for each fixed $\slp\in\SLF$, we have $(a_{\slp}\rot(\ccfs)-d)+ \sum_{\slp'\neq\slp} a_{\slp'} \rot(\ccfsp)=0$, and therefore we can use~\eqref{eq:rot_pm} to rotate each $\Ibm_{h|_{\slp'}}$ by the corresponding coefficient. The result of this operation is
		\begin{equation}\label{eq:rot_d}
			\text{replace $\Ibm_{h|_{\slp}}\mapsto \rotsig^{-d}(\Ibm_{h|_{\slp}})$ and preserve $\Ibm_{h|_{\slp'}}$ for all $\slp'\neq \slp$.}
		\end{equation}
		
		Fix $\slp\in\SLF$. Applying~\eqref{eq:rot_pm}, we can achieve $\Ibm_{h|_{\slp'}}=\Ibm_{h'|_{\slp'}}$ for all $\slp'\neq\slp$. Since $\minv(h)=\minv(h')$, we see that $\Ibm_{h|_{\slp}}$ and $\Ibm_{h'|_{\slp}}$ differ by rotation by a multiple of $d$, so applying~\eqref{eq:rot_d}, we achieve $\Ibm_{h|_{\slp}}=\Ibm_{h'|_{\slp}}$.

		By \cref{prop:Mar_A}, we have $h|_{\slp} \approx h'|_{\slp}$ for all $\slp\in\SLF$. Since $h|_{\slp}$ and $h'|_{\slp}$ are c-reduced, they have no crossings between different cycles. Thus, each c-equivalence in $h|_{\slp} \approx h'|_{\slp}$ swaps points from the same cycle. Such points are close together in $x$, and we therefore can lift these c-equivalences to $h$ so that we get $h|_{\slp} = h'|_{\slp}$ for all $\slp\in\SLF$. Replacing $h,h'$ with $\rotsig^r(h),\rotsig^r(h')$ for some $r$, we may assume that the cutoff points of $x$ and $x'$ are not $\eps$-close to any point in $\Im(x)\cup\Im(x')$. In this case, we still have $h|_{\slp} = h'|_{\slp}$ for all $\slp\in\SLF$. 
		Applying block-shifting to $x$ so that for $\slp\neq\slp'$, no point in $\slupp$ is $\eps$-close to a point in $\sluppp$, we find that $h$ and $h'$ are c-equivalent.
	\end{proof}

	\section{Relating affine permutations to bipartite graphs on a torus}\label{sec:aff_to_bip}
	The goal of this section is to apply the results of \cref{sec:affine-perm-cycl,sec:c-equiv-structure} to bipartite graphs embedded in $\T$ and to finish the proof of our main results, \cref{thm:intro:move_red,thm:intro:move_eq}.
	
	\subsection{The double affine symmetric group}
	The \textit{double affine symmetric group} $\dSaff$ is generated by
	$S \sqcup \bar{S} \sqcup \{\Lan\}$, where $S:=\{s_i \mid i \in \Z/n\Z \}$ and $\bar S := \{ s_{\overline{i}} \mid i \in \Z/n\Z \}$, subject to the relations
	\begin{align}
		s_i s_{i+1} s_{i}&=s_{i+1} s_i s_{i+1},           &      \Lan s_{i+1} &= s_{i}\Lan,          &  s_i^2&=1, & s_i s_j &=s_j s_i \quad\text{if }|i-j|>1, \nonumber\\
		s_{\overline{i}} s_{\overline{i+1}} s_{\overline{i}} &=s_{\overline{i+1}} s_{\overline{i}} s_{\overline{i+1}},       &  
		\Lan s_{\overline{i+1}} &= s_{\overline{i}}\Lan,   & 
		s_{\overline{i}}^2&=1, & s_{\overline{i}} s_{\overline{j}} &=s_{\overline{j}} s_{\overline{i}} \quad\text{if }|i-j|>1, \label{eq:dsaffnrelations}\\
		\Lan^n &=1,   &  s_i s_{\overline{j}}&=s_{\overline{j}} s_i.          &  &\nonumber
	\end{align}
	In other words, we have an isomorphism $\dSaff\cong\dSaffHuge$, where $\Lan$ acts on each copy of $\Saffon$ by conjugation. 	%
	Any element $\daffper \in \dSaff$ can be written as a product $\daffper=s_{i_1} s_{i_2} \cdots s_{i_l} \Lan^k s_{\overline{j_m}}s_{\overline{j_{m-1}}} \cdots s_{\overline{j_1}} $ for some $k \in \{0,1,\dots,n-1\}$ and $l,m \geq 0$. If $l+m$ is minimal among all such ways of writing $\daffper$ as a product, then $s_{i_1} s_{i_2} \cdots s_{i_l} \Lan^k s_{\overline{j_m}}s_{\overline{j_{m-1}}} \cdots s_{\overline{j_1}}$ is called a \textit{reduced expression} for $\daffper$, and $l+m$ is called the \textit{length} of $\daffper$ and denoted $\ell(\daffper)$. Note that 
	\begin{equation*}%
		\pera:=s_{i_1} s_{i_2} \cdots s_{i_l} \Lan^k \quad\text{and}\quad \perb:= s_{{n-j_1+1}}s_{{n-j_2+1}} \cdots s_{{n-j_m+1}} \Lan^{k}
	\end{equation*}
	are then reduced expressions for affine permutations $\pera,\perb\in\Saffn$. %
	We denote $\perab(\daffper):=(\pera,\perb)$ and call $(\pera,\perb)$ the \textit{pair of affine permutations} associated to $\daffper$. We have $\ell(\daffper)=\ell(\pera)+\ell(\perb)$. We explain the reasoning behind the formula for $\perb$ 
	in~\cref{rem:barsmap}. 
	\begin{remark}\label{rmk:perab_invertible}
		For any $k\in\Z$ and  $\pera,\perb\in\Saffkn$, there exists $\daffper\in\dSaff$ satisfying $\perab(\daffper)=(\pera,\perb)$.
	\end{remark}

	\subsection{Relating triple-crossing diagrams in $\A$ to double affine permutations}\label{sec:affine_plabic_fence}

	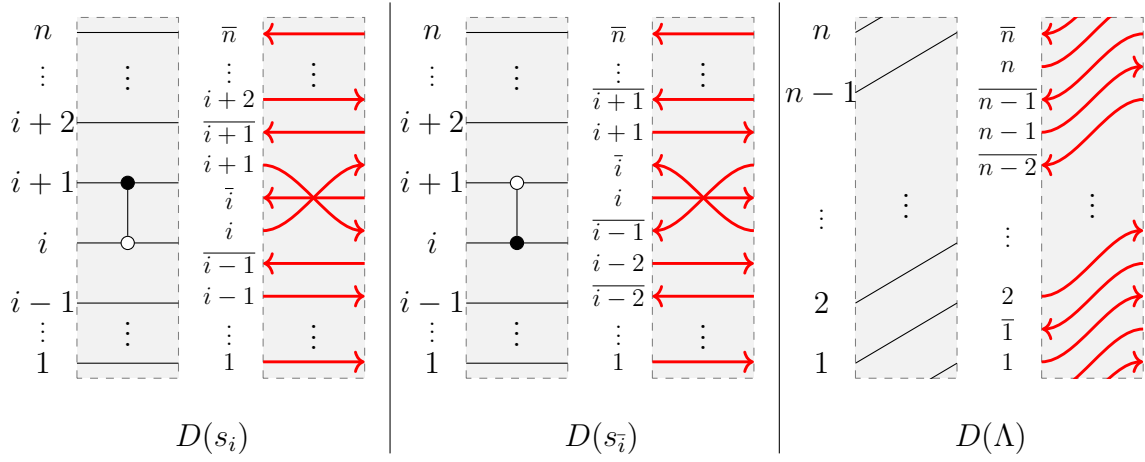
\begin{figure}
		\centering
		\def\phantomwhite(#1){\draw[black!5] (0.3,0.3)--(0.3,0.3)}
		\begin{tikzpicture}[yscale=0.4,xscale=0.45]
			\def\labsclG{1.0}
			\def\labsclD{0.8}
			
			\def\shifta{5.5}
			\def\shiftb{6.0}
			\def\vdotscl{0.8}
			\def\newlabels{
				\node[scale=\labsclG] (no) at (-1,0.5) {$i+1$}; 
				\node[scale=\labsclG] (no) at (-1,-1.5) {$i$}; 
				\node[scale=\labsclG] (no) at (-1,-3.5) {$i-1$}; 	
				\node[scale=\vdotscl] (no) at (-1,-4.25) {$\vdots$}; 
				\node[scale=\labsclG] (no) at (-1,-5.5) {$1$};
				\node[scale=\labsclG] (no) at (-1,2.5) {$i+2$}; 
				\node[scale=\vdotscl] (no) at (-1,4.25) {$\vdots$}; 
				\node[scale=\labsclG] (no) at (-1,5.5) {$n$};
			}
			
			\node[] (no) at (4,-8) {$D(s_i)$};
			\node[] (no) at (\shifta+\shiftb+4,-8) {$D(s_{\bar i})$};
			\node[] (no) at (2*\shifta+2*\shiftb+4,-8) {$D(\Lan)$};
			\draw[-] (9.25,6.5) -- (9.25,-8.5);
			\draw[-] (\shifta+\shiftb+9.25,6.5) -- (\shifta+\shiftb+9.25,-8.5);
			\begin{scope}[shift={(0,0)}] %
				\def\r{2};
				
				\fill[black!5]  (0,-6) rectangle (3,6);
				
				\draw[dashed,\dashcolor,-] (0,-6) rectangle (3,6);
				
				\draw[] (0,-5.5) -- (3,-5.5);
				\phantomwhite(1.5,-5.5);
				\draw[] (0,-3.5) -- (3,-3.5);
				\phantomwhite(1.5,-3.5);
				
				\draw[] (0,-1.5) -- (3,-1.5);
				\coordinate[wvert] (w) at (1.5,-1.5);
				
				\draw[] (0,.5) -- (3,.5);
				\coordinate[bvert] (b) at (1.5,.5);
				\phantomwhite(0,.5);
				\phantomwhite(3,.5);
				\draw (b)--(w);

				\draw[] (0,2.5) -- (3,2.5);
				\phantomwhite(1.5,2.5);
				
				\draw[] (0,5.5) -- (3,5.5);
				\phantomwhite(1.5,5.5);
				
				\node[] (no) at (1.5,-4.25) {$\vdots$}; 
				\node[] (no) at (1.5,4.25) {$\vdots$}; 
				
				\newlabels
				
			\end{scope}
			
			\begin{scope}[shift={(\shifta+\shiftb,0)}] %
				\def\r{2};
				
				\fill[black!5]  (0,-6) rectangle (3,6);
				
				\draw[dashed,\dashcolor,-] (0,-6) rectangle (3,6);
				
				\draw[] (0,-5.5) -- (3,-5.5);
				\phantomwhite(1.5,-5.5);
				\begin{scope}[yshift=-1cm]
					
					\draw[] (0,-2.5) -- (3,-2.5);
					\phantomwhite(1.5,-2.5);
					
					\draw[] (0,1.5) -- (3,1.5);
					\coordinate[wvert] (w) at (1.5,1.5);
					
					\draw[] (0,-.5) -- (3,-.5);
					\coordinate[bvert] (b) at (1.5,-.5);
					\phantomwhite(0,-.5);
					\phantomwhite(3,-.5);
					\draw (b)--(w);
					
					\draw[] (0,3.5) -- (3,3.5);
					\phantomwhite(1.5,3.5);

				\end{scope}
				
				\draw[] (0,5.5) -- (3,5.5);
				\phantomwhite(1.5,5.5);

				\node[] (no) at (1.5,-4.25) {$\vdots$}; 
				\node[] (no) at (1.5,4.25) {$\vdots$}; 
				\newlabels
				
			\end{scope}
			
			\begin{scope}[shift={(2*\shifta+2*\shiftb,0)}]%
				\def\r{2};
				
				\fill[black!5]  (0,-6) rectangle (3,6);
				
				\draw[dashed,\dashcolor,-] (0,-6) rectangle (3,6);

				\node[] (no) at (1.5,0) {$\vdots$};
				\begin{scope}
					\clip (0,-6) rectangle (3,6);
					
					\draw[] (3,-1.5) -- (0,-3.5);
					\phantomwhite(1.5,-2.5); 
					
					\draw[] (3,-3.5) -- (0,-5.5);
					\phantomwhite(1.5,-4.5); 
					\draw[] (3,-5.5) -- (0,-7.5);
					\phantomwhite(1.5,-4.5); 
					\draw[] (3,7.5) -- (0,5.5);
					\phantomwhite(1.5,6.5);  
					\draw[] (3,5.5) -- (0,3.5);
					\phantomwhite(1.5,4.5); 
				\end{scope}

				\node[scale=\labsclG] (no) at (-1,-5.5) {$1$};
				\node[scale=\labsclG] (no) at (-1,-3.5) {$2$};
				\node[scale=\vdotscl] (no) at (-1,-0.5) {$\vdots$}; 
				\node[scale=\labsclG] (no) at (-1,3.5) {$n-1$}; 
				\node[scale=\labsclG] (no) at (-1,5.5) {$n$};

			\end{scope}

			\begin{scope}[shift={(\shifta,0)},yscale=1.090909]
				\def\r{2};
				
				\fill[black!5]  (0,-5.5) rectangle (3,5.5);
				\draw[dashed,\dashcolor,-] (0,-5.5) rectangle (3,5.5) ;
				
				\draw[red,->,line width=\lw] (0,-1) .. controls +(1,0) and +(-1,0) ..
				(3,1);
				\draw[red,->,line width=\lw] (0,1) .. controls +(1,0) and +(-1,0) ..
				(3,-1) ;
				\draw[red,->,line width=\lw] (3,0) -- (0,0);
				\draw[red,->,line width=\lw]
				(3,2)--(0,2);
				\draw[red,->,line width=\lw]
				(0,3)--(3,3)
				;
				\draw[red,->,line width=\lw]
				(3,-2)--(0,-2);
				\draw[red,->,line width=\lw]
				(0,-3)--(3,-3)
				;
				
				\draw[red,->,line width=\lw]
				(3,5)--(0,5);
				\draw[red,->,line width=\lw]
				(0,-5)--(3,-5)
				;
				\node[scale=\labsclD] (no) at (-1,0) {$\overline{i}$}; 
				\node[scale=\labsclD] (no) at (-1,1) {$i+1$}; 
				\node[scale=\labsclD] (no) at (-1,-1) {$i$}; 
				\node[scale=\labsclD] (no) at (-1,-2) {$\overline{i-1}$}; 
				\node[scale=\labsclD] (no) at (-1,-3) {$i-1$}; 	
				
				\node[scale=\vdotscl] (no) at (-1,-4) {$\vdots$}; 
				\node[scale=\labsclD] (no) at (-1,-5) {$1$};
				
				\node[scale=\labsclD] (no) at (-1,2) {$\overline{i+1}$}; 
				\node[scale=\labsclD] (no) at (-1,3) {$i+2$}; 	
				
				\node[] (no) at (1.5,-4) {$\vdots$}; 
				\node[] (no) at (1.5,4) {$\vdots$}; 
				\node[scale=\vdotscl] (no) at (-1,4) {$\vdots$}; 
				\node[scale=\labsclD] (no) at (-1,5) {$\overline n$};

			\end{scope}

			\begin{scope}[shift={(2*\shifta+\shiftb,0)},yscale=1.090909]
				\def\r{2};
				\fill[black!5]  (0,-5.5) rectangle (3,5.5);
				\draw[dashed,\dashcolor,-] (0,-5.5) rectangle (3,5.5) ;
				
				\draw[red,<-,line width=\lw] (0,-1) .. controls +(1,0) and +(-1,0) ..
				(3,1);
				\draw[red,<-,line width=\lw] (0,1) .. controls +(1,0) and +(-1,0) ..
				(3,-1) ;
				\draw[red,<-,line width=\lw] (3,0) -- (0,0);
				\draw[red,<-,line width=\lw]
				(3,2)--(0,2);
				\draw[red,<-,line width=\lw]
				(0,3)--(3,3)
				;
				\draw[red,<-,line width=\lw]
				(3,-2)--(0,-2);
				\draw[red,<-,line width=\lw]
				(0,-3)--(3,-3)
				;
				
				\draw[red,->,line width=\lw]
				(3,5)--(0,5);
				\draw[red,->,line width=\lw]
				(0,-5)--(3,-5)
				;
				\node[scale=\labsclD] (no) at (-1,0) {$i$}; 
				\node[scale=\labsclD] (no) at (-1,1) {$\overline{i}$}; 
				\node[scale=\labsclD] (no) at (-1,-1) {${\overline {i-1}}$}; 
				\node[scale=\labsclD] (no) at (-1,-2) {$i-2$}; 
				\node[scale=\labsclD] (no) at (-1,-3) {$\overline{i-2}$}; 	
				
				\node[scale=\vdotscl] (no) at (-1,-4) {$\vdots$}; 
				\node[scale=\labsclD] (no) at (-1,-5) {$1$};
				
				\node[scale=\labsclD] (no) at (-1,2) {$i+1$}; 
				\node[scale=\labsclD] (no) at (-1,3) {$\overline{i+1}$}; 	
				
				\node[] (no) at (1.5,-4) {$\vdots$}; 
				\node[] (no) at (1.5,4) {$\vdots$}; 
				\node[scale=\vdotscl] (no) at (-1,4) {$\vdots$}; 
				\node[scale=\labsclD] (no) at (-1,5) {$\overline{n}$};

			\end{scope}
			
			\begin{scope}[shift={(3*\shifta+2*\shiftb,0)},yscale=1.090909]
				\def\r{2};
				\fill[black!5]  (0,-5.5) rectangle (3,5.5);
				\draw[dashed,\dashcolor,-] (0,-5.5) rectangle (3,5.5) ;
				\begin{scope}
					\clip (0,-5.5) rectangle (3,5.5) ;
					
					\draw[red,<-,line width=\lw] (0,-4) .. controls +(1,0) and +(-1,0) ..
					(3,-2);
					\draw[red,<-,line width=\lw] (0,-6) .. controls +(1,0) and +(-1,0) ..
					(3,-4);
					
					\draw[red,<-,line width=\lw] (0,1) .. controls +(1,0) and +(-1,0) ..
					(3,3);
					\draw[red,<-,line width=\lw] (0,3) .. controls +(1,0) and +(-1,0) ..
					(3,5);
					\draw[red,<-,line width=\lw] (0,5) .. controls +(1,0) and +(-1,0) ..
					(3,7);
					
					\draw[red,->,line width=\lw] (0,-5) .. controls +(1,0) and +(-1,0) ..
					(3,-3);
					\draw[red,->,line width=\lw] (0,-3) .. controls +(1,0) and +(-1,0) ..
					(3,-1);
					\draw[red,->,line width=\lw] (0,-7) .. controls +(1,0) and +(-1,0) ..
					(3,-5);
					\draw[red,->,line width=\lw] (0,2) .. controls +(1,0) and +(-1,0) ..
					(3,4);
					\draw[red,->,line width=\lw] (0,4) .. controls +(1,0) and +(-1,0) ..
					(3,6);
				\end{scope}

				\node[scale=\labsclD] (no) at (-1,4) {$n$};	
				\node[scale=\labsclD] (no) at (-1,3) {$\overline{n-1}$};	
				\node[scale=\labsclD] (no) at (-1,2) {$n-1$};	
				\node[scale=\labsclD] (no) at (-1,1) {$\overline{n-2}$};
				
				\node[scale=\labsclD] (no) at (-1,5) {$\overline{n}$};
				
				\node[scale=\labsclD] (no) at (-1,-4) {$\overline{1}$};	
				\node[scale=\labsclD] (no) at (-1,-3) {$2$};	
				
				\node[scale=\vdotscl] (no) at (-1,-1) {$\vdots$};
				\node[] (no) at (1.5,0) {$\vdots$};
				
				\node[scale=\labsclD] (no) at (-1,-5) {$1$};

			\end{scope}

		\end{tikzpicture}
		
		\caption{Plabic graphs and triple-crossing diagrams in $\A$ associated to generators.}\label{fig:tcdgenerators}
	\end{figure}
	\begin{figure}
		\def\twd{0.23\textwidth}
		
		\includegraphics[width=\twd]{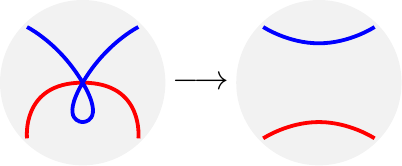}
		
		\caption{(R1)$''$ Thurston's $1-0$ move. }\label{fig:thurston10}
	\end{figure}
	Let $\daffper$ be a double affine permutation and let $w_1 w_2 \cdots w_l$ be an expression for $\daffper$, where $w_i \in S \sqcup \bar S \sqcup \{\Lan\}$. Following Fock and Marshakov \cite{FM}, we associate to the expression $w_1 w_2 \cdots w_l$ a triple-crossing diagram in $\A$ as follows. Each generator $s \in S \sqcup \bar S \sqcup \{\Lan\}$ is assigned a triple-crossing diagram $D(s)$ in $\A$ as shown in Figure \ref{fig:tcdgenerators}. The triple-crossing diagram $D(w_1 w_2 \cdots w_l)$ for the expression $w_1 w_2 \cdots w_l$ is obtained by concatenating the diagrams $D({w_1}),D(w_2),\dots,D({w_l})$ from left to right, so that the right boundary of $D({w_i})$ is glued to the left boundary of $D({w_{i+1}})$ for $i \in \Z/l\Z$. The corresponding plabic graph in $\T$ is called an \textit{affine plabic fence}. As explained in \cite[Appendix D]{FM}, each relation in \cref{eq:dsaffnrelations} can be realized using isotopy and moves on the corresponding triple-crossing diagrams, except for the relations $s_i^2=1$ and $s_{\overline{i}}^2=1$, which are realized using Thurston's $1-0$ move (R1)$''$ (Figure \ref{fig:thurston10}). Note that the left-hand side of (R1)$''$ is the same as (R1)$'$ (but the right-hand side is not), and therefore a triple-crossing diagram $D$ is move-reduced if and only if it is not move-equivalent to another triple-crossing diagram $D'$ to which either (R1)$''$ or (R2)$'$ can be applied.  
	\begin{remark}\label{rem:postnikov10}
		Postnikov's reduction (R1)$'$ leads to the relations $s_i^2=s_i$ and $s_{\overline {i}}^2=s_{\overline{i}}$ of the \emph{$0$-Hecke monoid}.
	\end{remark}
	\begin{remark} \label{rem:barsmap}
		Rotation by $180$ degrees acts on the triple-crossing diagrams by
		\[
		D(s_i) \mapsto D({s_{\overline{n-i+1}}}), \quad D(s_{\overline{i}}) \mapsto D({s_{{n-i+1}}}), \quad D(\Lan) \mapsto D(\Lan),\]
		and induces an antiautomorphism of $\dSaff$ sending $s_i \mapsto s_{\overline{n-i+1}}$, $s_{\overline{i}} \mapsto s_{n-i+1}$, and $\Lan \mapsto \Lan$. We have chosen $\perab(\daffper)=(\pera,\perb)$ so that rotation of $D(w)$ by $180$ degrees translates under $\perab$ into an automorphism of $\Saffn\times\Saffn$ sending $(\pera,\perb) \mapsto (\perb,\pera)$.
	\end{remark}

	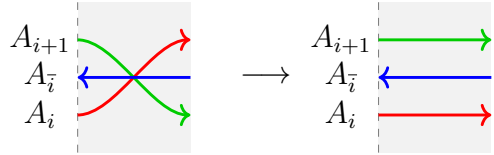
\begin{figure}
		\centering
		\begin{tikzpicture}[scale=0.5]
			
			\begin{scope}[shift={(3,0)}]
				\def\r{2};
				\fill[black!5] (0,-2) rectangle (3,2);
				\draw[dashed,\dashcolor,-] (0,-2)--(0,2) ;
				
				\draw[red,->,line width = \lw] (0,-1) --
				(3,-1);
				\draw[green!80!black,->,line width = \lw] (0,1)--
				(3,1) ;
				\draw[blue,->,line width = \lw] (3,0) -- (0,0);
				\node[] (no) at (-1,0) {$A_{\overline{i}}$}; 
				\node[] (no) at (-1,1) {$A_{i+1}$}; 
				\node[] (no) at (-1,-1) {$A_{i}$};

			\end{scope}

			\begin{scope}[shift={(-5,0)}]
				\def\r{2};
				
				\fill[black!5] (0,-2) rectangle (3,2);
				\draw[dashed,\dashcolor,-] (0,-2)--(0,2) ;
				
				\draw[red,->,line width = \lw] (0,-1) .. controls +(1,0) and +(-1,0) ..
				(3,1);
				\draw[green!80!black,->,line width = \lw] (0,1) .. controls +(1,0) and +(-1,0) ..
				(3,-1) ;
				\draw[blue,->,line width = \lw] (3,0) -- (0,0);
				\node[] (no) at (-1,0) {$A_{\overline{i}}$}; 
				\node[] (no) at (-1,1) {$A_{i+1}$}; 
				\node[] (no) at (-1,-1) {$A_{i}$};

			\end{scope}
			\node[](no) at (0,0){$\longrightarrow$};
			
		\end{tikzpicture}
		\caption{Uncrossing a triple crossing near the left boundary of $\A$ (dashed).}\label{fig:uncrossing}
	\end{figure}
	
	\begin{lemma} \label{lemma:tcdtoper}
		Suppose $D$ is a move-reduced triple-crossing diagram in $\T$ whose Newton polygon $N$ is not a single point. Then, there is a double affine permutation $\daffper=\daffper(D)$ such that $D$ is move-equivalent to $D(\daffper)$, and such that $\pera,\perb$ are both c-reduced, where $\perab(w)=(\pera,\perb)$.
	\end{lemma}
	\begin{proof}
		\def\Dp{D}
		\def\amDp{\am_{\Dp}}
		Since $N$ is not a point, after applying a move-equivalence using \cref{thm:intro:vertical}, we may assume that  
		the number of intersections of strands in $\Dp$ with the sides of $\A$ is minimal. Let $\am:=\amDp$ denote the affine matching of $\Dp$; cf. \cref{sec:affine-matchings}. Then, we have $\am(A)=B$ and $\am(\overline B)=\overline A$. 
		
		By \cref{rmk:strand_word}, for any strand $\sa$ in $\Dp$, the word $w_{\sa}$ is given by $w_{\sa}=xy^kx$ for $x\in\{\l,\r\}$ and $y\in\{\u,\d\}$. If the strand $\sa$ intersects itself in $\A$ then, after repeatedly applying move (T) to the $\u-\d$ side of $\A$ and applying \cref{thm:diskminimal}, we see that $\Dp$ is not move-reduced, a contradiction. From now on, we assume that no strand of $\Dp$ has a self-intersection in $\A$.

		We construct the element $\daffper=\daffper(\Dp)$ by induction on the number of triple crossings in $\Dp$. Suppose $\Dp$ contains no triple crossings. Then, the affine matching is of the form $\am(A_{i})=B_{i+m}$ and $\am(B_{\overline i})=A_{\overline{i-m}}$ for some $m \in \Z$. We assign the double affine permutation $\daffper:=\Lan^m$ to $\Dp$. 
		
		Suppose the number of triple crossings in $\Dp$ is nonzero. Since no strand has a self-intersection, there must be three distinct strands in $\A$ at a triple crossing, so two of them must have their in-endpoints on the same side of $\A$. This implies that there exists $i \in [n]$ such that the strands $\sa$ and $\sb$ emanating respectively from either $A_i$ and $A_{i+1}$ or from $B_{\overline i}$ and $B_{\overline {i+1}}$ cross in $\A$. Arguing as in the proofs of Lemmas \ref{lemma:movep} and \ref{lemma:cross}, we can create a triple crossing between $\sa_1$ and $\sb_1$ near the boundary of $\A$. Let $\Dp'$ be the triple-crossing diagram in $\A$ obtained by uncrossing this triple crossing (Figure \ref{fig:uncrossing}). We let $\daffper := s_i \daffper(\Dp')$ (resp., $\daffper:=\daffper(\Dp') s_{\overline{i}}$) if the two strands emanate from $A_i$ and $A_{i+1}$ (resp., $B_{\overline i}$ and $B_{\overline {i+1}}$). 
		
		Clearly, $D(\daffper)$ is isotopic to $\Dp$. Let $\perab(\daffper)=(\pera,\perb)$. We show that $\pera$ and $\perb$ are c-reduced. Suppose not. By~\cref{thm:HN}, there is a c-reduced pair $(\pera',\perb')$ such that $\pera \rightarrow \pera'$ and $\perb \rightarrow \perb'$, and we must have used either $s_i^2=1$ or $s_{\overline{i}}^2=1$ at least once. This implies that $D(\daffper)$ is not move-reduced, a contradiction.
	\end{proof}

	\subsection{Proof of Theorem~\ref{thm:tcdmove_red}}\label{sec:proof:move_red}

	\eqref{item:D:move_red} $\Longrightarrow$ \eqref{item:D:area}: Suppose $D$ is move-reduced. Since $N$ is not a point, by~\cref{lemma:tcdtoper}, $D$ is move-equivalent to a triple-crossing diagram $D(\daffper)$, where $\daffper$ is a double affine permutation. Therefore, $D$ and $D(\daffper)$ have the same number of triple crossings.  By~\cref{lemma:c_red_ell}, the number of triple crossings in $D(w)$ is $\ell(\pera)+\ell(\perb)=\Area(\Zon(\VC_\pera))+\Area(\Zon(\VC_\perb))+\excess{\bfla}$, where $\perab(w)=(\pera,\perb)$. %
	By~\eqref{eq:zon_vs_N} below, we have $\ell(\pera)+\ell(\perb)=2\Area(N)+\excess{\bfla}$. If $D$ had a contractible connected component $D'$, then $D'$ must have a loop strand. By~\cref{thm:diskminimal}, $D'$, and therefore $D$ is not move-reduced. 

For the converse implication, we will need the following result.
	\begin{lemma}\label{lem:movereduction}
		Let $D$ be a triple-crossing diagram with weakly decorated Newton polygon $\Nwdec$. If $D$ is not move-reduced, then there is a move-reduced triple-crossing diagram $D'$ with weakly decorated Newton polygon $\Nwdec$ containing strictly fewer triple crossings than $D$.
	\end{lemma}
	\begin{proof}
		Recall the reduction move (R1)$''$ shown in Figure~\ref{fig:thurston10}. The move (R1)$''$ preserves the connectivity of the strands, and therefore the does not change $\Nwdec$. If $D$ is not move-reduced, then we can use moves (M1)$'$, (R1)$''$ and (R2)$'$ to get a move-reduced $D'$ with weakly decorated Newton polygon $\Nwdec$. Since $D$ has no contractible components, (M1)$'$ cannot create contractible components, and therefore we must use (R1)$''$ at least once before we can use (R2)$'$. Since we decrease the number of triple crossings when we apply (R1)$''$, $D'$ contains strictly fewer triple crossings than $D$.
	\end{proof}

         \eqref{item:D:area} $\Longrightarrow$\eqref{item:D:move_red}: Suppose that $D$ has $2\Area(N)+\excess{\bfla}$ triple crossings and that $D$ has no contractible connected components. If $D$ is not move-reduced, there is a move-reduced $D'$ with weakly decorated Newton polygon $\Nwdec$ with fewer than $2\Area(N)+\excess{\bfla}$ triple crossings by~\cref{lem:movereduction}, contradicting \eqref{item:D:move_red} $\Longrightarrow$ \eqref{item:D:area}. \qed

	\subsection{Proof of Proposition~\ref{prop:propertiesofmoveredtcd}}\label{sec:proof:propertiestcdmovered}
	
	\begin{figure}
		\begin{tabular}{ccc}
			\includegraphics[width=0.23\textwidth]{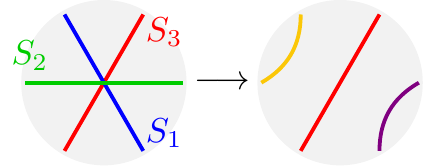}
			& \qquad &
			\includegraphics[width=0.23\textwidth]{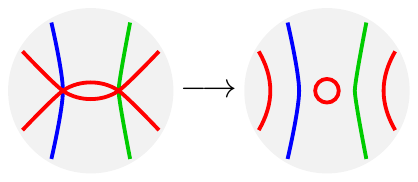}
			\\
			(a) & & (b)
		\end{tabular}
		\caption{\label{fig:uncrossingmove} (a) Uncrossing the strands $\sa_1$ (blue) and $\sa_2$ (green), while the strand $\sa_3$ (red) is unaffected. (b) The uncrossing move applied to the two strands participating in both triple crossings on the left-hand side of (M1)$'$.}
	\end{figure}

	Let $p$ be a triple-crossing at which three strands $\sa_1,\sa_2,\sa_3$ meet. We call the variant of the skein relation shown in \figref{fig:uncrossingmove}(a) \textit{uncrossing $\sa_1$ and $\sa_2$ at $p$}. 
	
	By~\cref{lemma:tcdtoper}, $D$ is move-equivalent to a triple-crossing diagram $D(\daffper)$, where $\daffper \in \dSaff$ for some $n$ and the associated affine permutations $\pera,\perb$ are c-reduced. 

        To show part~\eqref{prop:appB_propertiesofmoveredtcd1}, suppose there is a closed loop $\tilde \sa$ in $\tilde D$. Then, the projection $\sa:=\pi(\tilde \sa)$ of this closed loop is a strand with $[\sa]=(0,0)$. Since move-equivalence preserves homology classes of strands, $\sa$ becomes a zero-homology strand in $D(w)$. Since every strand in $D(w)$ moves monotonously to the left or to the right, there are no zero-homology strands in $D(w)$, a contradiction. If $\tilde D$ contains a strand $\tilde{\sa}$ with a self-intersection, then uncrossing $\sa:=\pi(\tilde \sa)$ at the triple point with the self-intersection yields a triple-crossing diagram with the same weakly decorated Newton polygon but with fewer triple crossings, contradicting~\cref{thm:tcdmove_red}.

We now show part~\eqref{prop:appB_propertiesofmoveredtcd2}. By~\cref{cor:c_red_charact}, parts~\eqref{prop:appB_propertiesofmoveredtcd2} and~\eqref{prop:appB_propertiesofmoveredtcd3} are true for $D(w)$. Suppose part~\eqref{prop:appB_propertiesofmoveredtcd2} is false for $D$. Then, $D$ is move-equivalent to $D'$ for which part~\eqref{prop:appB_propertiesofmoveredtcd2} is false, but upon applying (M1)$'$ to $D'$, it becomes true. Since (M1)$'$ only removes crossings between the two anti-parallel strands, the two anti-parallel strands $T_1$ and $T_2$ that cross on the left-hand side of (M1)$'$ should both be portions of $\sa$. Upon uncrossing $T_1$ and $T_2$ at both the triple crossings (see~\figref{fig:uncrossingmove}(b)), the Newton polygon is unchanged, and the strand $\sa$ splits into a loop and at most two other strands, so $2\Area(N)+\excess{\bfla}$ can decrease by at most one, but the number of triple crossings decreases by two, contradicting~\cref{thm:tcdmove_red}.
        
	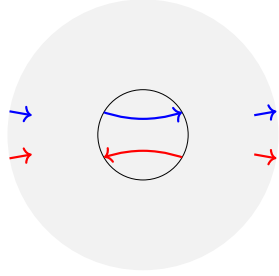
\begin{figure}
		\resizebox{1.5in}{!}{
			\begin{tikzpicture}[scale=0.4]
				\def\r{2};
				\fill[black!5] (0,0) circle (1*3*\r cm);
				\draw[] (0,0) circle (1.0*\r cm);
				\draw [line width = \lw,red,->] (190:3*\r) -- (190:2.5*\r);
				\draw [line width = \lw,blue,->] (170:3*\r) -- (170:2.5*\r);
				
				\draw [line width = \lw,red,<-] (-10:3*\r) -- (-10:2.5*\r);
				\draw [line width = \lw,blue,<-] (10:3*\r) -- (10:2.5*\r);

				\begin{scope}[rotate=45+90]

					\coordinate[] (b1) at (-0.5,0.5);
					\coordinate[] (b2) at (0.5,-0.5);

					\coordinate[] (t1) at (15:\r);
					\coordinate[] (t2) at (120-45:\r);
					\coordinate[] (t3) at (150-45:\r);
					\coordinate[] (t4) at (210-45:\r);
					\coordinate[] (t5) at (240-45:\r);
					\coordinate[] (t6) at (300-45:\r);
					\coordinate[] (t7) at (330-45:\r);
					\coordinate[] (t8) at (30-45:\r);
					
					\draw [line width = \lw,red,->] plot [smooth, tension=1] coordinates {(t5) (b1) (t2)};
					\draw [line width = \lw,blue,->] plot [smooth, tension=1] coordinates {(t1) (b2) (t6)};

				\end{scope}

			\end{tikzpicture}
		}
		\caption{There is no way to complete the red and blue strands so that they do not cross without creating a self-intersection.}\label{fig:monogonparallel}
              \end{figure}
              
              To show part~\eqref{prop:appB_propertiesofmoveredtcd3}, we will need the following lemma.
	\begin{lemma}\label{lem:abcd}
		Suppose $\sa,\sa' \in \SD$ are two distinct parallel strands that do not intersect. Let $R$ be a closed topological disk in $\T$ whose interior contains some portion of $\sa$ and $\sa'$. Let $a$ and $b$ (resp., $c$ and $d$) denote the in- and out-endpoints of $\sa$ (resp., $\sa'$) around the boundary of $R$. Then, the cyclic order of the endpoints around the boundary of $R$ cannot be $abcd$ or $dcba$. 
	\end{lemma}
	\begin{proof}
		Let $N\gg1$ be a large positive integer, and consider a circle of radius $N$ centered at a lift of $R$. Then, we have a strand with a self-intersection in the preimage of $D'$ (Figure \ref{fig:monogonparallel}) in $\R^2$ which contradicts part~\eqref{prop:appB_propertiesofmoveredtcd1} of~\cref{prop:propertiesofmoveredtcd}. 
	\end{proof}

        Similarly to the above, suppose that part~\eqref{prop:appB_propertiesofmoveredtcd3} is false for $D'$, but upon applying (M1)$'$ to $D'$, it becomes true. The two anti-parallel strands that cross on the left-hand side of (M1)$'$ should be portions of $\sa, \sa'$ respectively. Upon uncrossing $\sa$ and $\sa'$ at both triple crossings, the union of $\sa$ and $\sa'$ becomes the union of a loop and a strand $T$ with homology class $[T]=[\sa]+[\sa']$. Therefore, 
		$N$ is unchanged and $2\Area(N)+\excess{\bfla}$ decreases by one, but the number of triple crossings decreases by two, again contradicting~\cref{thm:tcdmove_red}.
		
		Finally, suppose there is a face $F$ of $D$ with portions of $\sa,\sa'$ in its boundary. Recall from \cref{dfn:tcd_T} that the strands in $D$ induce a consistent orientation around the boundary of $F$. We let $R$ be a disk that contains a portion of $F$ together with parts of $\sa$ and $\sa'$, and get a contradiction with \cref{lem:abcd}. 
	\qed

	\subsection{Proof of Proposition~\ref{prop:intro:exists}}\label{sec:proof:exists}
	Suppose $\Ndec=(N,\bfcc)$ is a strongly decorated Newton polygon and $\minv \in \Z/\clicks(\bfcc)$. Recall from \cref{sec:eps_straight_construct} that for $\e=(a,b)\in\Z^2$, we denote $\nop(\e):=a$ and $\kop(\e):=b$, and $\slp(\e)=\kop(\e)/\nop(\e)$. Using an $\operatorname{SL}_2(\Z)$ transformation, we can assume that $\nop(\e) \neq 0$ for all $\e \in E(N)$. We assign to $\Ndec$ the pair $(\VCDDe_+,\VCDDe_-)$ of strongly decorated vector configurations, consisting of edges of $N$ oriented to the right and left, respectively, as follows. We define:
	\begin{enumerate}
		\item $\VC_+:=\{\e\mid \e \in E(N),\ \nop(\e)>0\}$ and $\bfcc_+=(\cc^\e)_{\e \in \VCDDe_+}$; and 
		\item $\VC_-:=\{-\e\mid \e \in E(N),\ \nop(\e)< 0\}$ and $\bfcc_-=(\rev(\cc^\e))_{-\e \in \VCDDe_-}$, where for a cyclic composition $\cc=(\cc_1,\cc_2,\dots,\cc_m)$,  $\rev(\cc) := (\cc_m,\cc_{m-1},\dots,\cc_1)$ is the cyclic composition with the cyclic order reversed. 
	\end{enumerate}
	Similarly to \cref{rem:barsmap}, we have rotated the vectors in $\VC_-$ by $180$ degrees. We have the following basic relation between the area of $N$ and the areas of the zonotopes $\Zon(\VC_+)$, $\Zon(\VC_-)$:
	\begin{equation}\label{eq:zon_vs_N}
		2\Area(N)=\Area(\Zon(\VC_+))+\Area(\Zon(\VC_-)).
	\end{equation}
	To see this, observe that the lower boundary of $\Zon(\VC_+)$ coincides with the lower boundary of $N$ (given by the vectors in $\VC_+$), and the upper boundary of $\Zon(\VC_+)$ is obtained by rotating its lower boundary by $180$ degrees. A similar statement holds for $\Zon(\VC_-)$, from which the result follows; see \cref{fig:zon_vs_N}.
	
	\begin{figure}
		\def\twd{0.55\textwidth}
		\includegraphics[width=\twd]{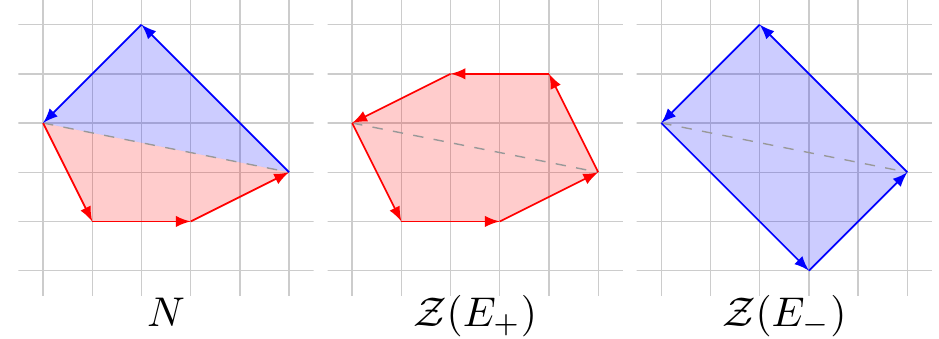}
		\caption{\label{fig:zon_vs_N} Proof of~\eqref{eq:zon_vs_N}: the dashed line subdivides $N$ into two polygons whose areas are $\frac12\Area(\Zon(\VC_+))$ and $\frac12\Area(\Zon(\VC_-))$.}
	\end{figure}

	Let $\pera$ and $\perb$ be a pair of c-reduced  affine permutations with $\VCDDe_{\pera}=\VCDDe_+$ and $\VCDDe_{\perb}=\VCDDe_-$ constructed as in \cref{sec:eps_straight_construct}. Observe that $\sum_{\e\in \VC_+} \kop(\e)=\sum_{\e\in\VC_-}\kop(\e)$, and thus by \cref{rmk:perab_invertible}, there exists $\daffper\in\dSaff$ satisfying $\perab(\daffper)=(\pera,\perb)$. 
	By~\eqref{eq:zon_vs_N}, the triple-crossing diagram $D:=D(\daffper)$ has the correct number of triple crossings, it is move-reduced by~\cref{thm:tcdmove_red}. 
	Let $\Gamma:=\Gamma(D)$ be the associated bipartite graph (cf.~\cref{sec:bipartite-to-tcd}). By construction, $\Ndec(\Gamma)=\Ndec$, and replacing $\pera$ with $\rotsig^r(\pera)$ for some $r$ while keeping $\perb$ fixed, we can achieve $\minv(\Gamma)=\minv$. 
	
	Finally, we show that $\Gamma$ has a perfect matching. Let $ \daffper=s_{i_1} s_{i_2} \cdots s_{i_l} \Lan^k s_{\overline{j_1}}s_{\overline{j_2}} \cdots s_{\overline{j_m}} $ be a reduced expression. Omitting all generators $s_{i_k}$ and $s_{\overline{i_k}}$ such that the corresponding vertical edge in $\Gamma$ is traversed by the same strand in the opposite directions (i.e., yields a self-intersection in $D(\daffper)$), we get a triple-crossing diagram $D'$ with strongly decorated Newton polygon $(N,\bfcc')$ satisfying $(\cc')^e=(1,1,\dots,1)$ for all $\e \in E(N)$. By~\cref{thm:tcdmove_red} and part~\eqref{prop:appB_propertiesofmoveredtcd2} of~\cref{prop:propertiesofmoveredtcd}, $D'$ is move-reduced, so it is minimal in the sense of \cite{GK13}. 
	The corresponding bipartite graph $\Gamma':=\Gamma(D')$ has a perfect matching by \cite[Lemma 3.11]{GK13}, and since $\Gamma'$ is obtained from $\Gamma$ by deleting a subset of edges, so does $\Gamma$.\qed

	\begin{figure}
		\begin{tikzpicture}[xscale=0.3,yscale=0.4] %
			
			\def\r{2};
			\def\labsclG{1.0}
			\fill[black!5]  (0,-1.5) rectangle (27,6.5);
			
			\draw[dashed,\dashcolor,-] (0,-1.5) rectangle (27,6.5);
			\def\labsclD{0.8}
			\node[scale=\labsclD] (no) at (-1,0) {$\overline{1}$}; 
			\node[scale=\labsclD] (no) at (-1,1) {$2$}; 
			\node[scale=\labsclD] (no) at (-1,-1) {$1$};

			\node[scale=\labsclD] (no) at (-1,2) {$\overline{2}$}; 
			\node[scale=\labsclD] (no) at (-1,3) {$3$}; 
			\node[scale=\labsclD] (no) at (-1,4) {$\overline{3}$};
			\node[scale=\labsclD] (no) at (-1,5) {$4$}; 
			\node[scale=\labsclD] (no) at (-1,6) {$\overline{4}$};
			
			\begin{scope}[]		
				\draw[red,line width=\lw] (0,-1) .. controls +(1,0) and +(-1,0) ..
				(3,1);
				\draw[red,line width=\lw] (0,1) .. controls +(1,0) and +(-1,0) ..
				(3,-1) ;
				\draw[blue,line width=\lw] (0,3) -- (3,3);
				\draw[blue,line width=\lw] (0,5) -- (3,5);
				\draw[green!80!black,->,line width=\lw] (3,0) -- (0,0);
				\draw[green!80!black,->,line width=\lw] (3,2) -- (0,2);
				\draw[green!80!black,->,line width=\lw] (3,4) -- (0,4);
				\draw[green!80!black,->,line width=\lw] (3,6) -- (0,6);
				\node[] (no) at (1.5,-2) {$s_1$};
				
			\end{scope}
			
			\begin{scope}[shift={(3,0)}]		
				\draw[blue,line width=\lw] (0,3) .. controls +(1,0) and +(-1,0) ..
				(3,5);
				\draw[blue,line width=\lw] (0,5) .. controls +(1,0) and +(-1,0) ..
				(3,3) ;
				\draw[red,line width=\lw] (0,1) -- (3,1);
				\draw[red,line width=\lw] (0,-1) -- (3,-1);
				
				\draw[green!80!black,line width=\lw] (3,0) -- (0,0);
				\draw[green!80!black,line width=\lw] (3,2) -- (0,2);
				\draw[green!80!black,line width=\lw] (3,4) -- (0,4);
				\draw[green!80!black,line width=\lw] (3,6) -- (0,6);

				\node[] (no) at (1.5,-2) {$s_3$};
				
			\end{scope}
			
			\begin{scope}[shift={(6,0)}]
				\begin{scope}
					\clip(0,-1.5) rectangle (3,6.5);
					\draw[blue,line width=\lw] (0,5) .. controls +(1,0) and +(-1,0) ..
					(3,7);
					\draw[red,line width=\lw] (0,7) .. controls +(1,0) and +(-1,0) ..
					(3,5) ;
					
					\draw[blue,line width=\lw] (0,-3) .. controls +(1,0) and +(-1,0) ..
					(3,-1);
					\draw[red,line width=\lw] (0,-1) .. controls +(1,0) and +(-1,0) ..
					(3,-3) ;
					\draw[red,line width=\lw] (0,1) -- (3,1);
					\draw[blue,line width=\lw] (0,3) -- (3,3);
					
					\draw[green!80!black,line width=\lw] (3,0) -- (0,0);
					\draw[green!80!black,line width=\lw] (3,2) -- (0,2);
					\draw[green!80!black,line width=\lw] (3,4) -- (0,4);
					\draw[green!80!black,line width=\lw] (3,6) -- (0,6);
				\end{scope}

				\node[] (no) at (1.5,-2) {$s_4$};
				
			\end{scope}
			
			\begin{scope}[shift={(9,0)}]		
				\draw[blue,line width=\lw] (0,3) .. controls +(1,0) and +(-1,0) ..
				(3,5);
				\draw[red,line width=\lw] (0,5) .. controls +(1,0) and +(-1,0) ..
				(3,3) ;
				\draw[red,line width=\lw] (0,1) -- (3,1);
				\draw[blue,line width=\lw] (0,-1) -- (3,-1);
				
				\draw[green!80!black,line width=\lw] (3,0) -- (0,0);
				\draw[green!80!black,line width=\lw] (3,2) -- (0,2);
				\draw[green!80!black,line width=\lw] (3,4) -- (0,4);
				\draw[green!80!black,line width=\lw] (3,6) -- (0,6);

				\node[] (no) at (1.5,-2) {$s_3$};
				
			\end{scope}
			
			\begin{scope}[shift={(12,0)}]		
				\draw[blue,line width=\lw] (0,-1) .. controls +(1,0) and +(-1,0) ..
				(3,1);
				\draw[red,line width=\lw] (0,1) .. controls +(1,0) and +(-1,0) ..
				(3,-1) ;
				\draw[red,line width=\lw] (0,3) -- (3,3);
				\draw[blue,line width=\lw] (0,5) -- (3,5);
				\draw[green!80!black,line width=\lw] (3,0) -- (0,0);
				\draw[green!80!black,line width=\lw] (3,2) -- (0,2);
				\draw[green!80!black,line width=\lw] (3,4) -- (0,4);
				\draw[green!80!black,line width=\lw] (3,6) -- (0,6);
				\node[] (no) at (1.5,-2) {$s_1$};
				
			\end{scope}
			
			\begin{scope}[shift={(15,0)}]
				\begin{scope}
					\clip(0,-1.5) rectangle (3,6.5);
					\draw[blue,line width=\lw] (0,5) .. controls +(1,0) and +(-1,0) ..
					(3,7);
					\draw[red,line width=\lw] (0,7) .. controls +(1,0) and +(-1,0) ..
					(3,5) ;
					
					\draw[blue,line width=\lw] (0,-3) .. controls +(1,0) and +(-1,0) ..
					(3,-1);
					\draw[red,line width=\lw] (0,-1) .. controls +(1,0) and +(-1,0) ..
					(3,-3) ;
					\draw[blue,line width=\lw] (0,1) -- (3,1);
					\draw[red,line width=\lw] (0,3) -- (3,3);
					
					\draw[green!80!black,line width=\lw] (3,0) -- (0,0);
					\draw[green!80!black,line width=\lw] (3,2) -- (0,2);
					\draw[green!80!black,line width=\lw] (3,4) -- (0,4);
					\draw[green!80!black,line width=\lw] (3,6) -- (0,6);
				\end{scope}	
				\node[] (no) at (1.5,-2) {$s_4$};
				
			\end{scope}
			
			\begin{scope}[shift={(18,0)}]
				\begin{scope}
					\clip(0,-1.5) rectangle (3,6.5);
					\draw[red,line width=\lw] (0,5) .. controls +(1,0) and +(-1,0) ..
					(3,7);

					\draw[red,line width=\lw] (0,3) .. controls +(1,0) and +(-1,0) ..
					(3,5);
					
					\draw[blue,line width=\lw] (0,1) .. controls +(1,0) and +(-1,0) ..
					(3,3);
					
					\draw[blue,line width=\lw] (0,-1) .. controls +(1,0) and +(-1,0) ..
					(3,1);
					\draw[red,line width=\lw] (0,-3) .. controls +(1,0) and +(-1,0) ..
					(3,-1);
					
					\draw[green!80!black,line width=\lw] (0,-2) .. controls +(1,0) and +(-1,0) ..
					(3,0);
					\draw[green!80!black,line width=\lw] (0,0) .. controls +(1,0) and +(-1,0) ..
					(3,2);
					\draw[green!80!black,line width=\lw] (0,2) .. controls +(1,0) and +(-1,0) ..
					(3,4);
					\draw[green!80!black,line width=\lw] (0,4) .. controls +(1,0) and +(-1,0) ..
					(3,6);
					
					\draw[green!80!black,line width=\lw] (0,6) .. controls +(1,0) and +(-1,0) ..
					(3,8);

				\end{scope}	
				\node[] (no) at (1.5,-2) {$\Lan$};
				
			\end{scope}
			\begin{scope}[shift={(21,0)}]
				\begin{scope}
					\clip(0,-1.5) rectangle (3,6.5);
					\draw[red,line width=\lw] (0,5) .. controls +(1,0) and +(-1,0) ..
					(3,7);

					\draw[blue,line width=\lw] (0,3) .. controls +(1,0) and +(-1,0) ..
					(3,5);
					
					\draw[blue,line width=\lw] (0,1) .. controls +(1,0) and +(-1,0) ..
					(3,3);
					
					\draw[red,line width=\lw] (0,-1) .. controls +(1,0) and +(-1,0) ..
					(3,1);
					\draw[red,line width=\lw] (0,-3) .. controls +(1,0) and +(-1,0) ..
					(3,-1);
					
					\draw[green!80!black,line width=\lw] (0,-2) .. controls +(1,0) and +(-1,0) ..
					(3,0);
					\draw[green!80!black,line width=\lw] (0,0) .. controls +(1,0) and +(-1,0) ..
					(3,2);
					\draw[green!80!black,line width=\lw] (0,2) .. controls +(1,0) and +(-1,0) ..
					(3,4);
					\draw[green!80!black,line width=\lw] (0,4) .. controls +(1,0) and +(-1,0) ..
					(3,6);
					
					\draw[green!80!black,line width=\lw] (0,6) .. controls +(1,0) and +(-1,0) ..
					(3,8);
					
				\end{scope}	
				\node[] (no) at (1.5,-2) {$\Lan$};
				
			\end{scope}
			
			\begin{scope}[shift={(24,0)}]		
				\draw[green!80!black,line width=\lw] (0,4) .. controls +(1,0) and +(-1,0) ..
				(3,6);
				\draw[green!80!black,line width=\lw] (0,6) .. controls +(1,0) and +(-1,0) ..
				(3,4) ;
				\draw[green!80!black,line width=\lw] (0,2) -- (3,2);
				\draw[green!80!black,line width=\lw] (0,0) -- (3,0);
				\draw[blue,line width=\lw,->] (0,3) -- (3,3);
				\draw[blue,line width=\lw,->] (0,5) -- (3,5);
				\draw[red,line width=\lw,->] (0,1) -- (3,1);
				\draw[red,line width=\lw,->] (0,-1) -- (3,-1);
				
				\node[] (no) at (1.5,-2) {$s_{\overline{4}}$};
				
			\end{scope}
			
		\end{tikzpicture}
		\caption{The triple-crossing diagram $D(w)$ with strongly decorated Newton polygon $\Ndec=(N,\bfcc)$ from~\cref {example:graphfromsdcnp}. }\label{fig:eg182}
		\bigskip
		\def\twd{0.35\textwidth}
		\scalebox{0.95}{
			\begin{tabular}{ccc}
				\includegraphics[width=\twd]{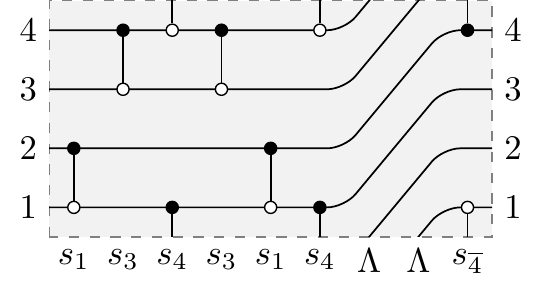} &
				\includegraphics[width=\twd]{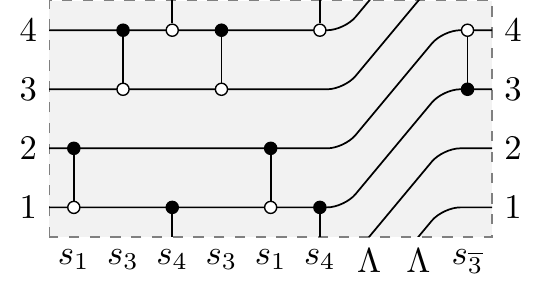}
				&
				\begin{tikzpicture}[baseline=(Z.base)]
					\coordinate(Z) at (0,-1.75);
					\node(A) at (0,0) { \includegraphics[width=0.23\textwidth]{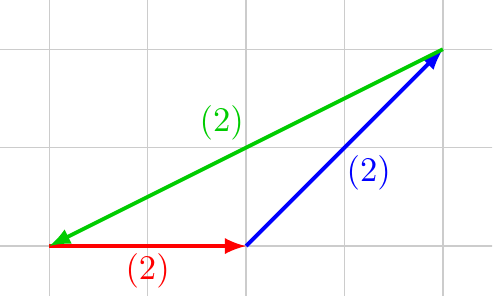}};
				\end{tikzpicture}
				\\
				(a) $\Gamma(w)$. & (b) $\Gamma(w')$. & (c) $\Ndec(\Gamma(w))=\Ndec(\Gamma(w'))$.
			\end{tabular}
		}
		\caption{\label{fig:two_big} Two plabic graphs $\Gamma(w)$, $\Gamma(w')$ from \cref{ex:two_big} having the same strongly decorated Newton polygons but different modular invariants. According to \cref{thm:intro:move_eq}, these graphs are not move-equivalent. 
		}
	\end{figure}
	\begin{example} \label{example:graphfromsdcnp}
		Let $\Ndec=(N,\bfcc)$ be the strongly decorated Newton polygon with edges $\eer = (2,0)$, $\eeb=(2,2)$ and $\eeg=(-4,-2)$, and $\cc^\eer=\cc^\eeb=\cc^\eeg=(2)$ shown in \figref{fig:two_big}(c). The strongly decorated vector configuration $\VCDDe_+$ and its $\eps$-straight arrow diagram $\DP(\VCDDe_+)$ are shown in  \figref{fig:arrowdiagram}(a--b). From $\DP(\VCDDe_+)$, we find the reduced expression $\pera = s_1 s_3 s_4 s_3 s_1 s_4 \Lan^2$. Similarly, we have $\perb= s_{{1}} \Lan^2$, so that $\daffper=s_1 s_3 s_4 s_3 s_1 s_4 \Lan^2 s_{\overline{4}}$. The corresponding triple-crossing diagram $D(w)$ is shown in~\cref{fig:eg182}.
	\end{example}
	\begin{example}\label{ex:two_big}
		Let $\daffper'=s_1 s_3 s_4 s_3 s_1 s_4 \Lan^2 s_{\overline{3}}$ be obtained from $\daffper$ in \cref{example:graphfromsdcnp} by replacing $s_{\overline{4}}$ with $s_{\overline{3}}$. The associated plabic\footnote{Strictly speaking, the graphs shown in \cref{fig:two_big} are not plabic in the language of \cref{sec:plabictcd} since they have edges with both endpoints black. To convert them into plabic graphs, one has to add a degree two white vertex in the middle of each such edge.} graphs $\Gamma(w),\Gamma(w')$ shown in \cref{fig:two_big} have the same strongly decorated Newton polygons but different modular invariants in $\Z/\clicks(\bfcc)\Z$, where $\clicks(\bfcc)=2$.
	\end{example}

	\subsection{Proof of Theorem~\ref{thm:intro:move_eq}}\label{sec:proof:move_eq} 
	
	The $\Longrightarrow$ direction is clear, since both $\Ndec$ and $\minv$ are invariant under move-equivalence; see Sections~\ref{sec:intro:move_eq} and~\ref{sec:intro:minv}.
	
	For the $\Longleftarrow$ direction, using Lemma \ref{lemma:tcdtoper}, we assume that the triple-crossing diagram $D$ (resp., $D'$) associated to $\Gamma$ (resp., $\Gamma'$) is of the form $D(\daffper)$ (resp., $D(\daffper')$) for some double affine permutations $\daffper,\daffper'$. Let $(\pera,\perb)$ (resp., $(\pera',\perb')$) be the pair of affine permutations associated to $\daffper$ (resp., $\daffper'$). Let $\rotsig(\daffper)=\Lan \daffper \Lan^{-1}$ be the rotation operator, and let $\rotsig(\Gamma)$ be the bipartite graph associated to the triple-crossing diagram $D(\sigma(\daffper))$. Note that $\minv(\rotsig(\Gamma)) = \minv(\Gamma)$, but $\minv(\rotsig(\pera))=\minv(\pera)+1$ and $\minv(\rotsig(\perb))=\minv(\perb)-1$. Therefore, replacing $\Gamma$ with $\rotsig^{\minv(\pera')-\minv(\pera)}(\Gamma)$, we can assume that $\minv(\pera)=\minv(\pera')$. 
	
	We will show that there is an $r \in \Z$ such that $\rotsig^r (\pera ) \approx \pera'$ and $\rotsig^r (\perb ) \approx \perb'$. Since $\Ndec(\Gamma)=\Ndec(\Gamma')$ implies that $\VCDD_{\pera}= \VCDD_{\pera'}$ and $\VCDD_{\perb}= \VCDD_{\perb'}$, by~\cref{thm:c-equivalent}, it suffices to show that there is an $r \in \Z$ such that $\minv(\rotsig^r(\pera))=\minv(\pera')$ and $\minv(\rotsig^r(\perb))=\minv(\perb')$, or equivalently, such that $r \equiv 0\pmod{\clicks(\bfccpera)}$ and $r  \equiv \minv(\perb )-\minv(\perb' ) \pmod{\clicks(\bfccperb)}$. 
	Note that $\clicks(\bfcc)= \gcd(\clicks(\bfccpera),\clicks(\bfccperb))$ and $\minv(\Gamma)\equiv\minv(\pera )+\minv(\perb )\pmod{\clicks(\bfcc)}$. Since $\mu(\Gamma)=\mu(\Gamma')$ and $\minv(\pera)=\minv(\pera')$, we have $\minv(\perb)-\minv(\perb') \equiv 0 \pmod{\clicks(\bfcc)}$. The existence of such an $r$ follows from \cref{lemma:rexists}.

	\begin{lemma}\label{lemma:rexists}
		Let $d_1,d_2$ be positive integers, and let $d = \gcd(d_1,d_2)$. Then, there exists $r \in \Z$ such that $r \equiv 0\pmod{d_1}$ and $r\equiv d\pmod{d_2}$.
	\end{lemma}
	\begin{proof}
		Let $x,y \in \Z$ be such that $x d_1 + y d_2 = d$. Take $r:=x d_1$.
	\end{proof}
	
	\appendix
	\section{From bipartite graphs to triple-crossing diagrams}\label{sec:bipartite-to-tcd}

	\begin{figure}
		\def\twd{0.18\textwidth}
		\begin{tabular}{ccccccc}
			\includegraphics[width=\twd]{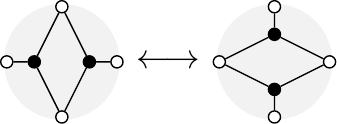}
			& \qquad &
			\includegraphics[width=\twd]{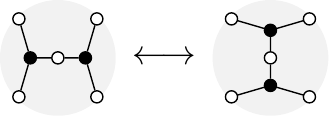}
			& \qquad &
			\includegraphics[width=\twd]{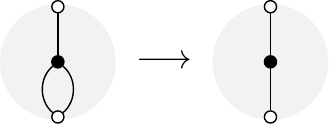}
			& \qquad &
			
			\includegraphics[width=\twd]{figures/move_R2}
			\\
			(M1b). &&(M2b\ithree) Resplit move.  & & (R1\ithree). & & (R2w\ithree).
		\end{tabular}
		\caption{\label{fig:bipartitetcdmoves} Moves and reductions on graphs that correspond to moves and reductions on triple-crossing diagrams. Under $\Gamma \mapsto D(\Gamma)$, (M1b) and (M2b\ithree) become (M1)$'$, and (R1\ithree) and (R2w\ithree) become (R1)$'$.
		}
	\end{figure}
	
	\begin{figure}
		\def\twd{0.18\textwidth}
		\begin{tabular}{ccccccc}
			\includegraphics[width=\twd]{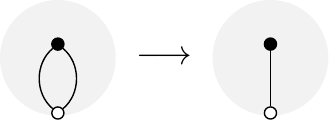}
			& \qquad &
			\includegraphics[width=\twd]{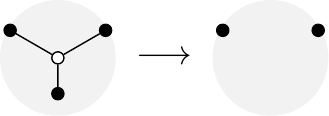}
			& \qquad &
			\includegraphics[width=\twd]{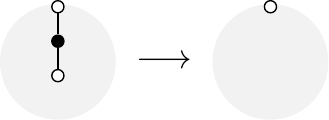}
			& \qquad &
			
			\includegraphics[width=\twd]{figures/move_Rdipole}
			\\
			(R1\itwo) White-based  && (R2b). & & (R2w\itwo). & & (R3).\\
			loop reduction.  &&  & & & & 
		\end{tabular}
		\caption{\label{fig:localbipspecial} Moves and reductions on graphs that are not compatible with triple-crossing diagrams. (R1\itwo), (R2b) and (R3) contain black leaves. Both sides of (R2w\itwo) correspond to the same triple-crossing diagram. 
		}
	\end{figure}
	
	The goal of this section is to give a relation (\cref{lem:parredmoveeq,lem:tcd=bipartite}) between bipartite graphs and triple-crossing diagrams; cf. \cref{rmk:translate2}. We will use these results to deduce \cref{thm:intro:move_red} from \cref{thm:tcdmove_red}. Unless otherwise stated, all graphs in this section are bipartite.

	Let (M1b) (resp., (M1w)) denote the version of (M1) with black (resp., white) trivalent interior vertices, and let (M2b) (resp., (M2w)) denote the version of (M2) 
	contracting/uncontracting black (resp., white) vertices.
	Note that (M1w) can be realized using (M1b) and (M2b) --- we uncontract all the black vertices using (M2b), apply (M1b) and then contract using (M2b).

	We say that $\Gamma$ is (M2w)-\textit{reduced} if contraction using (M2w) cannot be applied to $\Gamma$.
	
	\begin{lemma} \label{lem:m2reduced}
		Two (M2w)-reduced graphs $\Gamma$ and $\Gamma'$ are move-equivalent if and only if they are related by (M1b) and (M2b). An (M2w)-reduced graph is move-reduced if and only if it is not move-equivalent to an (M2w)-reduced graph to which one of \RRs can be applied. 
	\end{lemma}
	\begin{proof}
		By inspection, we see that no move or reduction, except possibly (M2b), involves a degree-two black vertex that can be contracted using (M2w). (Moves (R1) and (R2) might have degree-two black vertices but they cannot be contracted using (M2w).) Applying (M2b) with a degree-two black vertex is the same as applying (M2w).
	\end{proof}
	
	For the rest of this section, we assume that our graphs are (M2w)-reduced. A \textit{white-based loop} in $\Gamma$ is a parallel edge in which the black vertex has degree two (see the left-hand side of (R1\itwo) in Figure \ref{fig:localbipspecial}). In (R1), if the black vertex has degree greater than three, then we can uncontract using (M2b) to make it degree three, and denote this case of (R1) by (R1\ithree). Otherwise, we have a white-based loop and we denote this case by (R1\itwo) (Figure \ref{fig:localbipspecial}).
	
	Let (R2b) (resp., (R2w)) denote black (resp., white) leaf removal. If the white leaf in (R2w) is incident to a black vertex of degree greater than three, we can uncontract the black vertex using (M2b) to get a white leaf incident to a black vertex of degree three and call this (R2w\ithree). If we have a black vertex of degree two, we call it (R2w\itwo) (Figure \ref{fig:localbipspecial}).

	Let $\Gamma$ be a graph in $\T$. Use (R2b), (R2w\itwo) and (R3) to remove all black leaves and white leaves incident to degree-two black vertices. Use (M2b) to uncontract black vertices with degree greater than three until every black vertex has degree either zero, two or three. We call such a graph \textit{partially reduced}. Remove isolated black vertices and omit all degree-two black vertices, converting the two incident edges into a single edge to get a plabic graph. Use the procedure in Figure \ref{fig:tcd=plabic} to obtain a triple-crossing diagram $D(\Gamma)$. Under this procedure, zig-zag paths in $\Gamma$ become strands of $D(\Gamma)$. The choices in applying (M2b) lead to move-equivalent triple-crossing diagrams. Let (M2b\ithree) denote the resplit move (Figure \ref{fig:bipartitetcdmoves}), which consists of two applications of (M2b).
	
	\begin{lemma} \label{lem:moveeqparred}
		Two partially reduced graphs $\Gamma$ and $\Gamma'$ are move-equivalent if and only if they are related by (M1b) and (M2b\ithree). A partially reduced graph $\Gamma$ is move-reduced if and only if it is not move-equivalent to a partially reduced graph $\Gamma'$ to which either (R1\itwo), (R1\ithree) or (R2w\ithree) can be applied.
	\end{lemma}
	\begin{proof}
		Since $\Gamma$ is partially reduced, no black vertices involved in (M2b) have degree two. Any applications of (M2b) involving black vertices of degree greater than three can be decomposed into multiple applications of (M2b\ithree). 
		
		Clearly, (M1b) preserves partial reducedness. Contracting/uncontracting using (M2b) does not change whether any of the moves (R2b), (R2w\itwo) or (R3) can be applied. Therefore, if $\Gamma$ is related to $\Gamma'$ using (M2b), then we can further apply (M2b) to make $\Gamma'$ partially reduced. The reductions (R1\itwo), (R1\ithree) and (R2w\ithree) are the only ones that can be applied to a partially reduced graph.
	\end{proof}
	
	Conversely, we obtain a graph $\Gamma(D)$ from a triple-crossing diagram $D$ as follows. Use the procedure in Figure \ref{fig:tcd=plabic} to obtain a plabic graph, contract any white-white edges incident to distinct white vertices and place a black vertex at the midpoint of each white-white edge to obtain a bipartite graph $\Gamma(D)$. The different choices in applying \figref{fig:tcd=plabic}(b) all lead to the same plabic graph when we contract any white-white edges incident to distinct white vertices. 
	
	Note that $\Gamma(D)$ is partially reduced and has no isolated black vertices. Therefore, $\Gamma \mapsto D(\Gamma)$ and $D \mapsto \Gamma(D)$ are inverse functions between partially reduced graphs without isolated black vertices and triple-crossing diagrams.
	
	\begin{lemma} \label{lem:parredmoveeq}
		The functions $\Gamma \mapsto D(\Gamma)$ and $D \mapsto \Gamma(D)$ between partially reduced graphs without isolated black vertices and triple-crossing diagrams respect move-equivalence.
	\end{lemma}	
	\begin{proof}
		Under the correspondence, (M1)$'$ becomes either (M1b) or (M2b\ithree), so the result follows from~\cref{lem:moveeqparred}.
	\end{proof}
	
	\begin{figure}
		\def\twd{0.18\textwidth}
		\begin{tabular}{ccccccc}
			\includegraphics[width=\twd]{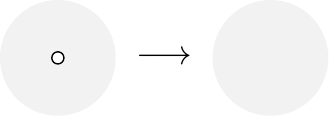}
			& \qquad &
			\includegraphics[width=\twd]{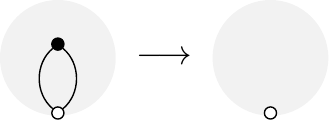}
			& \qquad &
			\includegraphics[width=\twd]{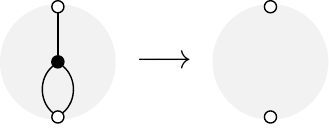}
			& \qquad &
			
			\includegraphics[width=\twd]{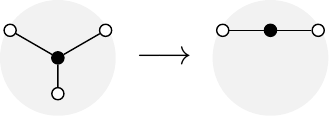}
			\\
			(a) &&(b)  && (c) && (d)
		\end{tabular}
		\caption{ Reductions on triple-crossing diagrams that do not correspond to reductions on graphs under $D \mapsto \Gamma(D)$. Here, (a) and (b) are the two versions of (R2)$'$, and (c) and (d) the two versions of (R1)$''$. In (d), we further contract the degree-two black vertex to get an (M2w)-reduced graph.
		}\label{fig:reductionstcdvsplabic} 
	\end{figure}

	\begin{remark}
		The functions $\Gamma \mapsto D(\Gamma)$ and $D \mapsto \Gamma(D)$ do not commute with reductions; see Figures \ref{fig:localbipspecial} and \ref{fig:reductionstcdvsplabic}. 
	\end{remark}

	\begin{lemma} \label{lem:tcd=bipartite}
		The function $\Gamma \mapsto D(\Gamma)$ is a bijection between move-equivalence classes of move-reduced graphs without isolated vertices and move-equivalence classes of move-reduced triple-crossing diagrams, with inverse $D \mapsto \Gamma(D)$.
	\end{lemma}
	\begin{proof}
		If $\Gamma$ is move-reduced and has no isolated black vertices, then it can be transformed using (M2b) into a partially reduced graph. Therefore, move-equivalence classes of move-reduced graphs without isolated black vertices are in bijection with move-equivalence classes of move-reduced partially reduced graphs without isolated black vertices.  
		
		By~\cref{lem:parredmoveeq}, $\Gamma \mapsto D(\Gamma)$ is a bijection between move-equivalence classes of move-reduced graphs and some subset $T$ of move-equivalence classes of triple-crossing diagrams that we need to identify. 
		
		Let $\Gamma$ be move-reduced and partially reduced without isolated vertices. $\Gamma$ has no white-based loops; otherwise (R1\itwo) can be applied. Since isolated loops in $D:=D(\Gamma)$ correspond to isolated white vertices or white-based loops in $\Gamma$, $D$ contains no isolated loops. Since (M1)$'$ cannot create isolated loops, $D$ is move-reduced if and only if it is not move-equivalent to a $D'$ to which (R1)$'$ can be applied. Under the functions $\Gamma \mapsto D(\Gamma)$ and $D \mapsto \Gamma(D)$, (R1\ithree) and (R2w\ithree) become (R1)$'$ (Figure \ref{fig:bipartitetcdmoves}), so $D$ is move-reduced by~\cref{lem:moveeqparred}. Therefore, $T$ is contained in the set of move-equivalence classes of move-reduced triple-crossing diagrams. 
		
		Let $D$ be a move-reduced triple-crossing diagram, and let $\Gamma:=\Gamma(D)$. Since
		\begin{enumerate} 
			\item (R1\ithree) and (R2w\ithree) become (R1)$'$ (Figure \ref{fig:bipartitetcdmoves});
			
			\item An isolated white vertex becomes the left-hand side of (R2)$'$ (Figure \ref{fig:reductionstcdvsplabic}(a)); and
			\item The left-hand side of (R1\itwo) becomes the left-hand side of (R2)$'$ (Figure \ref{fig:reductionstcdvsplabic}(b)),
		\end{enumerate}  
		$\Gamma$ has no isolated white vertices, and $\Gamma$ is move-reduced by~\cref{lem:moveeqparred}. Therefore, $T$ contains the set of move-equivalence classes of move-reduced triple-crossing diagrams.
	\end{proof}

	The following lemma is used in the proof of Theorem~\ref{thm:intro:move_red}.
	
	\begin{lemma}\label{lem:affinelpabicfence}
		Let $\Gamma$ be a move-reduced graph without isolated vertices. Assume $N(\Gamma)$ is not a single point. The number of contractible faces of $\Gamma$ is equal to the number of degree-three black vertices of $\Gamma$.
	\end{lemma}
	\begin{proof}
		If $\Gamma$ is the affine plabic fence associated to $\Lan^k$, then both numbers are zero. Each $s_i$ and $s_{\overline{i}}$ increases both numbers by one. Therefore, the result holds for affine plabic fences. If $\Gamma$ is move-reduced, then $\Gamma$ is move-equivalent to the bipartite graph associated with an affine plabic fence by~\cref{lemma:tcdtoper} and~\cref{lem:tcd=bipartite}, and move-equivalence does not change the number of contractible faces.
	\end{proof}

	\begin{proof}[Proof of Theorem~\ref{thm:intro:move_red}]
		If $\Gamma$ and $\Gamma'$ are related by (M2w) or (M2b), then each of the conditions \eqref{item:intro:move_red}$-$\eqref{item:intro:area} holds for $\Gamma$ if and only if it holds for $\Gamma'$. Each of \eqref{item:intro:move_red}$-$\eqref{item:intro:area} imply that $\Gamma$ is leafless. Therefore, we can assume that $\Gamma$ is partially reduced. Moreover, since $\Gamma$ has a perfect matching, $\Gamma$ has no isolated vertices.
		
		\eqref{item:intro:move_red} $\Longrightarrow$ \eqref{item:intro:area}: Since $\Gamma$ has a perfect matching, $N$ is not a single point by~\cref{thm:intro:vertical}. The implication follows from~\cref{lem:tcd=bipartite},~\cref{thm:tcdmove_red} and~\cref{lem:affinelpabicfence}.

		\eqref{item:intro:area} $\Longrightarrow$ \eqref{item:intro:move_red}:  Suppose $N$ is a single point. Then, $2\Area(N)+\excess{\bfla}=0$ so $\Gamma$ has no contractible faces. Since \MMs cannot create leaves or contractible faces, none of the reductions \RRs can be applied to any graph move-equivalent to $\Gamma$, so $\Gamma$ is move-reduced. 

		Assume $N$ is not a single point. By~\cref{lem:moveeqparred}, $\Gamma$ is not move-reduced if and only if it is move-equivalent to a partially reduced $\Gamma''$ to which either (R1\itwo), (R1\ithree) or (R2w\ithree) can be applied. Since $\Gamma$ has no leaves and \MMs cannot create leaves, either (R1\itwo) or (R1\ithree) can be applied to $\Gamma'$. Let $D:=D(\Gamma)$ be the associated triple-crossing diagram. Then, either (R1)$''$ or (R2)$'$ can be applied to $D'$. We decrease the number of contractible faces in $\Gamma$ when we apply either reduction (see Figure \ref{fig:reductionstcdvsplabic} (b) and (c)). Transform $D'$ into a move-reduced $D''$ by further using (M1)$'$, (R1)$''$ and (R2)$'$. The graph $\Gamma''=\Gamma(D'')$ has strictly fewer contractible faces than $\Gamma$, no isolated vertices, and satisfies $\Nwdec(\Gamma'')=\Nwdec(\Gamma)$. Since $N$ is not a single point, $\Gamma''$ has $2\Area(N)+\excess{\bfla}$ contractible faces by~\cref{lem:tcd=bipartite},~\cref{thm:tcdmove_red} and~\cref{lem:affinelpabicfence}, a contradiction.
	\end{proof}

	\bibliographystyle{alpha}
	\bibliography{biblio}

\end{document}